\documentclass[reqno]{amsart}
\usepackage{mathtools,amssymb,wasysym,mathscinet,enumitem,stmaryrd}
\usepackage{mathrsfs}  
\usepackage{geometry}
 \numberwithin{equation}{section}

\usepackage[numbers,sort&compress]{natbib}
\usepackage{hyperref}
\usepackage{esint} 
 
\theoremstyle{plain}
\newtheorem{theorem}{Theorem}[section]
\newtheorem{lemma}[theorem]{Lemma}

\newtheorem{proposition}[theorem]{Proposition}

\theoremstyle{definition}

\theoremstyle{remark}
\newtheorem{remark}[theorem]{Remark}

 \makeatletter
 \newcommand{\norm}{\@ifstar{\@normb}{\@normi}}
 \newcommand{\@normb}[2]{\left\Vert{#1}\right\Vert_{#2}}
 \newcommand{\@normi}[2]{\Vert{#1}\Vert_{#2}}
 \newcommand{\norma}{\@ifstar{\@normba}{\@normia}}
 \newcommand{\@normba}[2]{\left\vert{#1}\right\vert_{#2}}
 \newcommand{\@normia}[2]{\vert{#1}\vert_{#2}}
 \newcommand{\normb}{\@ifstar{\@normba}{\@normib}}
 \newcommand{\@normbb}[2]{\left[{#1}\right]_{#2}}
 \newcommand{\@normib}[2]{[{#1}]_{#2}}
 \makeatother

 \global\long\def\Sob#1#2{{W}^{#1}_{#2}} 
  
 \global\long\def\tSob#1{{H}^{#1}} 
 
 \global\long\def\thSob#1{\dot{H}^{#1}}

 \global\long\def\hbes#1#2{\dot{B}^{#1}_{#2}}

 \global\long\def\Leb#1{L_{#1}} 
 \global\long\def\tLeb#1{\tilde{L}_{#1}}

 \global\long\def\bes#1#2{B^{#1}_{#2}}
 
 \global\long\def\BMO{\mathrm{BMO}}
 
 \newcommand{\action}[1]{\left<#1 \right>}
  
   \DeclareMathOperator{\Div}{div}
 \newcommand{\relphantom}[1]{\mathrel{\phantom{#1}}}
 
 \newcommand{\myd}[1]{\,d{#1}}

 \DeclareMathOperator{\supp}{supp}

\makeatletter
\def\@tocline#1#2#3#4#5#6#7{\relax
  \ifnum #1>\c@tocdepth % then omit
  \else
    \par \addpenalty\@secpenalty\addvspace{#2}%
    \begingroup \hyphenpenalty\@M
    \@ifempty{#4}{%
      \@tempdima\csname r@tocindent\number#1\endcsname\relax
    }{%
      \@tempdima#4\relax
    }%
    \parindent\z@ \leftskip#3\relax \advance\leftskip\@tempdima\relax
    \rightskip\@pnumwidth plus4em \parfillskip-\@pnumwidth
    #5\leavevmode\hskip-\@tempdima
      \ifcase #1
       \or\or \hskip 1em \or \hskip 2em \else \hskip 3em \fi%
      #6\nobreak\relax
    \hfill\hbox to\@pnumwidth{\@tocpagenum{#7}}\par% <---- \dotfill -> \hfill
    \nobreak
    \endgroup
  \fi}
\makeatother

\usepackage{tcolorbox}

\allowdisplaybreaks

\keywords{Porous media; Dirichlet-Neumann operator; free boundary problem; global solution; asymptotic behavior}
\subjclass[2020]{35R35, 35Q35, 35D35, 76B03}
\begin{document}

\title{Global well-posedness of the one-phase Muskat problem with surface tension}

\author[H. Dong]{Hongjie Dong}
\address{Division of Applied Mathematics, Brown University, 182 George Street, Providence, RI 02912, USA}
\email{hongjie\_dong@brown.edu }

\author[H. Kwon]{Hyunwoo Kwon}
\address{Division of Applied Mathematics, Brown University, 182 George Street, Providence, RI 02912, USA}
\email{hyunwoo\_kwon@brown.edu }
\thanks{H. Dong and H. Kwon were partially supported by the NSF under agreement DMS-2350129.}

\begin{abstract}
In this paper, we establish the global well-posedness of the one-phase Muskat problem with surface tension for small initial data. This problem describes the motion of the interface separating a wet region from a dry region within a porous medium, a process governed by Darcy's law. While physically essential, the inclusion of surface tension introduces an additional challenge. We prove that if the initial free boundary is sufficiently small in $H^s$, $s>d/2+1$, then the problem admits a unique global strong solution. Moreover, the solution converges to zero in the Lipschitz norm as $t\rightarrow\infty$. To the best of our knowledge, this work constitutes the first global well-posedness result for the one-phase Muskat problem with surface tension.
\end{abstract}

\maketitle

\section{Introduction}

For nearly a century, the dynamics of immiscible fluids in porous media have been extensively studied due to their importance in groundwater hydrology \cite{KH40}, petroleum engineering \cite{M34}, and many other related areas. The flow in a porous medium is modeled by the experimental Darcy's law \cite{D1856}:
\begin{equation}\label{eq:Darcy-law}
 \frac{\mu}{\kappa} u +\nabla_{x,y} p =-\rho\mathfrak{g}e_{d+1},\quad \Div_{x,y}u=0.
 \end{equation}
Here $u$ denotes the velocity of the fluid, $p$ denotes the pressure, $\mu>0$ stands for the dynamic viscosity, $\kappa>0$ the permeability of the porous media, $\rho>0$ the density of the fluid, and $\mathfrak{g}>0$ the gravity constant.

The corresponding free boundary value problem is known as the Muskat problem, first introduced by Muskat \cite{M34} to model the dynamics of the flow generated by two immiscible fluids in petroleum engineering. Note that \eqref{eq:Darcy-law} is mathematically equivalent to the vertical Hele-Shaw problem driven by gravity. While the horizontal Hele-Shaw problem driven by injection or suction has been widely studied mathematically over the past three decades \cite{C93,CJK07,CJK09,CK06,CP93,ES97a,ES97b,K03,KZ24}, the Muskat problem exhibits a different physical nature and mathematical structure since it is driven by gravitational force.

Although Muskat introduced the two-fluid model, it is also of interest to study the dynamics of the interface between a fluid region and an adjacent dry region. To describe the problem, we denote the interface by $\Sigma_t$, which is the graph of the function $f(x,t)$:
\[ \Sigma_f(t) = \{(x,f(x,t)): x\in \mathcal{O}\} \]
where $\mathcal{O} \in \{\mathbb{R}^d,\mathbb{T}^d\}$ and $d\geq 1$ is the horizontal dimension. The associated time-dependent fluid domain is given by 
\[ \Omega_f(t) =\{(x,y): x\in\mathcal{O},\, y<f(x,t)\}.\]

We assume that the interface satisfies the kinematic boundary condition 
\begin{equation}\label{eq:eta-kinematic}
\partial_t f =\sqrt{1+|\nabla f|^2}(u\cdot n)\quad \text{on } \Sigma_f(t).
\end{equation}
Also, by the Young-Laplace equation, the pressure jump is proportional to the mean curvature of the interface:
\begin{equation}\label{eq:p-boundary}
p=\mathfrak{s}H(f)\quad \text{on } \Sigma_f(t),
\end{equation}
where $\mathfrak{s}\geq 0$ is the surface tension coefficient and the mean curvature $H(f)$ is given by
\[ H(f)=-\Div\left(\frac{\nabla f}{\sqrt{1+|\nabla f|^2}}\right). \]
The system consisting of \eqref{eq:Darcy-law}, \eqref{eq:eta-kinematic}, and \eqref{eq:p-boundary} is called the \emph{one-phase Muskat problem with surface tension}.

To describe the motion of the interface between two immiscible fluids, Hou, Lowengrub, and Shelley \cite{HLS94}  introduced a new formulation of the Muskat problem by arc-length parametrization to capture the pattern formation. In this formulation, Ambrose \cite{A04} proved the local well-posedness of the two-phase Muskat problem without surface tension. Later, Guo, Hallstrom, and Spirn \cite{GHS07} studied the stability and instability of a small-scale perturbation of the zero surface. While this formulation enables us to apply a numerical scheme and stability analysis, the possibility of interface self-intersection makes it difficult to establish global well-posedness for general initial data. To overcome this difficulty, Kim \cite{K03} introduced the notion of viscosity solution for the horizontal Hele-Shaw problem and proved the existence and uniqueness of the viscosity solution to the one-phase Hele-Shaw problem. Later, the regularity of solutions was investigated in \cite{CJK07,CJK09,CK06}. Unlike the formulation in \cite{HLS94}, this formulation is based on the pressure of the fluid, making it less straightforward to explicitly track the geometric evolution of the interface.

Alternatively, C\'ordoba and Gancedo \cite{CG07} observed that the two-phase Muskat problem with equal viscosities can be reformulated solely in terms of the interface via a contour formulation. In this framework, they proved local well-posedness of strong solutions for the stable case when the dense fluid is below the lighter fluid. They also proved the ill-posedness for the unstable regime when the fluids are located in the reverse order. Following this, C\'ordoba, C\'ordoba, and Gancedo \cite{CCG11} extended the formulation to include arbitrary viscosity jumps and proved the local well-posedness of the problem including the one-phase Muskat problem as well as the nongraphical interfaces. With this formulation, a series of works successfully established global well-posedness, finite-time singularity formation, and the regularity properties of the interface \cite{AN23,CCFG13,CCCG12,CCFG16,CCGRS16,CCGS13,CP17,GGHP24,GGPS19,CL21,GL22,GGPS23,S23,S24}. For comprehensive surveys on the two-phase Muskat problem, we refer to \cite{A25,GGHP24}. 

For the one-phase Muskat problem without surface tension, i.e., $\mathfrak{s}=0$, Cheng, Granero-Belinch\'on, and Shkoller \cite{CGS16} first proved local well-posedness for the 2D case via a Lagrangian approach. Later, \cite{CCG11} gave a different proof via contour formulation. Using this new formulation, Castro, C\'ordoba, Fefferman, and Gancedo \cite{CCFG13} proved the existence of splash singularities in the stable regime. Splat singularities were proven to be impossible \cite{CP17}, which are possible in water waves \cite{CS14,CCFG12} even in the presence of surface tension. We note that splash singularities are not possible for the two-phase Muskat problem, which was proved by Gancedo and Strain \cite{GS14}. Another distinct phenomenon arises for initial data with a corner. Agrawal, Patel, and Wu \cite{APW23} observed the waiting-time phenomenon for the one-phase Muskat problem when the initial data have acute corners. However, in the two-fluid case, if we assume that the initial data is in $\Leb{2}$ and has a small Lipschitz constant, then Chen, Nguyen, and Xu \cite{CNX22} proved that the global solution exists and is instantaneously smooth. Moreover, this desingularization mechanism was further studied by Garc\'ia-Ju\'arez, G\'omez-Serrano, Haziot, and Pausader \cite{GGHP24}. This smoothing mechanism was first observed in the construction of the self-similar solution for the two-fluid Muskat problem even though the initial data has a corner \cite{GGNP22,N25}.   

Inspired by the theory of water waves, in the graphical interface case, Alazard, Meunier, and Smets \cite{AMS20} and Nguyen and Pausader \cite{NP20} independently introduced another reformulation of the Muskat problem in terms of the interface using the Dirichlet-Neumann operator. More precisely, the one-phase Muskat problem is equivalent to 
\begin{equation}\label{eq:one-phase-reformulation}
\partial_t f =-\frac{\kappa}{\mu}G(f)(\rho \mathfrak{g}f+\mathfrak{s}H(f)),
\end{equation}
where $G(f)g$ is the Dirichlet-Neumann operator  (see Section \ref{subsec:DN-defn}). The advantage of this formulation is that it allows a rough bottom in the problem and one can apply paradifferential calculus to establish low-regularity well-posedness for the Muskat problem.

To the leading order, note that $\rho\mathfrak{g}f+\mathfrak{s}H(f)\sim -\mathfrak{s}\Delta f$ for small $f$. Then the problem becomes
\begin{equation}\label{eq:natural-scaling-M}
\partial_t f =-\frac{\kappa}{\mu}G(f)(-\mathfrak{s}\Delta f).
\end{equation}

 In the absence of the bottom, if $f$ solves \eqref{eq:natural-scaling-M}, then so does $f_\lambda(x,t)=\lambda^{-1}f(\lambda x,\lambda^3 t)$, $\lambda>0$. We note that when $\mathfrak{s}=0$, the natural scaling of \eqref{eq:one-phase-reformulation} is $f_\lambda(x,t)=\lambda^{-1}f(\lambda x,\lambda t)$. The following function spaces are scaling invariant under both scalings:
\[
\dot{H}^{d/2+1},\quad \hbes{d/p+1}{p,q},\quad \dot{W}^{1,\infty},\quad p,q\in [1,\infty].
\]

When we neglect the surface tension, Alazard, Meunier, and Smets \cite{AMS20} and Nguyen and Pausader \cite{NP20} independently proved that the problem \eqref{eq:one-phase-reformulation} is locally well-posed for any subcritical initial data $\tSob{s}$, $s>d/2+1$. Moreover, \cite{NP20} proved the corresponding result for the two-phase case as well, including viscosity contrasts, rough bottom, and the ceiling. 

The Dirichlet-Neumann formulation also enables us to establish global well-posedness. For sufficiently small initial data, Nguyen \cite{N22} proved small data global well-posedness of the one-phase and two-phase Muskat problem without surface tension. While a nongraphical data might lead to a splash singularity \cite{CCFG13}, the first-named author, Gancedo, and Nguyen \cite{DGN23,DGN23b} proved global well-posedness of the one-phase Muskat problem with arbitrarily large periodic Lipschitz initial data. We also note that Schwab, Tu, and Turanova proved that the one-phase Muskat problem has a unique viscosity solution for any bounded uniformly continuous initial data \cite{STT24}.

Although it is crucial to consider the effect of surface tension in porous media, there are relatively few results on the well-posedness of the problem. Escher and Matioc \cite{EM11} obtained local well-posedness in H\"older space by finding an appropriate initial data set as well as the instability of finger-shaped steady-states. With the same viscosity, Matioc \cite{M19} proved local well-posedness of the two-phase Muskat problem with surface tension when the initial data is in $\tSob{s}(\mathbb{R})$, $s\in (2,3)$. Nguyen \cite{N20} extended the result to arbitrary dimensions $\tSob{s}(\mathbb{R}^d)$, $s>d/2+1$, including viscosity contrasts and the one-phase Muskat problem. Later, Flynn and Nguyen \cite{FN21} obtained a more refined estimate to consider the vanishing surface tension limit problem. We also note that Matioc and Matioc \cite{MM21} established the local well-posedness for $\Leb{p}$-based Sobolev initial data, while Chen, Hu, and Nguyen \cite{CHN24} under different assumptions on the initial data, allowing for a rough bottom.

For global well-posedness, very recently, Jacobs, Kim, and M\'esz\'aros \cite{JKM21} proved the global existence of weak solutions of the two-phase Muskat problem by using optimal transport theory. Also, Chen, Hu, and Nguyen \cite{CHN24} proved global well-posedness of strong solutions for the 2D two-phase Muskat problem in a stable regime with the same viscosity. This result was also proved by Lazar \cite{L24} under different assumptions on the initial data. For the one-phase Muskat problem, Bocchi, Castro, and Gancedo \cite{BCG25} obtained global-in-time a priori estimates for solutions near equilibrium when the fluid is confined within a vessel with vertical walls and below a dry region. We also note that for the one-phase horizontal Hele-Shaw problem with surface tension, Agrawal and Patel \cite{AP25} recently constructed a family of self-similar solutions from initial data having a small corner.

Recently, traveling-wave solutions to the one-phase Muskat problem was also studied. Nguyen and Tice \cite{NT24} first constructed traveling waves with small amplitude when it has a finite flat bottom. Later, Nguyen and Stevenson \cite{NS26} constructed traveling waves with large amplitude. With surface tension, Nguyen \cite{N24} constructed large amplitude traveling wave solutions. Brownfield and Nguyen \cite{BN24} proved the existence of slowly traveling waves near large stationary Darcy flow equilibria.

To the best of our knowledge, there are no global well-posedness results for the one-phase Muskat problem with surface tension. The purpose of this paper is to show that the one-phase Muskat problem is globally well-posed for sufficiently small initial data in $\tSob{s}$, $s>d/2+1$ (Theorem \ref{thm:A}). Moreover, the solution converges to zero in the Lipschitz norm as $t\rightarrow\infty$. Our results can also be interpreted as the stability of the trivial interface, generalizing \cite{GHS07} to higher dimensions for the one-phase case.

\subsubsection*{Idea of the proof}
Let us outline the proofs of the main theorems. We first reformulate the one-phase Muskat problem with surface tension via the Dirichlet-Neumann operator in terms of the interface. One may attempt to use the layer potential approach, as in \cite{DGN23,DGN23b}, to construct a unique global viscosity solution for large initial data. However, when surface tension effects are taken into account, the problem becomes, at leading order, a third-order parabolic equation. Hence unlike the problem without surface tension, we cannot expect the comparison principle to hold for the solution to the one-phase Muskat problem with surface tension, which plays a crucial role in \cite{DGN23,DGN23b} to construct a global solution for large initial data. 

Another possible attempt is to use a paralinearization framework as in \cite{ABZ14,NP20,N22,FN21}. While it is highly successful in obtaining large data local well-posedness, a choice of Alinhac's good unknown will lead to a temporal growth in the a priori estimate, which seems to preclude global well-posedness, as seen in \eqref{eq:energy-estimate-surface-tension}.  

If we assume smallness on the initial data, then one can expand the Dirichlet-Neumann operator to isolate the linear term and the nonlinear remainder (Theorem \ref{thm:first-order-expansion})
\begin{equation}\label{eq:DN-decomposition}
G(f)g=|\nabla|g + R(f;g).
\end{equation}
Such a structure was observed in several papers including \cite{AD15,N23}. The advantage of this expansion is that it reduces the quasilinear problem to a semilinear type problem. For instance, if we neglect the surface tension, then Nguyen \cite{N22} proved that the one-phase Muskat problem admits a global solution in a critical space $\hbes{1}{\infty,1}$. However, with a surface tension term, the mean curvature operator $H(f)$ induces another nonlinearity in the problem. Moreover, the operator $H(f)$ makes it difficult to work on the problem in the critical homogeneous space. Hence, unlike \cite{N22}, we need to expand the Dirichlet-Neumann operator in an inhomogeneous space, which introduces another challenge in obtaining energy estimates for the problem.

The aforementioned difficulty can be overcome by showing that the $\Leb{2}$-norm is the Lyapunov functional for the problem (Theorem \ref{thm:Lyapunov}). Without surface tension, it is an easy consequence of the divergence theorem. However, if we have a surface tension, then it is unclear whether
\[ \int_{\mathbb{R}^d} fG(f)H(f) \myd{x} \geq 0\]
holds. Surprisingly, the above quantity can be rewritten in terms of hydraulic pressures (see \eqref{eq:identity-Hessian}). This hidden structure was first observed in the periodic domain case by Alazard and Bresch \cite{AB24}. To extend this result to an unbounded domain, we use a suitable approximation argument by a smooth function $f$ having compact support and regularity estimates for harmonic functions. These enable us to control the boundary behavior of harmonic functions when we apply the divergence theorem in a truncated domain, and to take the limit to obtain the desired identity.

With these preparations, we first construct a family of solutions $f_R$ of the $R$-truncated problems \eqref{eq:ODE-hilbert}. Owing to the Lyapunov property (Theorem \ref{thm:Lyapunov}), each $f_R$ exists globally in time. Under a suitable smallness assumption on the initial data, a standard bootstrap argument yields that the family ${f_R}$ is uniformly bounded (Proposition \ref{prop:a-priori-estimate}) in the Chemin–Lerner space $\tilde{L}_\infty(0,\infty;\tSob{s})$, for $s > d/2 + 3$ (see Section \ref{sec:LP} for definitions), via interpolation. Moreover, it is uniformly bounded in $\Leb{2}(0,\infty;\thSob{s+3/2})$. Then the contraction estimate (Proposition \ref{prop:contraction-solution}) will give the existence and uniqueness of solutions. However, this is not enough to show global well-posedness for low-regularity initial data. Unlike the remainder in paralinearization of the Dirichlet-Neumann operator \cite{NP20}, the remainder in \eqref{eq:DN-decomposition} does not have enough smoothing effect, and so we lose two derivatives when we estimate $R(f;H(f))$. 

To overcome this difficulty, we use the local well-posedness result of Nguyen \cite{N22} to achieve the desired regularity after a small time by using parabolic smoothing for low regularity initial data $\tSob{s}$, $s>d/2+1$. Hence, by assuming the smallness of the initial data in $\tSob{s}$, $s>d/2+1$, we prove the global well-posedness of the problem. Moreover, by the energy estimate in $\thSob{s}$ and interpolation inequality, we can show the asymptotic behavior of the solution of the problem in the Lipschitz norm.

\subsubsection*{Organization of the paper}
The remainder of this paper is organized as follows. Notation and main results will be presented in  Section \ref{sec:main-theorem}. The necessary preliminaries are then provided in Section \ref{sec:prelim}, including the Dirichlet-Neumann operator, Littlewood-Paley projections, Besov spaces, and Chemin-Lerner spaces. In Section \ref{sec:Lyapunov}, we show that $\Leb{2}$-norm is a Lyapunov functional for the one-phase Muskat problem. In Section \ref{sec:first-order-expansion}, we obtain the first-order expansion of the Dirichlet-Neumann operator. Proofs of technical lemmas required for this expansion are deferred to Appendix \ref{app:Qa-Qb-estimates}. The main theorems will be proved in Section \ref{sec:GWP}. Finally, in Appendix \ref{app:interpolation}, we provide results on interpolation spaces.

\section{Notation and Main results}\label{sec:main-theorem}

\subsection{Notation}

Let $\mathbb{R}^d$ and $\mathbb{T}^d=\mathbb{R}^d/\mathbb{Z}^d$ denote the standard $d$-dimensional Euclidean spaces and the $d$-dimensional torus, respectively. For $x_0 \in \mathbb{R}^d$, $d\geq 1$, $B_r(x_0)$ denotes the open ball of radius $r$ centered at $x_0$. When $x_0=0$, we shall write $B_r=B_r(0)$. We write $\nabla$ the spatial gradient and $\nabla_{x,y}=(\nabla,\partial_y)$. 

For $I\subset \mathbb{R}$, a Banach space $X$, and $p\in [1,\infty]$, we write $\Leb{p}(I;X)$ which consists of all strongly measurable function $f:I\rightarrow X$ with 
\[ \norm{u}{\Leb{p}(I;X)}=\left(\int_I \norm{u(t)}{X}^p dt\right)^{1/p}<\infty \]
if $p<\infty$ and 
\[ \norm{u}{\Leb{\infty}(I;X)}=\sup_{t\in I} \norm{u(t)}{X}<\infty \]
if $p=\infty$. We write $\Leb{p}^tX$ for abbreviation. We frequently write $\mathcal{F}:[0,\infty)\rightarrow [0,\infty)$ a nondecreasing function throughout this paper. For two Banach spaces $X$ and $Y$ with $X\subset Y$, we write $X\hookrightarrow Y$ if there exists a constant $C>0$ such that $\norm{u}{Y}\leq C\norm{u}{X}$ for all $u\in X$. For two Banach spaces $Y$ and $Z$, define 
\[ \norm{u}{Y\cap Z}=\norm{u}{Y}+\norm{u}{Z}.\]
Finally, we write $A\apprle B$ if there exists a constant $C>0$ independent of $A$ and $B$ such that $A\leq CB$. We write $A\apprle_{p_1,\dots,p_k} B$ if the constant $C$ depends on the parameters $p_1$, ..., $p_k$. For a $d$-dimensional hypersurface $\Sigma \subset \mathbb{R}^{d+1}$, we denote by $\mathcal{H}^d$ the $d$-dimensional Hausdorff measure on $\Sigma$.

\subsection{Main results}
 
Our first result is the global well-posedness of the one-phase Muskat problems with surface tension if the initial data is sufficiently small. 

\begin{theorem}\label{thm:A}
Let $s>d/2+1$. Then there exists $\varepsilon_0=\varepsilon_0(s,d)>0$ such that if 
\[ \norm{f_0}{\tSob{s}}<\varepsilon_0,\]
then the problem admits a unique global solution to \eqref{eq:one-phase-reformulation} satisfying
\[ f\in C([0,\infty);\tSob{s})\cap \Leb{2}(0,\infty;\thSob{s+3/2}),\quad \partial_t f\in \Leb{2}(0,T;\tSob{s-3/2}),\quad f|_{t=0}=f_0 \]
for any $T>0$.
\end{theorem}
\begin{remark}\leavevmode
\begin{enumerate}
\item Our result can be regarded as an extension of Guo, Hallstrom, and Spirn \cite{GHS07} for the one-phase Muskat problem in the stable regime.
\item In the absence of surface tension, Dong, Gancedo, and Nguyen \cite{DGN23,DGN23b} proved the global well-posedness for the one-phase Muskat problem with any periodic Lipschitz data. It is widely open whether we can prove global well-posedness of the one-phase Muskat problem with surface tension with the large initial data. See Remark \ref{rem:exponential-decay}.
\end{enumerate}
\end{remark}

We also have the following asymptotic behavior of solutions to the one-phase Muskat problems.
\begin{theorem}\label{thm:B}
Let $s>d/2+1$ and suppose that $f_0 \in \tSob{s}$. Then there exists $\varepsilon_0=\varepsilon_0(s,d)>0$ such that if 
\[ \norm{f_0}{\tSob{s}}<\varepsilon_0,\]
then the global solution $f$ constructed in Theorem \ref{thm:A} satisfies 
\[ \lim_{t\rightarrow \infty}\norm{f(t)}{\thSob{s}}=0.\]
Moreover, we have 
\[ \lim_{t\rightarrow \infty}\norm{f(t)}{\Sob{1}{\infty}}=0.\]
\end{theorem}
\begin{remark}\leavevmode
Without surface tension, Nguyen \cite{N23} proved decay of the $C^\alpha$-norm of the viscosity solution to the one-phase Muskat problem on periodic space without a smallness assumption on the initial data. In our result, we show that the solution converges to zero in $\tSob{s}$ exponentially in time if the domain is periodic. See Remark \ref{rem:exponential-decay}.
\end{remark}

\section{Preliminaries}\label{sec:prelim}
In this section, we first recall the definition of the Dirichlet-Neumann operator and list several mapping properties of the operator. Next, we reformulate the one-phase Muskat problem in terms of the Dirichlet-Neumann operator. In Section \ref{sec:LP}, we introduce Littlewood-Paley projections and define Besov spaces. We also recall the paradifferential theorem in Besov spaces, which will be used in this paper. Finally, we introduce Chemin-Lerner spaces, where we will expand the Dirichlet-Neumann operator in Section \ref{sec:first-order-expansion}.

\subsection{Dirichlet-Neumann operator}\label{subsec:DN-defn}
For $f:\mathbb{R}^d\rightarrow \mathbb{R}$, we define 
\begin{equation*}
\begin{gathered}
\Omega_f =\{(x,y): x\in \mathbb{R}^d, y<f(x)\},\quad \Sigma_f = \{(x,f(x)): x\in \mathbb{R}^d\},\\
N(x)=(-\nabla f(x),1),\quad n(x)=\frac{N(x)}{|N(x)|}.
\end{gathered}
\end{equation*}
The Dirichlet-Neumann operator $G(f)$ is defined by 
\[ (G(f)g)(x)=\partial_N \phi(x,f(x)):=\lim_{h\rightarrow 0-}\frac{1}{h}[\phi((x,f(x))+hN(x))-\phi(x,f(x))],\]
where $\phi(x,y)$ solves the elliptic problem
\begin{equation}\label{eq:DN-elliptic}
\left\{
\begin{alignedat}{2}
\Delta_{x,y}\phi&=0&&\quad \text{in } \Omega_f,\\
\phi&=g&&\quad \text{on } \Sigma_f,\\
\nabla_{x,y}\phi &\in \Leb{2}(\Omega_f).
\end{alignedat}
\right.
\end{equation}
Similarly, one can also define the Dirichlet-Neumann operator $G(f)$ when $f:\mathbb{T}^d\rightarrow\mathbb{R}$.

If $f$ and $g$ are time-dependent, we write
\[ (G(f)g)(x,t)=(G(f(t))g(t))(x).\]

The Dirichlet-Neumann operator is well defined when $f\in \Sob{1}{\infty}$ and $g\in \thSob{1/2}$, see \cite[Proposition 3.6]{NP20} for the proof.
\begin{proposition}\label{prop:existence-DN}
If $f\in \Sob{1}{\infty}$ and $g\in \thSob{1/2}$, then $G(f)g$ is well defined in $\thSob{-1/2}$. Moreover, there exists a constant $C>0$ such that 
\[
\norm{G(f)g}{\thSob{-1/2}}\leq C(1+\norm{f}{W^{1,\infty}})^2\norm{g}{\thSob{1/2}}.
\]
\end{proposition}

Many researchers studied the mapping property of the Dirichlet-Neumann operator in the context of water waves (see the history in \cite{IP18}). Using paradifferential calculus, Alazard, Burq, and Zuily \cite{ABZ14} obtained regularity estimates of the Dirichlet-Neumann operator and its contraction property as well. Later, Nguyen and Pausader \cite{NP20} refined the paralinearization formula for the Dirichlet-Neumann operator. See \cite[Theorem 3.12]{ABZ14} or \cite[Theorem 3.14]{NP20} for the proof.

\begin{proposition}\label{prop:continuity-DN}
Let $s_0>d/2+1$ and $\sigma \geq 1/2$. Then we have 
\[
\norm{G(f)g}{\tSob{\sigma-1}}\leq \mathcal{F}(\norm{f}{\tSob{s_0}})(\norm{g}{\tSob{\sigma}}+\norm{f}{\tSob{\sigma}}\norm{g}{\tSob{s_0}})
\]
for all $f,g\in \tSob{\max\{s_0,\sigma\}}$. Moreover, if $1/2\leq \sigma\leq s_0$, then 
\[
\norm{G(f_1)g-G(f_2)g}{\tSob{\sigma-1}}\leq \mathcal{F}(\norm{(f_1,f_2)}{\tSob{s_0}})\norm{f_1-f_2}{\tSob{\sigma}}\norm{g}{\tSob{s_0}}
\]
for all $f_j,g\in \tSob{s_0}$.
\end{proposition}

The following parabolicity estimate was shown in \cite[Proposition 4.3]{NP20}.

\begin{proposition}\label{prop:parabolicity-DN}
If $f\in \tSob{s}(\mathbb{R}^d)$ with $s>d/2+1$, $d\geq 1$, then there exists $c_0>0$ such that 
\[ G(f)f(x)<1-c_0\quad \text{for all } x\in\mathbb{R}^d.\]
\end{proposition}

Now we reformulate the one-phase Muskat problem with surface tension. See Nguyen \cite[Appendix A]{N20} for the proof.
\begin{proposition}\label{prop:one-phase-Muskat-DN}
If $(u,p,f)$ solve the one-phase Muskat problem, then $f:[0,\infty)\times\mathbb{R}^d\rightarrow\mathbb{R}$ obeys the equation
\begin{equation}\label{eq:DtoN-formulation-surface-tension}
\partial_t f =-\frac{\kappa}{\mu}G(f)(\mathfrak{s}H(f)+\rho \mathfrak{g}f).
\end{equation}
Conversely, if $f$ is a solution of \eqref{eq:DtoN-formulation-surface-tension}, then the one-phase Muskat problem has a solution which admits $f$ as the free surface.
\end{proposition}

\subsection{Littlewood-Paley projections}\label{sec:LP}
For $m :\mathbb{R}^d\rightarrow\mathbb{C}$, we define the Fourier multiplier operator $m(|\nabla|)$ via the Fourier transform: for any Schwarz function $f$, we define
\[ m(|\nabla|)f(x)=\frac{1}{(2\pi)^{d/2}} \int_{\mathbb{R}^d} e^{ix\cdot \xi} m(\xi)\hat{f}(\xi)d\xi.\]

Let $\phi:\mathbb{R}^d\rightarrow [0,1]$ be a smooth bump function whose support is in $B_{8/5}$ satisfying $\phi=1$ on $B_{5/4}$. Define $\psi(\xi)=\phi(\xi)-\phi(2\xi)$. Then we have 
\[ \phi(\xi)+\sum_{j\geq 0} \psi(2^{-j}\xi)=1\quad\text{for all } \xi \in \mathbb{R}^d,\]
and
\[ \sum_{j\in\mathbb{Z}} \psi(2^{-j}\xi)=1\quad \text{for all } \xi \in \mathbb{R}^d\setminus\{0\}.\]

We define Littlewood-Paley frequency projection operators as the Fourier multiplier operators:
\[ P_{\leq k} f = \phi\left(\frac{|\nabla|}{2^k}\right)f\quad \text{and}\quad P_k f =\psi\left(\frac{|\nabla|}{2^k}\right)f.\]
Note that 
\begin{equation}\label{eq:almost-orthogonality-LP}
P_k P_l f=0\quad \text{if } |k-l|\geq 2.
\end{equation}

We recall Bernstein inequalities (see e.g. \cite{BCD11}).
\begin{proposition}\label{prop:properties-LP}
For $1\leq p\leq q\leq \infty$, $s\in\mathbb{R}$, and $k\in\mathbb{Z}$, we have 
\begin{align*}
\norm{P_k f}{\Leb{q}}&\apprle_d 2^{dk(1/p-1/q)} \norm{P_k f}{\Leb{p}},\\
\norm{P_{\leq k} f}{\Leb{q}}&\apprle_{d} 2^{dk(1/p-1/q)} \norm{P_{\leq k} f}{\Leb{p}},\\
\norm{|\nabla|^s P_k f}{\Leb{p}}&\apprle_{s,d} 2^{sk}\norm{P_k f}{\Leb{p}}.
\end{align*}
Also, there exists a constant $c>0$ such that 
\[ \norm{e^{-t|\nabla|}P_k f}{\Leb{p}}\apprle_d e^{-ct2^k}\norm{P_k f}{\Leb{p}}\]
for any $p\in [1,\infty]$, $k\in\mathbb{Z}$, and $t>0$.
\end{proposition}

Let us introduce inhomogeneous Besov spaces. For $s\in\mathbb{R}$, $p\in (1,\infty)$, and $q\in [1,\infty]$, let 
\[ \norm{f}{\bes{s}{p,q}}=\left(\norm{P_{\leq 0} f}{\Leb{p}}^q+\sum_{j\geq 0} 2^{sjq} \norm{P_j f}{\Leb{p}}^q\right)^{1/q}\]
if $q<\infty$ and 
\[ \norm{f}{\bes{s}{p,\infty}}=\norm{P_{\leq 0} f}{\Leb{p}}+\sup_{j\geq 0} 2^{sj} \norm{P_j f}{\Leb{p}} \]
if $q=\infty$. The inhomogeneous Besov space $\bes{s}{p,r}$ is the space of functions satisfying $\norm{f}{\bes{s}{p,r}}<\infty$. Similarly, one can define homogeneous Besov spaces: we define
\[ \norm{f}{\hbes{s}{p,q}}=\left(\sum_{j} 2^{sjq} \norm{P_j f}{\Leb{p}}^q\right)^{1/q}\]
if $q \in [1,\infty)$ and 
\[ \norm{f}{\hbes{s}{p,\infty}}=\sup_{j\in\mathbb{Z}} 2^{sj} \norm{P_j f}{\Leb{p}} \]
if $q=\infty$. 

Let us recall several properties of Besov spaces. See e.g. \cite{BCD11} for the proof.
\begin{proposition}\label{prop:Besov}
Let $s\in\mathbb{R}$, $1\leq p,q\leq \infty$.
\begin{enumerate}[label=\textnormal{(\roman*)}]
\item If $1\leq q_1\leq q_2\leq \infty$, then $\bes{s}{p,q_1}\hookrightarrow \bes{s}{p,q_2}$;
\item If $s_1>s_2$, then $\bes{s_1}{p,q}\hookrightarrow \bes{s_2}{p,q}$;
\item $\hbes{d/p}{p,1}\hookrightarrow\Leb{\infty}$;
\item $\bes{s}{2,2}=\tSob{s}$ and $\hbes{s}{2,2}=\thSob{s}$ for $s>0$.
\end{enumerate}
\end{proposition}

We recall the following product rules in Besov spaces and Moser-type estimates for controlling nonlinear terms (see, e.g., \cite{BCD11}).
\begin{theorem}\label{thm:Besov-product}
For $(p,q)\in [1,\infty]^2$ and $s>0$, we have 
\[
\norm{u_1u_2}{\hbes{s}{p,q}}\apprle \norm{u_1}{\Leb{\infty}}\norm{u_2}{\hbes{s}{p,q}}+\norm{u_2}{\Leb{\infty}}\norm{u_1}{\hbes{s}{p,q}}.\]
A similar estimate holds for inhomogeneous Besov spaces.
\end{theorem}

\begin{theorem}\label{thm:Moser-estimates}
Let $F\in C^\infty$, $s>0$, and $(p,q)\in [1,\infty]^2$. 
\begin{enumerate}[label=\textnormal{(\roman*)}]
\item Suppose that $F(0)=0$. If $u\in \hbes{s}{p,q}\cap \Leb{\infty}$, then $F(u)\in \hbes{s}{p,q}\cap\Leb{\infty}$ and 
\[ \norm{F(u)}{\hbes{s}{p,q}}\leq \mathcal{F}(\norm{u}{\Leb{\infty}})\norm{u}{\hbes{s}{p,q}},\]
where $\mathcal{F}$ depends on $s$ and $F'$.
\item Suppose that $F'(0)=0$. If $u,v\in \hbes{s}{p,q}\cap \Leb{\infty}$, then $F(u)-F(v)$ belongs to the same space and 
\begin{align*}
&\norm{F(u)-F(v)}{\hbes{s}{p,q}}\\
\leq &\mathcal{F}(\norm{u}{\Leb{\infty}},\norm{v}{\Leb{\infty}})\left(\norm{u-v}{\hbes{s}{p,q}}\sup_{\tau \in[0,1]} \norm{u+\tau(v-u)}{\Leb{\infty}}+\norm{u-v}{\Leb{\infty}} \sup_{\tau \in [0,1]} \norm{u+\tau(v-u)}{\hbes{s}{p,q}} \right),
\end{align*}
where $\mathcal{F}$ depends on $F''$, $s$, $p$, and $q$.
\end{enumerate}
\end{theorem}

\subsection{Chemin-Lerner type spaces}

In this subsection, we introduce Chemin-Lerner type spaces. Such spaces were first introduced by Chemin and Lerner \cite{CL95} to obtain regularity results for the Navier-Stokes equations with initial data $\tSob{d/2-1}(\mathbb{R}^d)$.

For our purpose, we introduce additional variables to define Chemin-Lerner type spaces. For $I,J\subset \mathbb{R}$, $s\in\mathbb{R}$, and $1\leq p,q,\rho_1,\rho_2 \leq \infty$, we define 
\[ \norm{u}{\tLeb{\rho_1}^z\tLeb{\rho_2}^t\bes{s}{p,q}(I\times J\times\mathbb{R}^d)} = \norm{P_{\leq 0} u}{\Leb{\rho_2}^z\Leb{\rho_1}^t(I\times J;\Leb{p})}+\left(\sum_{j=0}^\infty 2^{jsq}\norm{P_j u}{\Leb{\rho_1}^z\Leb{\rho_2}^t(I\times J;\Leb{p}^x)}^q\right)^{1/q} \]
if $r<\infty$. One can define the space when $q=\infty$ via a standard modification. Similarly, one can define its homogeneous version by 
\[ \norm{u}{\tLeb{\rho_1}^z\tLeb{\rho_2}^t\hbes{s}{p,q}(I\times J\times\mathbb{R}^d)} =\left(\sum_{j=-\infty}^\infty 2^{jsr}\norm{P_j u}{\Leb{\rho_1}^z\Leb{\rho_2}^t(I\times J;\Leb{p}^x)}^q\right)^{1/q} \]
with the obvious modification when $r=\infty$. By Minkowski's inequality, we have
\begin{align*}
\norm{u}{\tLeb{\rho_1}^z\tLeb{\rho_2}^t(I\times J;\bes{s}{p,q})}&\leq  \norm{u}{\Leb{\rho_1}^z\tLeb{\rho_2}^t(I;\bes{s}{p,q})}\quad \text{if } q\geq \rho_1,\\
\norm{u}{\tLeb{\rho_1}^z\tLeb{\rho_2}^t(I\times J;\bes{s}{p,q})}&\leq  \norm{u}{\Leb{\rho_1}^z\Leb{\rho_2}^t(I;\bes{s}{p,q})}\quad \text{if } q\geq \max\{\rho_1,\rho_2\},\\
\norm{u}{\tLeb{\rho_1}^z\tLeb{\rho_2}^t(I\times J;\bes{s}{p,q})}&\geq  \norm{u}{\Leb{\rho_1}^z\tLeb{\rho_2}^t(I;\bes{s}{p,q})}\quad \text{if } q\leq  \rho_1,\\
\norm{u}{\tLeb{\rho_1}^z\tLeb{\rho_2}^t(I\times J;\bes{s}{p,q})}&\geq  \norm{u}{\Leb{\rho_1}^z\Leb{\rho_2}^t(I;\bes{s}{p,q})}\quad \text{if } q\leq  \min\{\rho_1,\rho_2\}.
\end{align*}
We frequently write $\tLeb{\rho_1}^z\tLeb{\rho_2}^t\bes{s}{p,r}$ instead of $\tLeb{\rho_1}^z\tLeb{\rho_2}^t(I\times J;\bes{s}{p,r})$. Moreover, it is compatible with the real interpolation method.  See Theorem \ref{thm:CL-interpolation}. We also note that many properties of Besov spaces, such as Theorems \ref{thm:Besov-product} and \ref{thm:Moser-estimates}, remain true for Chemin-Lerner type spaces, see \cite[Sections 2.6]{BCD11} for the discussions.

\section{Lyapunov functional for the one-phase Muskat problem}\label{sec:Lyapunov}

In this section, we prove that $\Leb{2}$-norm is a Lyapunov function for the one-phase Muskat problem with surface tension. Without surface tension, one can easily show that if $f$ is a strong solution of \eqref{eq:one-phase-reformulation}, then $\Leb{2}$-norm is a Lyapunov function for the problem. Moreover, Alazard and Nguyen \cite[Proposition 2.2]{A25} proved the following coercive estimate
\begin{equation}\label{eq:trace-gravity}
\int_{\mathbb{R}^d} fG(f)f  \myd{x} \geq \frac{c}{1+\norm{\nabla f}{\BMO}}\norm{f}{\thSob{1/2}}^2
\end{equation}
for some constant $c>0$ (see also Nguyen \cite{N23}). Here $\norm{g}{\BMO}$ denotes the seminorm of bounded mean oscillation. However, if we take surface tension effect into account, then it is unclear whether $\Leb{2}$-norm is a Lyapunov functional for the problem, which is crucial for establishing global well-posedness of the one-phase Muskat problem with surface tension.

Note first that the paring $\action{H(f),G(f)f}$ is well defined for all $f\in \tSob{s}$, $s>d/2+1$ and satisfies
\[ \left|\action{H(f),G(f)f} \right|\leq \mathcal{F}(\norm{f}{\tSob{s}})\norm{f}{\tSob{s}}\norm{G(f)f}{\tSob{s-1}}\]
for all $f\in \tSob{s}$, $s>d/2+1$. Indeed, if $f\in \tSob{s}$, then $H(f)\in \tSob{s-2}$ and 
\begin{equation*}
 \norm{H(f)}{\tSob{s-2}}\leq \mathcal{F}(\norm{f}{\tSob{s}})\norm{f}{\tSob{s}}.
\end{equation*}
Since $\tSob{s-2}\hookrightarrow \tSob{1-s}$ when $s\geq 3/2$, we have $\norm{H(f)}{\tSob{1-s}}\leq \norm{H(f)}{\tSob{s-2}}$. Hence, the pairing is well defined. 

Now we present the main result of this section. For simplicity, we normalize all physical parameters $\rho$, $\kappa$, $\mathfrak{s}$, $\mathfrak{g}$ to be 1.

\begin{theorem}\label{thm:Lyapunov}
Let $s>d/2+1$. Then 
\begin{equation*}\label{eq:DN-mean-curvatures}
 \action{H(f),G(f)f}\geq 0
\end{equation*}
for any $f\in \tSob{s}(\mathbb{R}^d)$. Moreover, if $f$ satisfies \eqref{eq:one-phase-reformulation}, then 
\begin{equation}\label{eq:Lyapunov-estimates}
\frac{1}{2}\frac{d}{dt}\norm{f}{\Leb{2}}^2+\frac{1}{\mathcal{F}(\norm{f}{\tSob{s}})}(\norm{f}{\thSob{1/2}}^2+\norm{f}{\thSob{3/2}}^2)\leq 0
\end{equation}
for some nondecreasing function $\mathcal{F}:[0,\infty)\rightarrow(0,\infty)$.
\end{theorem}

When the spatial domain is a periodic domain, Theorem \ref{thm:Lyapunov} was proved by Alazard and Bresch \cite{AB24}. In this section, we extend their result to the whole space $\mathbb{R}^d$.  

By density argument, it suffices to show Theorem \ref{thm:Lyapunov} when $f\in C_c^\infty(\mathbb{R}^d)$. Indeed, given $f\in \tSob{s}(\mathbb{R}^d)$, $s>d/2+1$, there exists $f_k \in C_c^\infty(\mathbb{R}^d)$ such that $f_k\rightarrow f$ in $\tSob{s}(\mathbb{R}^d)$. Also, we have
\begin{align*}
\action{H(f),G(f)f}-\action{H(f_n),G(f_n)f_n}&=\action{H(f)-H(f_n),G(f)f}+\action{H(f_n),G(f_n)f_n-G(f)f}\\
&\leq \norm{H(f)-H(f_n)}{\tSob{s-2}}\norm{G(f)f}{\tSob{s-1}}\\
&\relphantom{=}+\norm{H(f_n)}{\tSob{s-2}}\norm{G(f_n)f_n-G(f)f}{\tSob{s-1}}.
\end{align*}

If we write 
\[ H_1(p)=(1+|p|^2)^{-1/2}-1,\]
then 
\[ H(f)=-\Delta f -\Div(\nabla f H_1(\nabla f)).\]
Note that $H_1(0)=0$ and $\nabla H_1(0)=0$. Since 
\begin{align*}
H(f)-H(f_n)&=-\Delta(f-f_n)-\Div(\nabla (f-f_n)H_1(\nabla f))-\Div(\nabla f_n(H_1(\nabla f)-H_1(\nabla f_n))),
\end{align*}
it follows from Theorems \ref{thm:Besov-product}, \ref{thm:Moser-estimates},  and the Sobolev embedding theorem that 
\begin{equation}\label{eq:mean-curvature-conv}
\begin{aligned}
&\relphantom{=}\norm{H(f)-H(f_n)}{\tSob{s-2}}\\
&\leq \norm{f-f_n}{\tSob{s}}+\norm{\nabla (f-f_n)H_1(\nabla f)}{\tSob{s-1}}+\norm{\nabla f_n(H_1(\nabla f)-H_1(\nabla f_n))}{\tSob{s-1}}\\
&\leq \mathcal{F}(\norm{f}{\tSob{s}})\norm{f-f_n}{\tSob{s}}\rightarrow 0
\end{aligned}
\end{equation}
as $n\rightarrow\infty$. By Proposition \ref{prop:continuity-DN}, we have 
\begin{align*}
\norm{G(f)f}{\tSob{s-1}}&\leq \mathcal{F}(\norm{f}{\tSob{s}})\norm{f}{\tSob{s}},\\
\norm{G(f_n)f_n-G(f)f}{\tSob{s-1}}&\leq \norm{G(f_n)(f_n-f)}{\tSob{s-1}}+\norm{(G(f_n)-G(f))f}{\tSob{s-1}}\\
&\leq \mathcal{F}(\norm{f}{\tSob{s}})\norm{f_n-f}{\tSob{s}}\rightarrow 0
\end{align*}
as $n\rightarrow\infty$. Hence, it follows that 
\[ \action{H(f),G(f)f}=\lim_{n\rightarrow\infty}\action{H(f_n),G(f_n)f_n}.\]

Throughout this section, we fix $f\in C_c^\infty(\mathbb{R}^d)$ and choose $R_0>0$ so that $\supp f \subset B_{R_0}$ and $\norm{f}{\Leb{\infty}}+1<R_0$. We first prove that the gradient of $\phi$ belongs to not only in $\Leb{2}(\Omega_f)$ but also in $\Leb{p}(\Omega_f)$ for $1<p<2$ in Section \ref{subsec:Lp-integrability}. This proposition will play an important role in controlling the behavior of harmonic functions. Theorem \ref{thm:Lyapunov} will be proved in Section \ref{subsec:pressure}.

\subsection{\texorpdfstring{$L_p$}{}-integrability of the gradient of harmonic functions}\label{subsec:Lp-integrability}
We know that if $f\in \tSob{1/2}$, then $\nabla_{x,y}\phi\in\Leb{2}(\Omega_f)$, which was used in the proof of the existence of the Dirichlet-Neumann operator in Proposition \ref{prop:existence-DN}. If we assume $f\in C_c^\infty(\mathbb{R}^d)$, then we have the following integrability result:
\begin{lemma}\label{lem:boundary-behavior-harmonic}
Let $\phi$ solve \eqref{eq:DN-elliptic} with $g=f$. Then $\nabla_{x,y} \phi \in \Leb{p}(\Omega_f)$ for $1<p\leq 2$.
\end{lemma}
\begin{proof}
For simplicity, we write $X=(x,y)$, where $x\in\mathbb{R}^d$ and $y\in \mathbb{R}$. 
By  \cite[Proposition 3.6]{NP20}, there exists a unique variational solution $\phi \in \thSob{1}(\Omega_f)$ to \eqref{eq:DN-elliptic} with $g=f$.  

We claim that for large $|X|$, we have 
\begin{equation}\label{eq:decay-estimate-nabla-phi}
 |\nabla_{x,y} \phi(X)|\apprle \frac{1}{|X|^{d+1}}.
\end{equation}

Note that $\phi$ solves 
\begin{equation*}
\left\{
\begin{alignedat}{2}
\Delta_{x,y} \phi&=0&&\quad \text{in } \{X=(x,y) : |X|>R_0, y<0 \},\\
\phi&=0&&\quad \text{on } \{X=(x,y) : |x|>R_0, y=0 \}.
\end{alignedat}
\right.
\end{equation*}
Extend $\phi$ to $\mathbb{R}^{d+1}\setminus B_{R_0}$ by odd extension with respect to $y$. Then $\phi$ satisfies $\Delta_{X}\phi=0$ in $\mathbb{R}^{d+1}\setminus B_{R_0}$. Now choose a cut-off function $\eta \in C^\infty(\mathbb{R}^{d+1})$ so that $\eta = 0$ in $B_{3R_0/2}$ and $\eta=1$ in  $\mathbb{R}^{d+1}\setminus B_{2R_0}$. Let $\Gamma$ be the fundamental solution to the Laplacian $-\Delta_{X}$. Then if we define 
\[ v(X)=-\int_{\mathbb{R}^{d+1}} \Gamma(X-Y)\Delta(\eta \phi)(Y)dY,\]
then 
\[ \Delta_{X} v = \Delta_{X} (\eta \phi)\quad \text{in } \mathbb{R}^{d+1}.\]
Moreover, we have 
\[ \Delta_{X} (\nabla_{X} v)=\Delta_{X} (\nabla_{X}(\eta \phi)).\]

Since $\nabla_{X} \phi \in \Leb{2}(\Omega_f)$, it follows from the Calder\'on-Zygmund estimate that $\nabla_{X} v\in \Leb{2}(\mathbb{R}^{d+1})$. Hence, by the Liouville theorem for harmonic functions, we conclude that $\nabla_{X}v = \nabla_{X}(\eta \phi)$. 

For $|X|>4R_0$ and $|Y|<2R_0$, we take the Taylor expansion to get 
\begin{equation}
\nabla_{X}\Gamma(X-Y)=(\nabla_X \Gamma)(X)-(\nabla^2 \Gamma)(X)Y+O\left(\frac{1}{|X|^{d+2}}\right)
\end{equation}
uniformly in $Y\in B_{2R_0}$. Hence, it follows that 
\begin{equation}\label{eq:identity-expansion}
\begin{aligned}
&\relphantom{=}(\nabla_X v)(X)\\
&=-(\nabla_X \Gamma)(X)\int_{\mathbb{R}^{d+1}} \Delta_Y(\eta\phi)(Y)dY+(\nabla^2\Gamma)(X)\int_{\mathbb{R}^{d+1}}Y\Delta_{Y}(\eta \phi)(Y)\myd{Y}+O\left(\frac{1}{|X|^{d+2}}\right).
\end{aligned}
\end{equation}

Note that 
\[ \int_{\mathbb{R}^{d+1}}\Delta_{Y}(\eta \phi)(Y)dY=\int_{B_{2R_0}\setminus B_{3R_0/2}} \Delta_Y(\eta \phi)(Y)dY=0.\]
Indeed, it follows from the divergence theorem and the choice of the cutoff function that 
\[ \int_{B_{2R_0}\setminus B_{3R_0/2}} \Delta_{Y}(\eta \phi)(Y)dY=\int_{\partial B_{2R_0}} (\nabla_{Y} \phi)(Y)\cdot Y \myd{\mathcal{H}^d(Y)}.\]
Since $\phi$ is odd with respect to $y$, $\partial_{x_j}\phi$ is odd with respect to $y$ and 
\begin{align*}
 &\relphantom{=}\int_{\partial B_{2R_0}} (\nabla_{X} \phi)(Y)\cdot Y \myd{\mathcal{H}^d(Y)}\\
 &=\sum_{j=1}^d \int_{\partial B_{2R_0}} (\partial_{x_j}\phi)(Y)x_j \myd{\mathcal{H}^d(Y)}+\int_{\partial B_{2R_0}} (\partial_y \phi)(Y)y \myd{\mathcal{H}^d}(Y)=0.
\end{align*}

Hence, by \eqref{eq:identity-expansion}, we have
\begin{equation*}
|(\nabla_{X} \phi)(X)|=|(\nabla_{X} v)(X)|\apprle \frac{1}{|X|^{d+1}}\quad \text{for } |X|>4R_0,
\end{equation*}
which proves \eqref{eq:decay-estimate-nabla-phi}.

Now we are ready to complete the proof of Lemma \ref{lem:boundary-behavior-harmonic}. On a bounded region, it follows from H\"older's inequality that 
\begin{equation*}
\left(\int_{B_{4R_0}'\times (-4R_0,4R_0)} |\nabla_{X}\phi|^p\myd{x}dy\right)^{1/p}\apprle_{p,R_0} \norm{\nabla_{X}\phi}{\Leb{2}},
\end{equation*}
where $B_r'$ denotes the $d$-dimensional ball with radius $r$ centered at the origin.

For large $|Y|$, it follows from \eqref{eq:decay-estimate-nabla-phi} that 
\begin{equation*}
\int_{\mathbb{R}^{d+1}\setminus B_{4R_0}} |\nabla_{X}\phi|^p\myd{Y}\apprle \int_{\mathbb{R}^{d+1}\setminus B_{4R_0}} \frac{1}{|Y|^{(d+1)p}} dY\apprle_{R,p} 1.
\end{equation*}
This completes the proof of Lemma \ref{lem:boundary-behavior-harmonic}.
\end{proof}

\subsection{Estimates on the pressure}\label{subsec:pressure}
In this subsection, we estimate the hydraulic pressure $Q:\overline{\Omega}_f\rightarrow\mathbb{R}$ defined by 
\[ Q(x,y)=\phi(x,y)-y.\]

The following proposition will play an important role in showing that the $\Leb{2}$-norm is a Lyapunov functional for the one-phase Muskat problem, where the periodic case can be found in Alazard and Bresch \cite[Proposition 3.2]{AB24}.

\begin{proposition}\label{prop:estimates-on-pressure} Suppose that $f\in C_c^\infty(\mathbb{R}^d)$ and let $\phi$ solve \eqref{eq:DN-elliptic} with $g=f$.
\begin{enumerate}[label=\textnormal{(\roman*)}]
\item The function $Q$ satisfies the following properties: 
\[
\partial_n Q=-|\nabla_{x,y} Q|\quad\text{and}\quad n=-\frac{\nabla_{x,y}Q}{|\nabla_{x,y}Q|},
\]
where $n$ denotes the unit normal to $\Sigma_f$ given by 
\[ n=\frac{1}{\sqrt{1+|\nabla f|^2}}(-\nabla f,1).\]
Moreover, there exists a constant $c_0>0$ such that the Taylor coefficient defined by
\[ a(x)=-(\partial_y Q)(x,f(x))\]
satisfies $a(x)\geq c_0>0$ for all $x\in\mathbb{R}^d$.
\item For all $(x,y)\in \overline{\Omega_f}$, there holds 
\[ \partial_y Q(x,y)<0.\]
Furthermore, we have 
\begin{equation}
 \inf_{\overline{\Omega}_f} (-\partial_y Q)\geq \min\{\inf_{x\in\mathbb{R}^d} a(x),1\}.
\end{equation}
\item The function $|\nabla_{x,y} Q|$ belongs to $C^\infty(\overline{\Omega})$;
\item For $s>d/2+1$, we have the following bound:
\[  \sup_{(x,y)\in \overline{\Omega}} |\nabla_{x,y} Q(x,y)|^2\leq \mathcal{F}(\norm{f}{\tSob{s}}).\]
\item We have $\nabla_{x,y}^2 Q\in \Leb{2}(\Omega_f)$.
\end{enumerate}
\end{proposition}
\begin{proof}
(i) Following an argument as in \cite[Proposition 4.3]{NP20}, one can show that $Q\geq 0$ in $\Omega_f$. Moreover, $Q>0$ in $\Omega_f$. If not, there exists an interior point $X_0=(x_0,y_0)\in \Omega_f$ such that $Q(X_0)=0$ and $\overline{B_r(X_0)}\subset \Omega$ for some $r>0$. Since $Q\geq 0$, it follows from the strong minimum principle that $Q$ is identically zero in $B_r(X_0)$. Hence, by the unique continuation principle for harmonic functions, $Q$ is identically zero in $\Omega_f$. This implies that $\phi(x,y)=y$ in $\Omega_f$, which contradicts that $\nabla_{x,y} \phi\rightarrow 0$ as $y\rightarrow-\infty$. 

Since $\Sigma_f$ is a graph of a smooth function, it follows from Hopf's lemma that $\partial_n Q<0$ on $\Sigma_f$. Since $Q=0$ on $\Sigma_f$, we have $|\nabla_{x,y} Q|=|\partial_n Q|=-\partial_n Q$ on $\Sigma_f$. Also, we have
\[ 0=(\nabla_x Q)(x,f(x))+(\partial_y Q)(x,f(x))(\nabla f)(x).\]
By algebraic manipulation and Proposition \ref{prop:parabolicity-DN}, we have
\begin{equation}\label{eq:Taylor-coefficients-estimate}
 a(x)=-(\partial_y Q)(x,f(x))=\frac{\nabla_x f\cdot \nabla_x Q-\partial_y Q}{1+|\nabla f|^2}=\frac{1-G(f)f(x)}{1+|\nabla f(x)|^2}>c_0>0
 \end{equation}
for all $x\in\mathbb{R}^d$. 
This implies that 
\[ n= -\frac{\nabla_{x,y} Q}{|\nabla_{x,y} Q|}.\]

(ii) For $R>0$, write 
\begin{equation}\label{eq:truncated-domain}
 \Omega_f^R=\{(x,y) : |x|<R, -R<y<f(x)\}\quad \text{and}\quad \Sigma_f^R=\{(x,f(x)): |x|<R\}.
\end{equation}
Then 
\[ \partial\Omega_f^R=\Sigma_f^R\cup \{(x,y):|x|<R, y=-R\} \cup \{(x,y):|x|=R,-R<y<f(x)\}.\]

Since $-\partial_y Q$ is harmonic in $\Omega_f$, we can apply the minimum principle in $\Omega_f^R$ that 
\begin{equation}\label{eq:minimum-principle}
\begin{aligned}
 -\partial_y Q \geq \inf_{\partial \Omega_f^R} (-\partial_y Q)\geq \min\{\inf_{\Sigma_R}(-\partial_y Q), \inf_{\{|x|=R,-R<y<f(x)\}}(-\partial_y Q), \inf_{\{|x|<R, y=-R\}} (-\partial_y Q)\}.
\end{aligned}
\end{equation}
By Lemma \ref{lem:boundary-behavior-harmonic} and the interior estimate for harmonic functions,  letting $R\rightarrow\infty$ in \eqref{eq:minimum-principle}, we get 
\[ -\partial_y Q \geq \min \{ \inf_{\Sigma} (-\partial_y Q),1\}=\min\{\inf_{x\in\mathbb{R}^d} a(x),1\}.\]

(iii) By (i) and (ii), $|\nabla_{x,y}Q|$ has a positive lower bound. Since $Q$ is harmonic, it follows that $|\nabla_{x,y} Q|$ is smooth.

(iv) Since $Q$ is harmonic, we have
\[ \Delta_{x,y} |\nabla_{x,y} Q|^2 = 2|\nabla_{x,y}^2 Q|\geq 0.\]
Hence, it follows from the maximum principle for subharmonic functions that 
\[ \sup_{\overline{\Omega}_f^R} |\nabla_{x,y} Q|^2 =\sup_{\partial\Omega_f^R} |\nabla_{x,y} Q|^2.\]
Since
\[ |\nabla_{x,y} Q|^2 = |(\nabla_x \phi,\partial_y \phi-1)|^2=|\nabla_x \phi|^2+|\partial_y \phi-1|^2,\]
it follows from Lemma \ref{lem:boundary-behavior-harmonic} and the interior estimate for harmonic function that 
\[ \sup_{\overline{\Omega}_f} |\nabla_{x,y} Q|^2 =\max\left\{\sup_{\Sigma_f} |\nabla_{x,y}Q|^2,1\right\}.\]

On the other hand,  it follows from \eqref{eq:Taylor-coefficients-estimate} that
\begin{align*}
|\nabla_{x,y} Q|^2 &= |\nabla_x Q|^2+|\partial_y Q|^2 = (\partial_y Q)^2(1+|\nabla_x f|^2)\\
&=\frac{(1-G(f)f(x))^2}{1+|\nabla f(x)|^2}.
\end{align*}
Hence, by Proposition \ref{prop:continuity-DN} and the Sobolev embedding theorem, 
$$
|(\nabla_{x,y} Q)(x,f(x))|^2\leq \mathcal{F}(\norm{f}{\tSob{s}}).
$$

(v) It suffices to show that $\nabla^2_{x,y} \phi \in \Leb{2}(\Omega_f)$. We straighten the boundary as below. Define 
\[ \Phi(x,z)=(x,f(x)+z),\quad \Phi:\mathbb{R}^d\times (-\infty,0)\rightarrow \Omega_f.\]
Then the inverse of $\Phi$ is 
\[ \Psi(x,z)=(x,z-f(x)).\]
Define $v(x,z)=\phi(x,f(x)+z)$. Then $v:\mathbb{R}^d\times (-\infty,0)\rightarrow\mathbb{R}$ and $v(x,0)=f(x)$. Moreover, $v$ satisfies 
\[ \Div_{x,z} (A(x)\nabla_{x,z} v)=0,\quad v(x,0)=f(x),\]
where 
\[ A(x)=
\begin{bmatrix}
I & -\nabla f\\
-(\nabla f)^T &1+|\nabla_x f|^2
\end{bmatrix}.
\]

Since $f\in C_c^\infty(\mathbb{R}^d)$, the coefficient $A$ is bounded and uniformly continuous. Let $w(x,z)=e^{z|\nabla|} f(x)$ and define $u=v-w$. Then $u$ satisfies
\[ \Div_{x,z}(A(x)\nabla_{x,z}u)=-\Div_{x,z}(A(x)\nabla_{x,z} w),\quad u(x,0)=0\]
and 
\[ \norm{\nabla_{x,z}u }{\Leb{2}(\mathbb{R}^d\times \{z<0\})}\leq C\norm{\nabla_{x,z} w}{\Leb{2}(\mathbb{R}^d\times \{z<0\})}\apprle \norm{f}{H^1}.\]
Moreover, a method of finite difference quotient gives 
\[ \norm{\nabla_{x,z}^2u }{\Leb{2}(\mathbb{R}^d\times \{z<0\})}\apprle \norm{f}{H^{3/2}}.\]
Since $\nabla_{x,z}^2 w \in \Leb{2}(\mathbb{R}^d\times\{z<0\})$, we have $\nabla_{x,z}^2 v\in \Leb{2}(\mathbb{R}^d\times \{z<0\})$. Hence, $\nabla_{x,y}^2\phi \in \Leb{2}(\Omega_f)$. This completes the proof of Proposition \ref{prop:estimates-on-pressure}.
\end{proof}

\subsection{Proof of Theorem \ref{thm:Lyapunov}}
Let $f\in C_c^\infty(\mathbb{R}^d)$. We first observe the following identity
\begin{equation}\label{eq:Q-identity}
 J(f):=\int_{\mathbb{R}^d} H(f)G(f)f\myd{x}=\int_{\Sigma_f} \partial_n |\nabla_{x,y} Q|d\mathcal{H}^d.
\end{equation}

Indeed, by the definition of the Dirichlet-Neumann operator, we have  
\[  \int_{\mathbb{R}^d} H(f)G(f)f\myd{x}=\int_{\mathbb{R}^d} H(f)(x)(\partial_N \phi)(x,f(x))\myd{x}.\]

Since 
\[ \partial_N Q=\partial_N\phi-1\quad \text{on } \Sigma_f \quad \text{and}\quad \int_{\mathbb{R}^d} H(f)\myd{x}=0,\]
it follows that 
\[ J(f)=\int_{\mathbb{R}^d} H(f)(\partial_N Q)|_{y=f}\myd{x}=\int_{\Sigma_f} H(f)\partial_n Q \myd{\mathcal{H}^d}.\]

Since $\partial_y Q<0$ and $Q(x,f(x))=0$, we have
\[ (\nabla f)(x) =\frac{1}{1-(\partial_y \phi)(x,f(x))}(\nabla_x \phi)(x,f(x)).\]
This gives 
\begin{align*}
H(f)&=-\Div_x\left(\frac{\nabla_x \phi}{\sqrt{(1-\partial_y \phi)^2+|\nabla_x \phi|^2}}\right)-\partial_y\left(\frac{\partial_y\phi-1}{\sqrt{(1-\partial_y \phi)^2+|\nabla_x \phi|^2}} \right)\\
&=-\Div_{x,y}\left(\frac{\nabla_{x,y} Q}{|\nabla_{x,y} Q|} \right)
\end{align*}
on $\Sigma_f$. On the other hand, using $\partial_n Q=-|\nabla_{x,y} Q|$, we get 
\[ \int_{\Sigma_f} H(f)\partial_n Q \myd{\mathcal{H}^d}=\int_{\Sigma_f} \Div_{x,y}\left(\frac{\nabla_{x,y} Q}{|\nabla_{x,y} Q|}\right) |\nabla_{x,y} Q|\myd{\mathcal{H}^d}.\]

Since $\Div_{x,y} (\nabla_{x,y} Q)=0$ and $n=-\frac{\nabla_{x,y} Q}{|\nabla_{x,y} Q|}$, we can write
\[ \Div\left(\frac{\nabla_{x,y} Q}{|\nabla_{x,y} Q|}\right) |\nabla_{x,y} Q|=-\frac{\nabla_{x,y} Q}{|\nabla_{x,y} Q|} \cdot \nabla_{x,y} |\nabla_{x,y} Q|=n\cdot \nabla_{x,y} |\nabla_{x,y} Q|=\partial_n |\nabla_{x,y} Q|,\]
which implies the identity \eqref{eq:Q-identity}.

The following rigidity estimate will help us to show \eqref{eq:Lyapunov-estimates}.
\begin{lemma}\label{lem:norm-estimate-linear-algebra}
If $A$ is symmetric $(d+1)\times (d+1)$ matrix and has trace zero, then 
\[ |A|^2 |v|^2 -|Av|^2 \geq \frac{1}{d+1} |A|^2 |v|^2\]
for all $v\in\mathbb{R}^{d+1}$.
\end{lemma}
\begin{proof}
It suffices to show that 
\[ \frac{d}{d+1}|A|^2 |v|^2 \geq |Av|^2.\]

Since $A$ is symmetric, it follows from the spectral theorem that there exist an orthogonal matrix $Q$ and a diagonal matrix $D=\mathrm{diag}(\lambda_1,\dots,\lambda_{d+1})$ such that $Q^TAQ=D$. Hence, 
\begin{equation}\label{eq:linear-algebra-A-bound}
 |Av|^2 \leq \left(\max_{1\leq i\leq d+1} \lambda_i^2\right) |v|^2.
\end{equation}
Without loss of generality, $\lambda_1^2=\max_{1\leq i\leq d+1} \lambda_i^2$. Since $A$ has zero trace, it follows from Cauchy-Schwarz inequality that 
\[ |\lambda_1|^2=\left|-\sum_{i=2}^{d+1} \lambda_i \right|^2\leq d \sum_{i=2}^{d+1} \lambda_i^2=d\left(\sum_{i=1}^{d+1}\lambda_i^2-\lambda_1^2\right). \]
This implies that 
\[ (d+1)\lambda_1^2 \leq d \sum_{i=1}^{d+1}\lambda_i^2= d|A|^2.\]
Hence, by \eqref{eq:linear-algebra-A-bound}, we get the desired result.
\end{proof}

Theorem \ref{thm:Lyapunov} will be proved by observing another identity:
\begin{lemma}\label{lem:Lyapunov-lemma}
We have 
\begin{equation}\label{eq:identity-Hessian}
 J(f)=\int_{\Omega_f} \frac{|\nabla_{x,y}Q|^2 |\nabla_{x,y}^2 Q|^2 -|\nabla_{x,y} Q \cdot \nabla_{x,y}\nabla_{x,y} Q|^2}{|\nabla_{x,y} Q|^3}\myd{x}dy\geq \frac{1}{\mathcal{F}(\norm{f}{\tSob{s}})}\norm{f}{\thSob{3/2}}^2.
\end{equation}
\end{lemma}
\begin{proof}
First of all, it follows from Proposition \ref{prop:estimates-on-pressure} that $\nabla_{x,y}^2 Q\in \Leb{2}(\Omega_f)$ and there exists a constant $c>0$ such that 
\[ \mathcal{F}(\norm{f}{\tSob{s}})\geq |\nabla_{x,y}Q(x,y)|\geq c>0\quad\text{for all } (x,y)\in \Omega_f. \]
These imply that the integral in \eqref{eq:identity-Hessian} is well defined. Since $Q$ is harmonic in $\Omega_f$, by Proposition \ref{prop:estimates-on-pressure}, Lemma \ref{lem:norm-estimate-linear-algebra}, and the trace lemma, we get the inequality in \eqref{eq:identity-Hessian}.
 
 Hence, it suffices to show that the identity holds. Recall the definitions of $\Omega^R_f$ and $\Sigma^R_f$ given in \eqref{eq:truncated-domain}.  For $R>R_0$, where $R_0>0$ is a fixed number so that $\supp f \subset B_{R_0}$ and $\norm{f}{\Leb{\infty}}+1<R_0$, it follows from the divergence theorem that 
\begin{align*}
\int_{\Omega^R_f} \Delta_{x,y}(|\nabla_{x,y} Q|)\myd{x}dy&=\int_{\partial \Omega^R_f} \partial_n |\nabla_{x,y}Q|d\mathcal{H}^d\\
&=\int_{\Sigma^R_f} \partial_n |\nabla_{x,y} Q| \myd{\mathcal{H}^d}+\int_{\{(x,-R) : |x|<R\}} \partial_n |\nabla_{x,y} Q| \myd{\mathcal{H}^d}\\
&\relphantom{=}+\int_{\{(x,y) : |x|=R, -R<y<f(x)\}} \partial_n |\nabla_{x,y} Q| \myd{\mathcal{H}^d}.
\end{align*}	

We claim that 
\begin{equation}\label{eq:boundary-kill-term}
\int_{\{(x,-R) : |x|<R\}} \partial_n |\nabla_{x,y} Q| \myd{\mathcal{H}^d}+\int_{\{(x,y) : |x|=R, -R<y<f(x)\}} \partial_n |\nabla_{x,y} Q| \myd{\mathcal{H}^d}\rightarrow 0
\end{equation}
as $R\rightarrow\infty$. If so, then by the dominated convergence theorem, \eqref{eq:Q-identity}, and the identity
\[
\Delta_{x,y} |\nabla_{x,y} Q|=\frac{1}{|\nabla_{x,y} Q|^3} (|\nabla_{x,y} Q|^2 |\nabla^2_{x,y}Q|^2-|\nabla_{x,y} Q \cdot \nabla_{x,y}\nabla_{x,y} Q|^2),
\]
 we get the desired result.

To estimate the first integral in \eqref{eq:boundary-kill-term}, since
\[ |\partial_y |\nabla_{x,y} Q||\apprle |\nabla_{x,y}^2 Q|,\]
it follows that 
\begin{equation}\label{eq:boundary-norm-control-I}
\left|\int_{\{(x,-R):|x|<R\}} \partial_n |\nabla_{x,y} Q|\myd{\mathcal{H}^d}\right|\apprle \int_{B_R} |\nabla_{x,y}^2 Q(x,-R)|\myd{x}.
\end{equation}

Choose $1<p\leq 2$ so that $1-2/(d+1)<1/p$.  Since $\nabla_{y,z}\phi \in \Leb{p}(\Omega_f)$, it follows from the interior estimate for the harmonic function and H\"older's inequality, we have
\begin{equation}\label{eq:interior-Q}
\begin{aligned}
|\nabla_{x,y}^2 Q(x,-R)|=|\nabla_{x,y}^2\phi(x,-R)|&\apprle\frac{1}{R^{d+2}}\int_{B_{R/10}(x,-R)} |\nabla_{y,z} \phi| dzdy\\
&\apprle R^{-1-\frac{d+1}{p}}\norm{\nabla_{x,y}\phi}{\Leb{p}(\Omega_f)}.
\end{aligned}
\end{equation}
Hence, by \eqref{eq:boundary-norm-control-I} and \eqref{eq:interior-Q}, we have
\begin{equation}\label{eq:decay-rate-harmonic-estimates-1}
\left|\int_{\{(x,-R):|x|<R\}} \partial_n |\nabla_{x,y} Q|\myd{\mathcal{H}^d}\right|\apprle_{p,f,d} {R}^{d-1-\frac{d+1}{p}}.
\end{equation}

To estimate the second integral in \eqref{eq:boundary-kill-term}, we note that it is equal to 
\begin{equation}\label{eq:second-integral-estimate-interior}
 \int_{\partial B_R'} \int_{-R}^{f(x)} \frac{x}{R}\cdot \nabla_x (|\nabla_{x,y} Q|)dyd\mathcal{H}^{d-1}(x).\end{equation}

We may assume that $R>4R_0$. For $|x|=R$, since $\phi(z,0)=0$ for $z\in B_{R/10}'(x)$, by taking the odd extension, it follows from the interior estimates for harmonic functions and Proposition \ref{prop:estimates-on-pressure} that  
\begin{equation}\label{eq:hessian-phi}
|\nabla^2\phi(x,y)|\apprle \frac{1}{R^{d+2}}\int_{B_{R/10}(x,y)} |\nabla_{z,w} \phi |dzdw\apprle_{p,d,f} R^{-1-\frac{d+1}{p}}
\end{equation}
for all $(x,y) \in (B_{11R/10}'\setminus B_{9R/10}')\times (-R,0)$. Then by \eqref{eq:hessian-phi}, \eqref{eq:second-integral-estimate-interior} is bounded by  
\begin{equation}\label{eq:decay-rate-harmonic-estimates-2}
\int_{\partial B_R'} \int_{-R}^0 |\nabla_{x,y}^2 Q|\myd{y}d\mathcal{H}^{d-1}\apprle R^{d-1-\frac{d+1}{p}}.
\end{equation}
Hence, by letting $R\rightarrow\infty$ in \eqref{eq:decay-rate-harmonic-estimates-1} and \eqref{eq:decay-rate-harmonic-estimates-2}, the claim \eqref{eq:boundary-kill-term} holds and we get 
\[  J(f)=\lim_{R\rightarrow \infty}\int_{\Sigma_f^R} \partial_n |\nabla_{x,y}Q|d\mathcal{H}^d=\int_{\Omega_f} \frac{|\nabla_{x,y}Q|^2 |\nabla_{x,y}^2 Q|^2 -|\nabla_{x,y} Q \cdot \nabla_{x,y}\nabla_{x,y} Q|^2}{|\nabla_{x,y} Q|^3}\myd{x}dy.\]
This completes the proof of Lemma \ref{lem:Lyapunov-lemma}.
 \end{proof}

\begin{proof}[Proof of Theorem \ref{thm:Lyapunov}]
If $f$ is a strong solution of  \eqref{eq:one-phase-reformulation}, then 
\[ \frac{1}{2}\frac{d}{dt}\norm{f}{\Leb{2}}^2+\int_{\mathbb{R}^d} f G(f)f \myd{x}+\int_{\mathbb{R}^d} fG(f)H(f) \myd{x}=0.\]
Since $G(f)$ is self-adjoint on $\Leb{2}$, it follows from Lemma \ref{lem:Lyapunov-lemma}, Proposition \ref{prop:estimates-on-pressure}, and \eqref{eq:trace-gravity} that 
\[ \frac{1}{2}\frac{d}{dt}\norm{f}{\Leb{2}}^2+\frac{1}{\mathcal{F}(\norm{f}{\tSob{s}})}(\norm{f}{\thSob{1/2}}^2+\norm{f}{\thSob{3/2}}^2)\leq 0.\]
This completes the proof of Theorem \ref{thm:Lyapunov}.
\end{proof}

\section{First-order expansion of the Dirichlet-Neumann operator}\label{sec:first-order-expansion}

Since we reformulated the one-phase Muskat problem in terms of the Dirichlet-Neumann operator, it is crucial to understand the structure of the Dirichlet-Neumann operator. In this point of view, using Alinhac's good unknown and paralinearization argument, Nguyen \cite{NP20} proved local well-posedness of the one-phase Muskat problem with surface tension for any initial data in $\tSob{s}$, $s>d/2+1$. However, this approach is not suitable for obtaining global well-posedness of the problem since it will lead to temporal growth in the energy estimate. See \eqref{eq:energy-estimate-surface-tension}. 

Nevertheless, if we assume smallness of the interface in Chemin-Lerner spaces, then we have the first-order Taylor expansion of the Dirichlet-Neumann operator as follows.

\begin{theorem}\label{thm:first-order-expansion}
Let $1\leq p,q\leq\infty$, $s\geq \sigma>d/p+1$, and $\sigma\geq \sigma_0\geq 1$. Then there exists a constant $\varepsilon_0=\varepsilon_0(s,\sigma,d)>0$ such that if 
\begin{equation}\label{eq:smallness-DN}
\norm{f}{\tLeb{\infty}^t\bes{s}{p,q}}<\varepsilon_0,
\end{equation}
then $G(f)g\in \tLeb{\infty}^t\bes{\sigma-1}{p,q}$ for any $g\in \tLeb{\infty}^t\bes{\sigma}{p,q}$ and 
\begin{equation}\label{eq:linearlization-DN}
G(f)g=|\nabla|g + \mathcal{R}(f;g), 
\end{equation}
where 
\begin{equation}\label{eq:estimates-remainder-DN}
\begin{aligned}
\norm{\mathcal{R}(f;g)}{\tLeb{\infty}^t\bes{\sigma-1}{p,q}}&\leq \mathcal{F}(\norm{f}{\tLeb{\infty}^t\bes{s}{p,q}})\norm{f}{\tLeb{\infty}^t\bes{s}{p,q}}\norm{g}{\tLeb{\infty}^t\bes{\sigma}{p,q}},\\
\norm{\mathcal{R}(f;g)}{\tLeb{\infty}^t\hbes{\sigma_0-1}{p,q}}&\leq \mathcal{F}(\norm{f}{\tLeb{\infty}^t\bes{s}{p,q}})\norm{f}{\tLeb{\infty}^t\bes{s}{p,q}}(\norm{g}{\tLeb{\infty}^t\hbes{\sigma_0}{p,q}}+\norm{g}{\tLeb{\infty}^t\hbes{d/p+1}{p,1}}).
\end{aligned}
\end{equation}

Moreover, if $f_1,f_2$ satisfy \eqref{eq:smallness-DN}, then 
\begin{equation}\label{eq:contraction-DN}
\begin{aligned}
\norm{\mathcal{R}(f_1;g)-\mathcal{R}(f_2;g)}{\tLeb{\infty}^t\bes{\sigma-1}{p,q}}&\leq \mathcal{F}(\norm{f_1}{\tLeb{\infty}^t\bes{s}{p,q}},\norm{f_2}{\tLeb{\infty}^t\bes{s}{p,q}})\norm{f_1-f_2}{\tLeb{\infty}^t\bes{s}{p,q}}\norm{g}{\tLeb{\infty}^t\bes{\sigma}{p,q}}
\end{aligned}
\end{equation}
for all $g\in \tLeb{\infty}^t\bes{\sigma}{p,q}$.
\end{theorem}

\begin{remark}\leavevmode
\begin{enumerate}[label=\textnormal{(\roman*)}]
\item The Taylor expansion of the Dirichlet-Neumann operator \eqref{eq:linearlization-DN} is well known in the water wave literature. See \cite{GP25,N22} and references therein. However, the assumptions in each paper differ due to its purpose.
\item If $f$ and $g$ are time-independent, then one can obtain a similar theorem in Besov spaces.
\end{enumerate}
\end{remark}

This section is organized as follows. We first find the fixed point structure of the Dirichlet-Neumann operator to obtain estimates on the remainder in the appropriate function space, which will be presented in Section \ref{subsec:continuity-DN}. The contraction estimate of the Dirichlet-Neumann operator will be given in Section \ref{subsec:contraction-DN}. This will help us to establish global well-posedness of the one-phase Muskat problem with surface tension.

\subsection{Fixed point formulation}

We first find the fixed point structure of the Dirichlet-Neumann operator. Here we fix $t$ for simplicity.  Define 
\[  \rho(x,z)=z+\mathcal{P}(x,z),\quad \mathcal{P}(x,z)=e^{z|\nabla|} f(x),\quad (x,z)\in \mathbb{R}^d\times(-\infty,0).\]
Note that $\rho(x,0)=f(x)$ and $\rho(x,z)\rightarrow -\infty$ as $z\rightarrow -\infty$. Also,  the map $\Phi:(x,z)\mapsto (x,\rho(x,z))$ $\mathbb{R}^d\times(-\infty,0)$ to  $\Omega_f$. Also, note that  
\[ \partial_z \rho =1+|\nabla| \mathcal{P}.\]
By Proposition \ref{prop:properties-LP}, we have
\begin{equation*}
\norm{e^{z|\nabla|} |\nabla| f}{\Leb{\infty}}\apprle \sum_k e^{cz 2^k} 2^k \norm{P_k f}{\Leb{\infty}}\apprle \sum_k e^{cz 2^k} 2^{k(1+d/p)} \norm{P_k f}{\Leb{p}}\apprle\norm{f}{\hbes{d/p+1}{p,1}}
\end{equation*}
for any $z\leq 0$. Hence, there exists $\varepsilon_0>0$ such that if $f$ satisfies \eqref{eq:smallness-DN}, then $\Phi$ is Lipschitz diffeomorphism from $\mathbb{R}^d\times(-\infty,0)$ onto $\Omega_f$.

If $\phi$ satisfies $\Delta_{x,y}\phi=0$, then $v(x,z)=\phi(x,\rho(x,z))$ solves
\[ \Div_{x,z}(\mathcal{A}\nabla_{x,z} v)(x,z)=0\quad \text{in } \mathbb{R}^d\times(-\infty,0),\]
where 
\[ \mathcal{A}=\begin{bmatrix}
\partial_z \rho & -\nabla \rho\\
-(\nabla \rho)^T & \frac{1+|\nabla \rho|^2}{\partial_z\rho}
\end{bmatrix}
=I + \begin{bmatrix}
(|\nabla|\mathcal{P})I &-\nabla\mathcal{P}\\
-(\nabla_x\mathcal{P})^T & \frac{|\nabla_x\mathcal{P}|^2-|\nabla|\mathcal{P}}{1+|\nabla|\mathcal{P}}
\end{bmatrix}. \]
In other words, we have 
\begin{equation*}
 \Delta_{x,z} v =\partial_z Q_a[v]+\Div_x Q_b[v]\quad \text{in } \mathbb{R}^d\times(-\infty,0),
 \end{equation*}
 where
 \begin{equation}\label{eq:Qa-v-Qb-v}
 \begin{aligned}
 Q_a[v]&=\nabla_x \mathcal{P}\cdot \nabla_x v-\frac{|\nabla\mathcal{P}|^2-|\nabla|\mathcal{P}}{1+|\nabla|\mathcal{P}}\partial_z v,\\
 Q_b[v]&=(\partial_z v)\nabla_x\mathcal{P}-(\partial_z\mathcal{P})\nabla v.
 \end{aligned}
 \end{equation}
 
 By factorizing $\Delta_{x,z}v = (\partial_z+|\nabla|)(\partial_z-|\nabla|)v$, if we define 
 \begin{equation*}\label{eq:w-rep}
 w=(\partial_z-|\nabla|)v-Q_a[v],
\end{equation*}
 then $(w,v)$ satisfies
 \begin{equation}\label{eq:v-w-eqn}
 \left\{
 \begin{aligned}
 (\partial_z +|\nabla|)w&=-|\nabla|Q_a+\Div_x Q_b,\\
 (\partial_z-|\nabla|)v&=w+Q_a.
 \end{aligned}
 \right.
 \end{equation}

Since $v(x,z)=\phi(x,\rho(x,z))$, it follows from \eqref{eq:Qa-v-Qb-v} and \eqref{eq:v-w-eqn} that
\begin{equation}\label{eq:DN-remainder-expression}
\begin{aligned}
G(f)g&=\left.\frac{1+|\nabla\rho|^2}{\partial_z\rho}\partial_z v -\nabla\rho \cdot \nabla v\,\right|_{z=0}\\
&=\left.\left(1+\frac{|\nabla\mathcal{P}|^2-|\nabla|\mathcal{P}}{1+|\nabla|\mathcal{P}}\right)\partial_z v-\nabla\rho \cdot \nabla v\,\right|_{z=0}\\
&=(\partial_z v -Q_a[v])|_{z=0}\\
&=|\nabla|g+w|_{z=0}.
\end{aligned}
\end{equation}

Using $\partial_z v=|\nabla| v +w+Q_a[v]$ in \eqref{eq:v-w-eqn}, we note that $Q_a$ and $Q_b$ can be rewritten as 
\begin{equation*}
\begin{aligned}
Q_a[w,v;f]&=\frac{1}{1+\mathcal{B}}\nabla_x\mathcal{P}\cdot\nabla_x v- \frac{\mathcal{B}}{1+\mathcal{B}}(w+|\nabla|v),\\
Q_b[w,v;f]&=(|\nabla|v+w+Q_a[w,v;f])\nabla\mathcal{P}-(\partial_z\mathcal{P})\nabla v,
\end{aligned}
\end{equation*}
where 
\[ \mathcal{B}=\frac{|\nabla \mathcal{P}|^2-|\nabla|\mathcal{P}}{1+|\nabla|\mathcal{P}}.\]

By \eqref{eq:v-w-eqn}, we see that $(w,v)$ is a fixed point of 
\begin{equation}\label{eq:fixed-point-formula-DN}
\left\{
\begin{aligned}
w(x,z)&=\Pi_f^1(w,v)(x,z)\\
v(x,z)&=e^{z|\nabla|} g(x)+\Pi_f^2(w,v)(x,z).
\end{aligned}
\right.
\end{equation}
where 
\begin{equation*}
\begin{aligned}
\Pi_f^1(w,v)(x,z)&=\int_{-\infty}^z e^{-(z-\tau)|\nabla|} (\Div_x Q_b[w,v;f](x,\tau)-|\nabla|Q_a[w,v;f])(x,\tau)d\tau,\\
 \Pi_f^2(w,v)(x,z)&=-\int_z^0 e^{(z-\tau)|\nabla|} \{w(x,\tau)+Q_a[w,v;f](x,\tau)\}d\tau.
\end{aligned}
\end{equation*}

Hence, the goal is to find an appropriate function space so that \eqref{eq:fixed-point-formula-DN} has a unique fixed point $(w,v)$.

\subsection{Continuity estimates}\label{subsec:continuity-DN}

We use the Littlewood-Paley projection to design the function spaces. By Proposition \ref{prop:properties-LP} and Young's convolution inequality, for $1\leq p\leq \infty$ and $1\leq r_2\leq r_1\leq \infty$, we have
\begin{equation*}
\begin{aligned}
\norm{P_k \Pi_f^1(w,v)}{\Leb{r_1}^z\Leb{\infty}^t\Leb{p}^x}&\apprle 2^{k(1/r_2-1/r_1)}(\norm{P_k Q_a}{\Leb{r_2}^z\Leb{\infty}^t\Leb{p}^x}+\norm{P_k Q_b}{\Leb{r_2}^z\Leb{\infty}^t\Leb{p}^x}),\\
\norm{\nabla P_k\Pi_f^2(w,v)}{\Leb{r_1}^z\Leb{\infty}^t\Leb{p}^x}&\apprle 2^{k(1/r_2-1/r_1)}(\norm{P_k Q_a}{\Leb{r_2}^z\Leb{\infty}^t\Leb{p}^x}+\norm{P_k w}{\Leb{r_2}^z\Leb{\infty}^t\Leb{p}^x})
\end{aligned}
\end{equation*}
for all $k\in \mathbb{Z}$.

By the definition of the Chemin-Lerner spaces, for $\sigma\in\mathbb{R}$ and $1\leq q\leq \infty$, we have
\begin{equation}\label{eq:Pif-estimates}
\begin{aligned}
\norm{\Pi_f^1 (w,v)}{\tLeb{r_1}^z\tLeb{\infty}^t\hbes{\sigma+{1}/{r_1}-{1}/{r_2}}{p,q}}&\apprle \norm{Q_a}{\tLeb{r_2}^z\tLeb{\infty}^t\hbes{\sigma}{p,q}}+\norm{Q_b}{\tLeb{r_2}^z\tLeb{\infty}^t\hbes{\sigma}{p,q}}, \\
\norm{\nabla\Pi_f^2 (w,v)}{\tLeb{r_1}^z\tLeb{\infty}^t\hbes{\sigma+{1}/{r_1}-{1}/{r_2}}{p,q}}&\apprle \norm{Q_a}{\tLeb{r_2}^z\tLeb{\infty}^t\hbes{\sigma}{p,q}}+\norm{w}{\tLeb{r_2}^z\tLeb{\infty}^t\hbes{\sigma}{p,q}}.
\end{aligned}
\end{equation}

To control the inhomogeneous terms, it follows from \eqref{eq:Pif-estimates} with $(r_1,r_2)=(\infty,1)$ and $\sigma=1$ that
\begin{equation}\label{eq:Pif-estimates-inhomogeneous}
\begin{aligned}
\norm{\Pi_f^1(w,v)}{\Leb{\infty}^z{L}^t_{\infty}\Leb{p}^x}&\apprle\norm{Q_a}{\tilde{L}^z_1\tilde{L}^t_{\infty}\hbes{1}{p,1}}+\norm{Q_b}{\tilde{L}^z_1\tilde{L}^t_{\infty}\hbes{1}{p,1}},\\
\norm{\nabla\Pi_f^2(w,v)}{\Leb{\infty}^z{L}^t_{\infty}\Leb{p}^x}&\apprle\norm{Q_a}{\tilde{L}^z_1\tilde{L}^t_{\infty}\hbes{1}{p,1}}+\norm{w}{\tilde{L}^z_1\tilde{L}^t_{\infty}\hbes{1}{p,1}}.
\end{aligned}
\end{equation}

On the other hand, note that 
\begin{equation}\label{eq:g-estimate}
\norm{e^{z|\nabla|} g}{\tLeb{r_1}^z\tLeb{\infty}^t\hbes{\sigma+1/r_1}{p,q}}\apprle  \norm{g}{\tLeb{\infty}^t\hbes{\sigma}{p,q}}.
\end{equation}

From \eqref{eq:Pif-estimates} and \eqref{eq:Pif-estimates-inhomogeneous}, we introduce following norms:
\begin{align*}
 \norm{w}{\dot{Y}^\sigma}&=\norm{w}{\tLeb{\infty}^z\tLeb{\infty}^t\hbes{d/p}{p,1}}+\norm{w}{\tLeb{1}^z\tLeb{\infty}^t\hbes{d/p+1}{p,1}}+\norm{w}{\tLeb{\infty}^z\tLeb{\infty}^t\hbes{\sigma-1}{p,q}}+\norm{w}{\tLeb{1}^z\tLeb{\infty}^t\hbes{\sigma}{p,q}},\\
  \norm{w}{{Y}^\sigma}&= \norm{w}{\dot{Y}^\sigma}+\norm{w}{\Leb{\infty}^z\Leb{\infty}^t\Leb{p}^x}+\norm{w}{\tLeb{1}^z\tLeb{\infty}^t\hbes{1}{p,1}}.
 \end{align*}
Then by \eqref{eq:Pif-estimates}, \eqref{eq:Pif-estimates-inhomogeneous}, and the definition of norms, it remains to estimate $Q_a$ and $Q_b$ in the space $\tilde{L}^z_1\tilde{L}^t_{\infty}\hbes{\sigma}{p,q}$. 
For convenience, we introduce 
\[ \norm{w}{Z^\sigma_{p,q}}=\norm{w}{\Leb{\infty}^{z,t,x}}+\norm{w}{\tLeb{1}^z\tLeb{\infty}^t\hbes{\sigma}{p,q}}. \]

\begin{lemma}\label{lem:Qa-Qb-control}
Let $1\leq p,q\leq \infty$, $s>d/p+1$, and $s\geq \sigma \geq 1$. There exists $\varepsilon_0>0$ such that if 
\begin{equation}
\norm{f}{\tLeb{\infty}^t\bes{s}{p,q}}<\varepsilon_0, 
\end{equation}
then we have 
\begin{align*}
\norm{Q_a}{\tLeb{1}^z\tLeb{\infty}^t\hbes{\sigma}{p,q}}+\norm{Q_b}{\tLeb{1}^z\tLeb{\infty}^t\hbes{\sigma}{p,q}}&\leq \mathcal{F}(\norm{f}{\tLeb{\infty}^t\bes{s}{p,q}})\norm{f}{\tLeb{\infty}^t\bes{s}{p,q}}(\norm{\nabla v}{Z^\sigma_{p,q}}+\norm{w}{Z^\sigma_{p,q}}).
\end{align*}

\end{lemma}
The proof will be given in Appendix \ref{app:Qa-Qb-estimates}.  Then by the definition of $\dot{Y}^\sigma$, \eqref{eq:Pif-estimates} with $(r_1,r_2)\in\{(\infty,1),(1,1)\}$, and Lemma \ref{lem:Qa-Qb-control}, we have
\begin{equation}\label{eq:Pi1-mapping-Ydotsigma}
\begin{aligned}
\norm{\Pi_f^1(w,v)}{\dot{Y}^\sigma}&\leq \mathcal{F}(\norm{f}{\tLeb{\infty}^t\bes{s}{p,q}})\norm{f}{\tLeb{\infty}^t\bes{s}{p,q}}(\norm{w}{\dot{Y}^\sigma}+\norm{\nabla v}{\dot{Y}^\sigma})\\
\norm{\nabla\Pi_f^2(w,v)}{\dot{Y}^\sigma}&\leq \mathcal{F}(\norm{f}{\tLeb{\infty}^t\bes{s}{p,q}})\norm{f}{\tLeb{\infty}^t\bes{s}{p,q}}(\norm{w}{\dot{Y}^\sigma}+\norm{\nabla v}{\dot{Y}^\sigma})+\norm{w}{\dot{Y}^\sigma}.
\end{aligned}
\end{equation}
Also, by \eqref{eq:g-estimate}, we have
\begin{equation}\label{eq:g-lifting-homogeneous}
\norm{\nabla e^{z|\nabla|} g}{\dot{Y}^\sigma}\apprle \norm{\nabla g}{\tLeb{\infty}^t\hbes{d/p}{p,1}}+\norm{\nabla g}{\tLeb{\infty}^t\hbes{\sigma-1}{p,q}}\apprle \norm{g}{\tLeb{\infty}^t\hbes{d/p+1}{p,1}}+\norm{g}{\tLeb{\infty}^t\hbes{\sigma}{p,q}}.
\end{equation}

If $s,\sigma>d/p+1$, then \eqref{eq:Pif-estimates-inhomogeneous} and \eqref{eq:Pi1-mapping-Ydotsigma} give
\begin{equation}\label{eq:Pi1-mapping-Ysigma}
\begin{aligned}
\norm{\Pi_f^1(w,v)}{{Y}^\sigma}&\leq \mathcal{F}(\norm{f}{\tLeb{\infty}^t\bes{s}{p,q}})\norm{f}{\tLeb{\infty}^t\bes{s}{p,q}}(\norm{w}{{Y}^\sigma}+\norm{\nabla v}{{Y}^\sigma})\\
\norm{\nabla\Pi_f^2(w,v)}{{Y}^\sigma}&\leq \mathcal{F}(\norm{f}{\tLeb{\infty}^t\bes{s}{p,q}})\norm{f}{\tLeb{\infty}^t\bes{s}{p,q}}(\norm{w}{{Y}^\sigma}+\norm{\nabla v}{{Y}^\sigma})+\norm{w}{{Y}^\sigma},\\
\norm{\nabla e^{z|\nabla|} g}{{Y}^\sigma}&\apprle \norm{g}{\tLeb{\infty}^t\hbes{d/p+1}{p,1}}+ \norm{g}{\tLeb{\infty}^t\hbes{\sigma}{p,q}}+ \norm{g}{\tLeb{\infty}^t\hbes{1}{p,q}}\apprle \norm{g}{\tLeb{\infty}^t\bes{\sigma}{p,q}}.
\end{aligned}
\end{equation}

To perform Picard iteration and prove contraction estimates,  introduce
$$
\left\{
\begin{aligned}
	\delta Q_a &=Q_a[w_1,v_1;f_1]-Q_a[w_2,v_2;f_2],\\
	\delta Q_ b&=Q_b[w_1,v_1;f_1]-Q_b[w_2,v_2;f_2],\\
	\delta w & = w_1-w_2,\\
	\delta v & = v_1-v_2,\\
	\delta f &= f_1-f_2.
\end{aligned}
\right.$$

The following lemma will be proved in Appendix \ref{app:Qa-Qb-estimates}. 

\begin{lemma}\label{lem:difference-estimate}
Let $1\leq p,q\leq \infty$, $s\geq \sigma \geq 1$, and $s>d/p+1$. Then there exists $\varepsilon_0>0$ such that if 
\begin{equation}
\norm{f_i}{\tLeb{\infty}^t\bes{s}{p,q}}<\varepsilon_0, \quad i=1,2,
\end{equation} 
then 
\begin{align*}
\norm{\delta Q_a}{\tLeb{1}^z\tLeb{\infty}^t\hbes{\sigma}{p,q}}&\apprle  \mathcal{F}(\norm{f_1}{\tLeb{\infty}^t\bes{s}{p,q}},\norm{f_2}{\tLeb{\infty}^t\bes{s}{p,q}})\norm{\delta f}{\tLeb{\infty}^t\hbes{\sigma}{p,q}}(\norm{\nabla v_1}{Z^\sigma_{p,q}}+\norm{w_1}{Z^\sigma_{p,q}})\\
&\relphantom{=}+\mathcal{F}(\norm{f_2}{\tLeb{\infty}^t\bes{s}{p,q}})\norm{f_2}{\tLeb{\infty}^t\bes{s}{p,q}}(\norm{\delta \nabla v}{Z^\sigma_{p,q}}+\norm{\delta w}{Z^\sigma_{p,q}}),\\
\norm{\delta Q_b}{\tLeb{1}^z\tLeb{\infty}^t\hbes{\sigma}{p,q}}&\apprle \norm{\delta f}{\tLeb{\infty}^t\hbes{\sigma}{p,q}}\mathcal{F}(\norm{f_1}{\tLeb{\infty}^t\bes{s}{p,q}},\norm{f_2}{\tLeb{\infty}^t\bes{s}{p,q}})\left(\sum_{i=1}^2(\norm{\nabla v_i}{Z^\sigma_{p,q}}+\norm{w_i}{Z^\sigma_{p,q}})\right)\\
&\relphantom{=}+(\norm{f_1}{\tLeb{\infty}^t\bes{s}{p,q}}+\norm{f_2}{\tLeb{\infty}^t\bes{s}{p,q}})(\norm{\nabla_x(\delta v)}{Z^\sigma_{p,q}}+\norm{\delta w}{Z^\sigma_{p,q}})\\
&\relphantom{=}+\mathcal{F}(\norm{f_2}{\tLeb{\infty}^t\bes{s}{p,q}})\norm{f_2}{\tLeb{\infty}^t\bes{s}{p,q}}(\norm{\nabla_x(\delta v)}{Z^\sigma_{p,q}}+\norm{\delta w}{Z^\sigma_{p,q}}).
\end{align*}
\end{lemma}

\noindent\emph{Proof of Theorem \ref{thm:first-order-expansion}}. Define 
\[ (w_0,v_0)=(0,e^{z|\nabla|}g)\]
and 
\begin{equation*}
\left\{
\begin{aligned}
w_{n+1}&=\Pi_f^1(w_n,v_n)\\
v_{n+1}&=e^{z|\nabla|}g+\Pi_f^2(w_n,v_n)
\end{aligned}
\right.
\end{equation*}
for $n\geq 0$. By induction and \eqref{eq:Pi1-mapping-Ysigma}, we have $(w_n,\nabla v_n)\in {Y}^\sigma\times {Y}^\sigma$. Set 
$$
\left\{\begin{aligned}
	\delta w_n&=w_n-w_{n-1},\\
	\delta v_n&=v_n-v_{n-1},\\
	\delta Q_a&=Q_a[w_n,v_n;f]-Q_a[w_{n-1},v_{n-1};f],\\
	\delta Q_b&=Q_b[w_n,v_n;f]-Q_b[w_{n-1},v_{n-1};f].
\end{aligned}
\right.
$$
Then 
\begin{align*}
&\relphantom{=}\Pi_f^1(w_{n},v_{n})(x,t,z)-\Pi_f^1(w_{n-1},v_{n-1})(x,t,z)\\
&=\int_{-\infty}^z e^{-(z-\tau)|\nabla|} (\Div_x \delta Q_b[w,v;f](x,t,\tau)-|\nabla|\delta Q_a[w,v;f])(x,t,\tau)d\tau,\\
&\relphantom{=}\Pi_f^2(w_n,v_n)(x,t,z)-\Pi_f^2(w_{n-1},v_{n-1})(x,t,z)\\
&=-\int_{z}^0 e^{-(z-\tau)|\nabla|} (\delta w_n+\delta Q_a)(x,t,\tau)d\tau.
\end{align*}
By \eqref{eq:Pif-estimates}, \eqref{eq:Pif-estimates-inhomogeneous}, and Lemma \ref{lem:difference-estimate} with $f=f_1=f_2$, we have
and 
\begin{align*}
\norm{\delta w_{n+1}}{\dot{Y}^{\sigma_0}}&\leq \mathcal{F}(\norm{f}{\tLeb{\infty}^t\bes{s}{p,q}})\norm{f}{\tLeb{\infty}^t\bes{s}{p,q}}(\norm{\delta w_n}{\dot{Y}^{\sigma_0}}+\norm{\nabla(\delta v_n)}{\dot{Y}^{\sigma_0}}),\\
\norm{\nabla(\delta v_{n+1})}{\dot{Y}^{\sigma_0}}&\leq \mathcal{F}(\norm{f}{\tLeb{\infty}^t\bes{s}{p,q}})\norm{f}{\tLeb{\infty}^t\bes{s}{p,q}}(\norm{\delta w_n}{\dot{Y}^{\sigma_0}}+\norm{\nabla (\delta v_n)}{\dot{Y}^{\sigma_0}})+\norm{\delta w_n}{\dot{Y}^{\sigma_0}}.
\end{align*}
If we set $a_n=\norm{\delta w_n}{\dot{Y}^{\sigma_0}}$ and $b_n=\norm{\nabla (\delta v_n)}{\dot{Y}^{\sigma_0}}$, then we have 
\begin{align*}
a_{n+1}&\leq \mathcal{F}(\norm{f}{\tLeb{\infty}^t\bes{s}{p,q}})\norm{f}{\tLeb{\infty}^t\bes{s}{p,q}}(a_n+b_n),\\
b_{n+1}&\leq \mathcal{F}(\norm{f}{\tLeb{\infty}^t\bes{s}{p,q}})\norm{f}{\tLeb{\infty}^t\bes{s}{p,q}}(a_n+b_n)+\mathcal{F}(\norm{f}{\tLeb{\infty}^t\bes{s}{p,q}})a_n.
\end{align*}
In other words, we have 
\[
\begin{bmatrix}
a_{n+1} \\
b_{n+1}
\end{bmatrix}
\leq
\begin{bmatrix}
\mathcal{F}(\norm{f}{\tLeb{\infty}^t\bes{s}{p,q}})\norm{f}{\tLeb{\infty}^t\bes{s}{p,q}}& \mathcal{F}(\norm{f}{\tLeb{\infty}^t\bes{s}{p,q}})\norm{f}{\tLeb{\infty}^t\bes{s}{p,q}} \\
\mathcal{F}(\norm{f}{\tLeb{\infty}^t\bes{s}{p,q}})(\norm{f}{\tLeb{\infty}^t\bes{s}{p,q}}+1) & \mathcal{F}(\norm{f}{\tLeb{\infty}^t\bes{s}{p,q}})\norm{f}{\tLeb{\infty}^t\bes{s}{p,q}}
\end{bmatrix}
\begin{bmatrix}
a_n \\
b_n
\end{bmatrix}
\]
for all $n$. Similarly, if $a_n=\norm{\delta w_n}{Y^\sigma}$ and $b_n=\norm{\nabla(\delta v_n)}{Y^\sigma}$, then we get the same inequality. 

Choose $\varepsilon_0>0$ sufficiently small so that $\mathcal{F}(\norm{f}{\tLeb{\infty}^t\bes{s}{p,q}})\leq C$ and $C\left(\varepsilon_0+\sqrt{\varepsilon_0+\varepsilon_0^2}\right)\leq 1/4$. Then the spectral radius of 
\[ C\begin{bmatrix}
\varepsilon_0 & \varepsilon_0\\
1+\varepsilon_0 & \varepsilon_0
\end{bmatrix}
\]
is less than $1/4$, which implies that $\{(w_n,\nabla v_n)\}$ converges in $Y^\sigma\times Y^\sigma$. Hence, by \eqref{eq:Pi1-mapping-Ysigma} and Lemma \ref{lem:difference-estimate}, there exists a unique $(w,v)\in Y^\sigma\times (Y^\sigma/\mathbb{R})$ satisfying
\[  w=\Pi_f^1(w,v)\quad \text{and}\quad v= e^{z|\nabla|}g +\Pi_f^2(w,v). \] 
Moreover, it follows from \eqref{eq:Pi1-mapping-Ysigma} that 
\begin{align*}
\norm{w}{{Y}^\sigma}&\leq C\varepsilon_0(\norm{w}{{Y}^\sigma}+\norm{\nabla v}{{Y}^\sigma}),\\
\norm{\nabla v}{{Y}^\sigma}&\leq C_0 \norm{g}{\tLeb{\infty}^t\bes{\sigma}{p,q}}+C\varepsilon_0(\norm{w}{{Y}^\sigma}+\norm{\nabla v}{Y^\sigma})+\norm{w}{Y^\sigma}\\
&\leq C_0 \norm{g}{\tLeb{\infty}^t\bes{\sigma}{p,q}}+2C\varepsilon_0(\norm{w}{{Y}^\sigma}+\norm{\nabla v}{{Y}^\sigma})
\end{align*}
for some constant $C_0>0$.  Since $C\varepsilon_0\leq 1/4$, it follows that
\begin{equation}\label{eq:w-v-sum-estimate}
\norm{w}{Y^\sigma}+\norm{\nabla v}{Y^\sigma}\leq 4C_0 \norm{g}{\tLeb{\infty}^t\bes{\sigma}{p,q}}.
\end{equation}
Similarly, it follows from \eqref{eq:Pi1-mapping-Ydotsigma} and \eqref{eq:g-lifting-homogeneous} that $(w,v)$ satisfies
\begin{align*}
\norm{w}{\dot{Y}^{\sigma_0}}&\leq C\varepsilon_0(\norm{w}{\dot{Y}^{\sigma_0}}+\norm{\nabla v}{\dot{Y}^{\sigma_0}}),\\
\norm{\nabla v}{\dot{Y}^{\sigma_0}}&\leq C_0(\norm{g}{\tLeb{\infty}^t\hbes{d/p+1}{p,1}}+\norm{g}{\tLeb{\infty}^t\hbes{\sigma_0}{p,q}})+C\varepsilon_0(\norm{w}{\dot{Y}^{\sigma_0}}+\norm{\nabla v}{\dot{Y}^{\sigma_0}})+\norm{w}{\dot{Y}^{\sigma_0}}\\
&\leq C_0(\norm{g}{\tLeb{\infty}^t\hbes{d/p+1}{p,1}}+\norm{g}{\tLeb{\infty}^t\hbes{\sigma_0}{p,q}})+2C\varepsilon_0(\norm{w}{\dot{Y}^{\sigma_0}}+\norm{\nabla v}{\dot{Y}^{\sigma_0}}).
\end{align*}

Hence, by \eqref{eq:Pi1-mapping-Ysigma} and \eqref{eq:w-v-sum-estimate}, we get 
\begin{equation}\label{eq:w-estimate-fixed-point} 
\begin{aligned}
 \norm{w}{Y^\sigma}&=\norm{\Pi_f^1(w,v)}{Y^\sigma}\leq \mathcal{F}(\norm{f}{\tLeb{\infty}^t\bes{s}{p,q}})\norm{f}{\tLeb{\infty}^t\bes{s}{p,q}}\norm{g}{\tLeb{\infty}^t\bes{\sigma}{p,q}}\\
 \norm{\nabla v}{Y^\sigma}&\apprle \norm{g}{\tLeb{\infty}^t\bes{\sigma}{p,q}}+ \mathcal{F}(\norm{f}{\tLeb{\infty}^t\bes{s}{p,q}})\norm{f}{\tLeb{\infty}^t\bes{s}{p,q}}\norm{g}{\tLeb{\infty}^t\bes{\sigma}{p,q}}\\
 \norm{w}{\dot{Y}^{\sigma_0}}&=\norm{\Pi_f^1(w,v)}{\dot{Y}^{\sigma_0}}\leq \mathcal{F}(\norm{f}{\tLeb{\infty}^t\bes{s}{p,q}})\norm{f}{\tLeb{\infty}^t\bes{s}{p,q}}(\norm{g}{\tLeb{\infty}^t\hbes{d/p+1}{p,1}}+\norm{g}{\tLeb{\infty}^t\hbes{\sigma_0}{p,q}})\\
 \norm{\nabla v}{\dot{Y}^{\sigma_0}}&\apprle \norm{g}{\tLeb{\infty}^t\hbes{d/p+1}{p,1}}+\norm{g}{\tLeb{\infty}^t\hbes{\sigma_0}{p,q}}\\
 &\relphantom{=}+ \mathcal{F}(\norm{f}{\tLeb{\infty}^t\bes{s}{p,q}})\norm{f}{\tLeb{\infty}^t\bes{s}{p,q}}(\norm{g}{\tLeb{\infty}^t\hbes{d/p+1}{p,1}}+\norm{g}{\tLeb{\infty}^t\hbes{\sigma_0}{p,q}}).
 \end{aligned}
\end{equation}

Since $w,\nabla v\in Y^\sigma$, it follows from real interpolation theorem that $w\in \tLeb{q}^z\tLeb{\infty}^t\bes{\sigma-1+1/q}{p,q}$. See Theorem \ref{thm:CL-interpolation}.
Also, it follows from Lemma \ref{lem:Qa-Qb-control} that 
\[ \partial_z w = -|\nabla| w -|\nabla|Q_a+\Div_x Q_b \in \tLeb{q}^z\tLeb{\infty}^t([-D,0]\times J;\bes{\sigma-2+1/q}{p,q}).\]

Since $w\in \Leb{q}^z([-D,0];\tLeb{\infty}^t(J;\bes{\sigma-1+1/q}{p,q})) \cap \Sob{1}{q}([-D,0];\tLeb{\infty}^t(J;\bes{\sigma-2+1/q}{p,q}))$, it follows from Theorems \ref{thm:Lions-Peetre} and \ref{thm:CL-interpolation} that
\[ w \in C([-D,0];\tLeb{\infty}^t(J;\bes{\sigma-1}{p,q})).\]
Since $\sigma>1$, it follows from \eqref{eq:w-estimate-fixed-point} that
\begin{equation*}
G(f)g=|\nabla|g+w|_{z=0},
\end{equation*}
and 
\begin{equation*} 
\begin{aligned}
 \norm{R(f;g)}{\tLeb{\infty}^t\bes{\sigma-1}{p,q}}&\leq \mathcal{F}(\norm{f}{\tLeb{\infty}^t\bes{s}{p,q}})\norm{f}{\tLeb{\infty}^t\bes{s}{p,q}}\norm{g}{\tLeb{\infty}^t\bes{\sigma}{p,q}},\\
 \norm{R(f;g)}{\tLeb{\infty}^t\hbes{\sigma_0-1}{p,q}}&\leq \mathcal{F}(\norm{f}{\tLeb{\infty}^t\bes{s}{p,q}})\norm{f}{\tLeb{\infty}^t\bes{s}{p,q}}(\norm{g}{\tLeb{\infty}^t\hbes{d/p+1}{p,1}}+\norm{g}{\tLeb{\infty}^t\hbes{\sigma_0}{p,q}}).
\end{aligned}
\end{equation*}
This proves the first estimate in \eqref{eq:estimates-remainder-DN}.

\subsection{Contraction estimates}\label{subsec:contraction-DN}
Next, we show the contraction estimate \eqref{eq:contraction-DN}. For $f_1,f_2$ satisfying \eqref{eq:smallness-DN}, we note that by \eqref{eq:DN-remainder-expression}, we have
\[ G(f_1)g-G(f_2)g=w_1|_{z=0}-w_2|_{z=0},\]
where $(w_j,v_j)$ are the unique fixed point of 
$$
\left\{
\begin{aligned}
	w_j(x,t,z)&=\Pi_{f_j}^1(w_j,v_j)(x,t,z),\\
	v_j(x,t,z)&=e^{z|\nabla|}g(x,t)+\Pi_{f_j}^2(w_j,v_j)(x,t,z),\quad j=1,2.
\end{aligned}
\right.
$$

If we write $\delta w =w_1-w_2$ and $\delta v = v_1-v_2$, then 
\begin{align*}
\delta w(x,t,z)&=\int_{-\infty}^z e^{-(z-\tau)|\nabla|} (\Div_x \delta Q_b-|\nabla|\delta Q_a)(x,t,\tau)d\tau,\\
\delta v(x,t,z)&=-\int_z^0 e^{(z-\tau)|\nabla|} \{\delta w + \delta Q_a\}(x,t,\tau)d\tau,
\end{align*}
where 
\begin{align*}
\delta Q_a &= Q_a[w_1,v_1;f_1]-Q_a[w_2,v_2;f_2],\\
\delta Q_b&=Q_b[w_1,v_1;f_1]-Q_b[w_2,v_2;f_2].
\end{align*}

By \eqref{eq:Pif-estimates} and \eqref{eq:g-estimate}, we have 
\begin{align*}
\norm{\delta w}{Y^\sigma}&\apprle \norm{\delta Q_a}{\tilde{L}^z_1\tLeb{\infty}^t\hbes{1}{p,1}}+\norm{\delta Q_a}{\tilde{L}^z_1\tLeb{\infty}^t\hbes{\sigma}{p,q}}+\norm{\delta Q_a}{\tilde{L}^z_1\tLeb{\infty}^t\hbes{d/p+1}{p,1}}\\
&\relphantom{=}+\norm{\delta Q_b}{\tilde{L}^z_1\tLeb{\infty}^t\hbes{1}{p,1}}+\norm{\delta Q_b}{\tilde{L}^z_1\tLeb{\infty}^t\hbes{\sigma}{p,q}}+\norm{\delta Q_b}{\tilde{L}^z_1\tLeb{\infty}^t\hbes{d/p+1}{p,1}}\\
\norm{\nabla_x (\delta v)}{Y^\sigma}&\apprle \norm{\delta Q_a}{\tilde{L}^z_1\tLeb{\infty}^t\hbes{1}{p,1}}+\norm{\delta Q_a}{\tilde{L}^z_1\tLeb{\infty}^t\hbes{\sigma}{p,q}}+\norm{\delta Q_a}{\tilde{L}^z_1\tLeb{\infty}^t\hbes{d/p+1}{p,1}}\\
&\relphantom{=}+\norm{\delta w}{\tilde{L}^z_1\tLeb{\infty}^t\hbes{1}{p,1}}+\norm{\delta w}{\tilde{L}^z_1\tLeb{\infty}^t\hbes{\sigma}{p,1}}+\norm{\delta w}{\tilde{L}^z_1\tLeb{\infty}^t\hbes{d/p+1}{p,1}}.
\end{align*} 

Then by \eqref{eq:w-v-sum-estimate} and Lemma \ref{lem:difference-estimate}, we get 
\begin{equation}\label{eq:difference-estimate-3}
\begin{aligned}
\norm{\delta w}{{Y}^\sigma}+\norm{\nabla_x (\delta v)}{{Y}^\sigma}&\leq \mathcal{F}(\norm{f_1}{\tLeb{\infty}^t\bes{s}{p,q}},\norm{f_2}{\tLeb{\infty}^t\bes{s}{p,q}})\norm{\delta f}{\tLeb{\infty}^t\bes{s}{p,q}}\norm{g}{\tLeb{\infty}^t\bes{\sigma}{p,q}}\\
&\relphantom{=}+\mathcal{F}(\norm{f_1}{\tLeb{\infty}^t\bes{s}{p,q}})\norm{f_1}{\tLeb{\infty}^t\bes{s}{p,q}}(\norm{\delta w}{{Y}^\sigma}+\norm{\nabla_x(\delta v)}{{Y}^\sigma})\\
&\relphantom{=}+\mathcal{F}(\norm{f_2}{\tLeb{\infty}^t\bes{s}{p,q}})\norm{f_2}{\tLeb{\infty}^t\bes{s}{p,q}}(\norm{\delta w}{{Y}^\sigma}+\norm{\nabla_x(\delta v)}{{Y}^\sigma}).
\end{aligned}
\end{equation}
Hence, it follows from  \eqref{eq:difference-estimate-3} that if we choose $\varepsilon_0>0$ sufficiently small, then we have
\begin{equation*}\label{eq:contraction-estimate-in-the-proof}
\norm{\delta w}{{Y}^\sigma}\leq \mathcal{F}(\norm{f_1}{\tLeb{\infty}^t\bes{s}{p,q}},\norm{f_2}{\tLeb{\infty}^t\bes{s}{p,q}})\norm{\delta f}{\tLeb{\infty}^t\bes{s}{p,q}}\norm{g}{\tLeb{\infty}^t\bes{\sigma}{p,q}}.
\end{equation*}
This completes the proof of Theorem \ref{thm:first-order-expansion}. \hfill \qed

\section{Global well-posedness of the one-phase Muskat problem}\label{sec:GWP}

In this section, we prove the main theorems, Theorems \ref{thm:A} and \ref{thm:B}. Recall that the Muskat problem can be rewritten as 
\begin{equation}\label{eq:Muskat-reform-revisit}
 \partial_t f + \frac{\kappa}{\mu}G(f)(\rho\mathfrak{g}f+\mathfrak{s}H(f))=0.
\end{equation}
We only prove the case $\mathfrak{s}>0$. The case for $\mathfrak{s}=0$ is similar and in fact, is simpler. Since we are not interested in the dependence of the parameter $\kappa$, $\mu$, $\rho$, $\mathfrak{g}$, and $\mathfrak{s}$, from now on, we assume that those parameters are $1$.

To show the existence of the solution, we introduce the Fourier truncation operators $\mathcal{S}_R$ ($R>0$) defined by 
\[ \widehat{\mathcal{S}_Rf}(\xi)=\chi_{B_R}(\xi)\hat{f}(\xi),\quad f\in \Leb{2},\]
where $\chi_{B_R}$ is the characteristic function of $B_R$. The following properties hold for these truncation operators of which proof is omitted.
\begin{lemma}\label{lem:truncation-Fourier}
Let $0\leq s<t$.
\begin{enumerate}[label=\textnormal{(\roman*)}]
\item If $f\in \Leb{2}$, then $\mathcal{S}_Rf\in \tSob{s}$ and $\norm{\mathcal{S}_R f}{\tSob{s}}\apprle (1+R)^s\norm{f}{\Leb{2}}$.
\item If $f\in \tSob{s}$, then $\mathcal{S}_Rf \rightarrow f$ in $\tSob{s}$ as $R\rightarrow\infty$.
\item If $f\in \tSob{t}$, then $\norm{\mathcal{S}_Rf-f}{\tSob{s}}\apprle (1+R)^{-(t-s)}\norm{f}{\tSob{t}}$.
\end{enumerate}
\end{lemma}

For $R>0$, we denote by $V_R$ the space of all functions $f\in\Leb{2}$ such that $\supp\hat{f}\subset\overline{B_R}$. Note that $V_R$ is a closed subspace of $\Leb{2}$. Since $\mathcal{S}_Rf=f$ for $f\in V_R$, it follows from Lemma \ref{lem:truncation-Fourier} that $V_R$ is continuously embedded into $\tSob{s}$ for any $s\geq 0$. 

Let $0<R<\infty$ be fixed. For $f\in \Leb{2}$, we define 
\[ \mathcal{G}_R(f)=-\mathcal{S}_R[G(\mathcal{S}_R f)(\mathcal{S}_R f+H(\mathcal{S}_R f))].\]
By Lemma \ref{lem:truncation-Fourier} and Proposition \ref{prop:continuity-DN}, the operator $\mathcal{G}_R$ is a well defined mapping from $\Leb{2}$ to $V_R$.  

 Let us consider the following Cauchy problem on an ODE in the Hilbert space $V_R$:
\begin{equation}\label{eq:ODE-hilbert}
\left\{
\begin{aligned}
\partial_t f_R(t)&=\mathcal{G}_R(f_R(t))\\
f_R(0)&=\mathcal{S}_Rf_0.
\end{aligned}
\right.
\end{equation}

If we write 
\[ H_1(p)=(1+|p|^2)^{-1/2}-1,\]
then 
\begin{equation}\label{eq:mean-curvature-rewritten}
H(f)=-\Delta f -\Div(\nabla f H_1(\nabla f)).
\end{equation}
Thus, 
\begin{align*}
&\relphantom{=}H(\mathcal{S}_R f^1)-H(\mathcal{S}_R f^2)\\
&=-\Delta (\mathcal{S}_R f^1-\mathcal{S}_R f^2)-\Div(\nabla \mathcal{S}_R f^1 H_1(\nabla \mathcal{S}_R f^1))+\Div(\nabla \mathcal{S}_R f^2 H_1(\nabla \mathcal{S}_R f^2))\\
&=-\Delta \mathcal{S}_R \delta f-\Div(\nabla \mathcal{S}_R(\delta f) H_1(\nabla \mathcal{S}_R f^1))-\Div(\nabla \mathcal{S}_R f^2 (H_1(\nabla \mathcal{S}_R f^1)-H_1(\nabla \mathcal{S}_R f^2)).
\end{align*}
For $s>d/2+1$, it follows that
\begin{align*}
&\relphantom{=}\norm{H(\mathcal{S}_R f^1)-H(\mathcal{S}_R f^2)}{\tSob{s}}\\
&\apprle \norm{\mathcal{S}_R\delta f}{\tSob{s+2}}+\norm{\nabla \mathcal{S}_R (\delta f)}{\tSob{s+1}}\norm{H_1(\nabla \mathcal{S}_R (\delta f))}{\tSob{s+1}}\\
&\relphantom{=}+\norm{\nabla \mathcal{S}_R f^2}{\tSob{s+1}}\norm{H_1(\nabla \mathcal{S}_R f^1)-H_1(\nabla \mathcal{S}_R f^2)}{\tSob{s+1}}.
\intertext{Since $H_1(0)=0$ and $\nabla H_1(0)=0$, by Theorem \ref{thm:Moser-estimates} and Lemma \ref{lem:truncation-Fourier}, the right-hand side becomes }
&\apprle_R \mathcal{F}(\norm{f^1}{\Leb{2}},\norm{f^2}{\Leb{2}})\norm{\delta f}{\Leb{2}}.
\end{align*}

Hence, it follows from Proposition \ref{prop:continuity-DN} that 
\begin{align*}
\norm{\mathcal{G}_R(f^1)-\mathcal{G}_R(f^2)}{\Leb{2}}&\apprle\norm{G(\mathcal{S}_Rf^1)[\mathcal{S}_R \delta f+H(\mathcal{S}_Rf^1)-H(\mathcal{S}_Rf^2)]}{\Leb{2}}\\
&\relphantom{=}+\norm{[G(\mathcal{S}_R f^1)-G(\mathcal{S}_R f^2)](\mathcal{S}_R f^2+H(\mathcal{S}_R f^2))}{\Leb{2}}\\
&\apprle_R \mathcal{F}(\norm{\mathcal{S}_Rf^1}{\tSob{s}})\norm{\mathcal{S}_Rf^1}{\tSob{s}}(\norm{\mathcal{S}_R\delta f}{\tSob{s}}+\norm{H(\mathcal{S}_R f^1)-H(\mathcal{S}_R f^2)}{\tSob{s}})\\
&\relphantom{=}+\mathcal{F}(\norm{\mathcal{S}_Rf^1}{\tSob{s}},\norm{\mathcal{S}_Rf^2}{\tSob{s}})\norm{\mathcal{S}_Rf^2+H(\mathcal{S}_R f^2)}{\tSob{s}}\norm{\mathcal{S}_R\delta f}{\tSob{s}}\\
&\apprle_R \mathcal{F}(\norm{f^1}{\Leb{2}},\norm{f^2}{\Leb{2}})\norm{\delta f}{\Leb{2}}.
\end{align*}
This implies that $\mathcal{G}_R$ is locally Lipschitz on $V_R$. Therefore, it follows from the Picard theorem for ODEs in infinite-dimensional spaces that for every $f_0\in \tSob{s}$, $s>d/2+1$, the $R$-truncated problem \eqref{eq:ODE-hilbert} has a unique local classical solution in $C^1([0,T_R);V_R)$. Here $0<T_R\leq \infty$ denotes the maximal existence time of the local solution $f_R$. 

Since $\mathcal{S}_R$ is self-adjoint, by Theorem \ref{thm:Lyapunov}, we have 
\begin{equation}\label{eq:Lyapunov-R-estimate}
 \frac{d}{dt}\frac{1}{2}\norm{f_R}{\Leb{2}}^2+\frac{1}{\mathcal{F}(\norm{f_R}{\tSob{s}})}\left(\norm{f_R}{\thSob{1/2}}^2+\norm{f_R}{\thSob{3/2}}^2\right)\leq 0.
\end{equation}
This implies that 
\begin{equation}\label{eq:L2-norm-Lyapunov-R}
 \norm{f_R(t)}{\Leb{2}}\leq \norm{f_R(0)}{\Leb{2}}\leq \norm{f_0}{\Leb{2}},
 \end{equation}
which shows $T_R=\infty$. 
 
\subsection{A priori estimates} 
Since $V_R \hookrightarrow \tSob{s}$ for any $s \geq 0$, it follows from Bernstein's inequality that $f\in \tLeb{\infty}(0,\infty;\tSob{s})$. We show that if $\norm{f_0}{\tSob{s}}$ is sufficiently small, then $\{f_R\}$ is uniformly bounded in $\tilde{L}_\infty(0,\infty;\tSob{s})$ and $\Leb{2}(0,\infty;\thSob{s+3/2})$.

\begin{proposition}\label{prop:a-priori-estimate}
Let $s>d/2+3$ and $R>0$. There exists $\varepsilon_0=\varepsilon_0(s,d)>0$ such that if $\norm{f_0}{\tSob{s}}<\varepsilon_0$, then any solution $f_R\in C^1([0,\infty);V_R)$ to \eqref{eq:ODE-hilbert} satisfies 
\begin{equation}\label{eq:a-priori-Ltilde-estimate}
 \norm{f_R}{\tilde{L}_\infty(0,\infty;\tSob{s})}\leq 2\norm{f_0}{\tSob{s}}
 \end{equation}
and
\begin{equation}\label{eq:a-priori-global-estimate}
 \norm{f_R(t)}{\tSob{s}}^2+\int_0^t \norm{f_R(\tau)}{\thSob{s+3/2}}^2d\tau\leq \norm{f_0}{\tSob{s}}^2
\end{equation}
for all $t\geq 0$.
\end{proposition}

\begin{proof}
Note that $f_R \in C^1([0,\infty);V_R)$ satisfies 
\begin{equation}
 \partial_t f_R + \mathcal{S}_RG(f_R)(f_R+H(f_R))=0,\quad f_R(0)=\mathcal{S}_Rf_0.
\end{equation}

Define 
\[ T^*= \sup \{ T>0 : \norm{f_R}{\tilde{L}_\infty([0,T];\tSob{s})}\leq 2\norm{f_0}{\tSob{s}}\}.\]
We claim that there exists $\varepsilon_0>0$ such that if $\norm{f_0}{\tSob{s}}<\varepsilon_0$, then $T^*=\infty$. We first choose $\varepsilon_0>0$ so small that we can apply Theorem \ref{thm:first-order-expansion} and \eqref{eq:mean-curvature-rewritten} to get 
\begin{equation}\label{eq:fR-equantion-linearlization}
\partial_t f_R +|\nabla|(f_R+|\nabla|^2 f_R)=\mathcal{S}_R\tilde{\mathcal{R}}(f_R),
\end{equation}
where 
\begin{equation}\label{eq:tilde-R-expression}
\tilde{\mathcal{R}}(f_R)=|\nabla|\Div(\nabla f_R H_1(\nabla f_R))-\mathcal{R}(f_R;f_R+H(f_R)).
\end{equation}

Set $A=|\nabla|(1+|\nabla|^2)$. Then by Bernstein's inequality, we have
\begin{equation*}
 \int_{\mathbb{R}^d} A P_j f_R \cdot P_j f_R \myd{x}\geq c_1 (2^j+2^{3j})\norm{P_j f_R}{\Leb{2}}^2
\end{equation*}
for some constant $c_1>0$. Hence, it follows from \eqref{eq:fR-equantion-linearlization} that
\[ \frac{d}{dt} \frac{1}{2} \norm{P_j f_R}{\Leb{2}}^2+c_1(2^j +2^{3j})\norm{P_j f_R}{\Leb{2}}^2\leq \norm{P_j\tilde{\mathcal{R}}(f_R)}{\Leb{2}}\norm{P_j f_R}{\Leb{2}}.\]
Then by Gronwall's inequality, we get 
\begin{equation*}
 \norm{P_j f_R(t,\cdot)}{\Leb{2}}\leq e^{-c_1(2^j+2^{3j})t}\norm{P_j f_R(0)}{\Leb{2}}+\int_0^t \exp(-c_1(2^j+2^{3j})(t-\tau))\norm{P_j\tilde{\mathcal{R}}(f_R)(\tau)}{\Leb{2}}d\tau.
\end{equation*}
By taking supremum over $t$, we get 
\[ \norm{P_j f_R}{\Leb{\infty}([0,T];\Leb{2})}\leq \norm{P_j f_R(0)}{\Leb{2}}+\frac{C}{2^j+2^{3j}}\norm{P_j\tilde{\mathcal{R}}(f_R)}{\Leb{\infty}([0,T];\Leb{2})}\]
for some constant $C=C(d)>0$, which implies that 
\begin{equation}\label{eq:fR-tilde-estimate}
\norm{f_R}{\tLeb{\infty}\thSob{s}}\leq \norm{f_R(0)}{\thSob{s}}+C\norm{\tilde{\mathcal{R}}(f_R)}{\tLeb{\infty}^t\thSob{s-3}}.
\end{equation}

To estimate $\tilde{\mathcal{R}}(f_R)$, for the first part, it follows from Theorems \ref{thm:Besov-product} and \ref{thm:Moser-estimates} with $p=q=2$ that
\begin{equation}\label{eq:remainder-first-estimate}
\begin{aligned}
&\relphantom{=}\norm{|\nabla|\Div(\nabla f_R H_1(\nabla f_R))}{\tLeb{\infty}^t\thSob{s-3}}\\
&\leq \norm{\nabla f_R H_1(\nabla f_R)}{\tLeb{\infty}^t\thSob{s-1}}\\
&\apprle \norm{\nabla f_R}{\tLeb{\infty}^t\thSob{s-1}}\norm{H_1(\nabla f_R)}{\tLeb{\infty}^t\Leb{\infty}}+\norm{H_1(\nabla f_R)}{\tLeb{\infty}^t\thSob{s-1}}\norm{\nabla f_R}{\Leb{\infty}^{t,x}}\\
&\leq \mathcal{F}(\norm{ f_R}{\tLeb{\infty}^t\tSob{s}})\norm{ f_R}{\tLeb{\infty}^t\tSob{s}}\norm{\nabla f_R}{\tLeb{\infty}^t\thSob{s-1}\cap\tLeb{\infty}^t\hbes{d/2}{2,1}}.
\end{aligned}
\end{equation}
To estimate the second part, it follows from  Theorem \ref{thm:first-order-expansion} with $p=q=2$ with $\sigma_0=s-2$ that
\begin{equation}\label{eq:remainder-estimate-fR}
\norm{R(f_R;f_R+H(f_R))}{\tLeb{\infty}^t\thSob{s-3}}\leq \mathcal{F}(\norm{f_R}{\tLeb{\infty}^t\tSob{s}})\norm{f_R}{\tLeb{\infty}^t\tSob{s}}\norm{f_R+H(f_R)}{\tLeb{\infty}^t\hbes{d/2+1}{2,1}\cap\thSob{s-2}}.
\end{equation}

Since 
\[ H(f_R)=-\Delta f_R +\Div(\nabla f_R H_1(\nabla f_R)),\]
it follows that
\begin{equation}\label{eq:fR-estimate-remainder-1}
\begin{aligned}
\norm{f_R+H(f_R)}{\tLeb{\infty}^t\hbes{d/2+1}{2,1}\cap\tLeb{\infty}^t\thSob{s-2}}&\apprle \norm{f_R}{\tLeb{\infty}^t\hbes{d/2+1}{2,1}\cap\tLeb{\infty}^t\thSob{s-2}}+\norm{f_R}{\tLeb{\infty}^t\hbes{d/2+3}{2,1}\cap\tLeb{\infty}^t\thSob{s}}\\
&\relphantom{=}+\norm{\nabla f_R H_1(\nabla f_R)}{\tLeb{\infty}^t\hbes{d/2+2}{2,1}\cap\tLeb{\infty}^t\thSob{s-1}}.
\end{aligned}
\end{equation}
By Theorems \ref{thm:Besov-product} and \ref{thm:Moser-estimates} again, the last quantity is bounded by 
\begin{equation}\label{eq:fR-estimate-remainder-2}
\begin{aligned}
&\relphantom{=}\norm{\nabla f_R H_1(\nabla f_R)}{\tLeb{\infty}^t\hbes{d/2+2}{2,1}\cap\tLeb{\infty}^t\thSob{s-1}}\\
&\apprle \norm{\nabla f_R}{\Leb{\infty}^{t,x}}\norm{H_1(\nabla f_R)}{\tLeb{\infty}^t\hbes{d/2+2}{2,1}}+ \norm{\nabla f_R}{\tLeb{\infty}^t\hbes{d/2+2}{2,1}}\norm{H_1(\nabla f_R)}{\Leb{\infty}^{t,x}}\\
&\relphantom{=}+ \norm{\nabla f_R}{\Leb{\infty}^{t,x}}\norm{H_1(\nabla f_R)}{\tLeb{\infty}^t\thSob{s-1}}+ \norm{\nabla f_R}{\tLeb{\infty}^t\thSob{s-1}}\norm{H_1(\nabla f_R)}{\Leb{\infty}^{t,x}}\\
&\leq \mathcal{F}(\norm{f_R}{\tLeb{\infty}^t\tSob{s}})(\norm{f_R}{\tLeb{\infty}^t\hbes{d/2+3}{2,1}}+\norm{ f_R}{\tLeb{\infty}^t\thSob{s}}).
\end{aligned}
\end{equation}

Since $s>d/2+3$, it follows from the embedding $\tSob{s}\rightarrow B^{d/2+3}_{2,1}$, \eqref{eq:remainder-first-estimate}, \eqref{eq:remainder-estimate-fR},  \eqref{eq:fR-estimate-remainder-1}, and \eqref{eq:fR-estimate-remainder-2} that
\begin{equation}\label{eq:final-remainder-estimate}
\norm{\tilde{\mathcal{R}}(f_R)}{\tLeb{\infty}^t\thSob{s-3}}\leq \mathcal{F}(\norm{f_R}{\tLeb{\infty}^t\tSob{s}})\norm{f_R}{\tLeb{\infty}^t\tSob{s}}^2. 
\end{equation}
Therefore, by  \eqref{eq:L2-norm-Lyapunov-R}, \eqref{eq:fR-tilde-estimate}, and \eqref{eq:final-remainder-estimate}, we have
\begin{align*}
\norm{f_R}{\tilde{L}_\infty^t([0,T];\tSob{s})}&\leq \norm{f_R(0)}{\tSob{s}}+C\norm{\tilde{\mathcal{R}}(f_R)}{\tilde{L}_\infty^t([0,T];\tSob{s-3})}\\
&\leq \norm{f_0}{\tSob{s}}+\mathcal{F}(\norm{f_R}{\tilde{L}_\infty^t([0,T];\tSob{s})})\norm{f_R}{\tilde{L}_\infty^t([0,T];\tSob{s})}^2.
\end{align*} 

Since $g(T)=\norm{f_R}{\tilde{L}_\infty([0,T];\tSob{s})}$ is continuous, a standard bootstrap argument shows that there exists $\varepsilon_0>0$ such that if $\norm{f_0}{\tSob{s}}\leq \varepsilon_0$, then $T^*=\infty$ and 
\begin{equation}\label{eq:uniform-estimate-fR}
 \norm{f_R}{\tilde{L}_\infty([0,\infty);\tSob{s})}\leq 2\norm{f_0}{\tSob{s}}
 \end{equation}	
 for all $R>0$.

Take derivatives $|\nabla|^s$ to the equation \eqref{eq:ODE-hilbert}, multiply $|\nabla|^s f_R$, and integrate it over $\mathbb{R}^d$. Then we have
\begin{align*}
\frac{1}{2}\frac{d}{dt}\norm{f_R}{\thSob{s}}^2+\norm{f_R}{\thSob{s+1/2}}^2+\norm{f_R}{\thSob{s+3/2}}^2=\int_{\mathbb{R}^d} |\nabla|^{s-3/2} \tilde{\mathcal{R}}(f_R) |\nabla|^{s+3/2} f_R \myd{x}.
\end{align*}
Similar to \eqref{eq:remainder-first-estimate} and \eqref{eq:remainder-estimate-fR}, for fixed $t$, we have
\begin{equation}\label{eq:remainder-estimate-H-s-3/2}
\begin{aligned}
&\relphantom{=}\norm{\tilde{\mathcal{R}}(f_R)(t)}{\thSob{s-3/2}}\\
&\leq \mathcal{F}(\norm{ f_R(t)}{\tSob{s}})\norm{ f_R(t)}{\tSob{s}}(\norm{\nabla f_R(t)}{\hbes{d/2}{2,1}\cap\thSob{s-1/2}}+\norm{f_R(t)+H(f_R)(t)}{\hbes{d/2+1}{2,1}\cap\thSob{s-1/2}}).
\end{aligned}
\end{equation}
By a similar argument as in \eqref{eq:fR-estimate-remainder-1} and \eqref{eq:fR-estimate-remainder-2}, we get 
\begin{equation}\label{eq:fR-estimate-remainder-3}
\begin{aligned}
&\relphantom{=}\norm{f_R(t)+H(f_R)(t)}{\hbes{d/2+1}{2,1}\cap\thSob{s-1/2}}\\
&\apprle \norm{f_R(t)}{\hbes{d/2+1}{2,1}\cap \hbes{d/2+3}{2,1}}+\norm{f_R(t)}{\thSob{s-1/2}}+\norm{f_R(t)}{\thSob{s+3/2}}\\
&\relphantom{=}+\mathcal{F}(\norm{\nabla f_R(t)}{\Leb{\infty}})(\norm{f_R(t)}{\hbes{d/2+3}{2,1}}+\norm{ f_R(t)}{\thSob{s+3/2}}).
\end{aligned}
\end{equation}
Hence, by \eqref{eq:remainder-estimate-H-s-3/2} and \eqref{eq:fR-estimate-remainder-3}, we get the following differential inequality
\begin{equation}\label{eq:final-inequality-1}
\begin{aligned}
&\frac{d}{dt} \frac{1}{2}\norm{f_R}{\thSob{s}}^2+\norm{f_R}{\thSob{s+1/2}}^2+\norm{f_R}{\thSob{s+3/2}}^2\\
&\leq \mathcal{F}(\norm{f_R}{\tLeb{\infty}^t\tSob{s}})\norm{ f_R}{\tLeb{\infty}^t\tSob{s}}\norm{f_R}{\thSob{s+3/2}}^2\\
&\relphantom{=}+\mathcal{F}(\norm{ f_R}{\tLeb{\infty}^t\tSob{s}})\norm{ f_R}{\tLeb{\infty}^t\tSob{s}}(\norm{f_R}{\hbes{d/2+1}{2,1}\cap \hbes{d/2+3}{2,1}}+\norm{f_R}{\thSob{s-1/2}}+\norm{f_R}{\thSob{s+3/2}})\norm{f_R}{\thSob{s+3/2}}.
\end{aligned}
\end{equation}
By \eqref{eq:Lyapunov-R-estimate} and \eqref{eq:final-inequality-1}, we get 
\begin{equation}\label{eq:final-inequality-2}
\begin{aligned}
&\relphantom{=}\frac{d}{dt} \frac{1}{2}\norm{f_R}{\tSob{s}}^2+\frac{1}{\mathcal{F}(\norm{f_R}{\tSob{s}})}(\norm{f_R}{\thSob{1/2}}^2+\norm{f_R}{\thSob{3/2}}^2)+\norm{f_R}{\thSob{s+1/2}}^2+\norm{f_R}{\thSob{s+3/2}}^2\\
&\leq \mathcal{F}(\norm{ f_R}{\tLeb{\infty}^t\tSob{s}})\norm{ f_R}{\tLeb{\infty}^t\tSob{s}}\norm{f_R}{\thSob{s+3/2}}^2\\
&\relphantom{=}+\mathcal{F}(\norm{ f_R}{\tLeb{\infty}^t\tSob{s}})\norm{ f_R}{\tLeb{\infty}^t\tSob{s}}(\norm{f_R}{\hbes{d/2+1}{2,1}\cap \hbes{d/2+3}{2,1}}+\norm{f_R}{\thSob{s-1/2}}+\norm{f_R}{\thSob{s+3/2}})\norm{f_R}{\thSob{s+3/2}}.
\end{aligned}
\end{equation}

Since $(\thSob{1/2},\thSob{s+3/2})_{\theta,1}=\hbes{\frac{1-\theta}{2}+\left(s+\frac{3}{2}\right)\theta}{2,1}$, $0<\theta<1$, it follows from \eqref{eq:final-inequality-2} and Young's inequality that
\begin{align*}
&\relphantom{=}\frac{d}{dt} \frac{1}{2}\norm{f_R}{\tSob{s}}^2+\frac{1}{\mathcal{F}(\norm{f_R}{\tSob{s}})}\norm{f_R}{\thSob{1/2}}^2+\norm{f_R}{\thSob{s+3/2}}^2\\
&\leq \mathcal{F}(\norm{ f_R}{\tLeb{\infty}^t\tSob{s}})\norm{ f_R}{\tLeb{\infty}^t\tSob{s}}(\norm{f_R}{\thSob{s+3/2}}^2+\norm{f_R}{\thSob{1/2}}^2).
\end{align*}
Since $\mathcal{F}$ is nondecreasing, choose $\varepsilon_0>0$ so that 
\[ (\mathcal{F}(2\varepsilon_0)^2+\mathcal{F}(2\varepsilon_0))(2\varepsilon_0) \leq \frac{1}{2}.\]
Hence, by \eqref{eq:a-priori-Ltilde-estimate} and \eqref{eq:uniform-estimate-fR}, we have  
\begin{equation}\label{eq:half-control}
 \norm{f_R(t)}{\tSob{s}}^2+c\int_0^t \norm{f_R(\tau)}{\thSob{1/2}}^2d\tau+\int_0^t \norm{f_R(\tau)}{\thSob{s+3/2}}^2 d\tau \leq \norm{f_0}{\tSob{s}}^2.
\end{equation}
This completes the proof of Proposition \ref{prop:a-priori-estimate}.
\end{proof}

\subsection{Contraction estimates}
We show that the solution has continuous dependence on the initial data. 
\begin{proposition}\label{prop:contraction-solution}
Let $s>d/2+3$. Suppose that $f_i \in C([0,\infty);\tSob{s})\cap \tilde{L}_\infty(0,\infty;\tSob{s})$, $i=1,2$, are solutions of \eqref{eq:DtoN-formulation-surface-tension}. There exists $\varepsilon_0>0$ such that if 
\[ \norm{f_1}{\tilde{L}_\infty(0,\infty;\tSob{s})},\norm{f_2}{\tilde{L}_\infty(0,\infty;\tSob{s})}\leq 2\varepsilon_0,\]
then for any $T>0$, we have
\begin{equation*}
 \norm{f_1-f_2}{\tilde{L}_\infty(0,T;\tSob{s})}\leq C_0(s,d,T) \norm{(f_1-f_2)|_{t=0}}{\tSob{s}}.
\end{equation*}
\end{proposition}
\begin{proof}
Recall that $A=|\nabla|(1+|\nabla|^2)$. Write $\tilde{f}=f_1-f_2$. Then by Theorem \ref{thm:first-order-expansion}, we get 
\begin{equation}\label{eq:difference-eqn}
\partial_t \tilde{f}+A\tilde{f}=\mathfrak{R}(f_1,f_2),
\end{equation}
where 
\begin{align*}
\mathfrak{R}(f_1,f_2)&=\tilde{\mathcal{R}}(f_2)-\tilde{\mathcal{R}}(f_1)\\
&=|\nabla|\Div(\nabla f_1 H_1(\nabla f_1)-\nabla f_2H_1(\nabla f_2))\\
&\relphantom{=}-\mathcal{R}(f_1;f_1+H(f_1))+\mathcal{R}(f_2;f_2+H(f_2))\\
&=:\mathrm{I}+\mathrm{II}.
\end{align*}

Rewrite $\mathrm{I}$ into
\[
\mathrm{I}=|\nabla|\Div(\nabla\tilde{f} H_1(\nabla f_1)+\nabla f_2 (H_1(\nabla f_1)-H_1(\nabla f_2))).
\]
Then it follows from Proposition \ref{prop:a-priori-estimate}, the Sobolev embedding theorem, and Theorem \ref{thm:Moser-estimates} that 
\begin{equation}\label{eq:estimate-I-part1}
\begin{aligned}
&\relphantom{=}\norm{\mathrm{I}}{\tilde{L}_\infty^t\tSob{s-3}}\\
&\apprle \norm{|\nabla|\Div(\nabla \tilde{f} H_1(\nabla f_1))}{\tilde{L}_\infty^t\tSob{s-3}}\\
&\relphantom{=}+\norm{|\nabla|\Div((\nabla f_2)(H_1(\nabla f_1)-H_1(\nabla f_2)))}{\tilde{L}_\infty^t\tSob{s-3}}\\
&\apprle \norm{\nabla \tilde{f} H_1(\nabla f_1)}{\tilde{L}_\infty^t\tSob{s-1}}+\norm{(\nabla f_2)(H_1(\nabla f_1)-H_1(\nabla f_2))}{\tilde{L}_\infty^t\tSob{s-1}}\\
&\apprle \norm{\tilde{f}}{\tilde{L}_\infty^t\tSob{s}}\norm{H_1(\nabla f_1)}{\tilde{L}_\infty^t\tSob{s-1}\cap {L}_\infty^{t,x}}+\norm{f_2}{\tilde{L}_\infty^t\tSob{s}}\norm{H_1(\nabla f_1)-H_1(\nabla f_2)}{\tilde{L}_\infty^t\tSob{s-1}\cap \Leb{\infty}^{t,x}}\\
&\leq \mathcal{F}(2\varepsilon_0)\varepsilon_0\norm{\tilde{f}}{\tilde{L}_\infty^t\tSob{s}}.
\end{aligned}
\end{equation}

To estimate $\mathrm{II}$, by Theorems \ref{thm:Besov-product} and \ref{thm:Moser-estimates}, we have
\begin{equation}\label{eq:mean-curvature-difference}
\begin{aligned}
&\relphantom{=}\norm*{H(f_1)-H(f_2)}{\tilde{L}_\infty^t\tSob{s-2}}\\
&\leq \norm{\Delta\tilde{f}}{\tilde{L}_\infty^t\tSob{s-2}}+\norm{\Div(\nabla\tilde{f} H_1(\nabla f_1))}{\tilde{L}_\infty^t\tSob{s-2}}+\norm{\Div[\nabla f_2(H_1(\nabla f_2)-H_1(\nabla f_1))]}{\tilde{L}_\infty^t\tSob{s-2}}\\
&\apprle \norm{\tilde{f}}{\tilde{L}_\infty^t\tSob{s}}+\mathcal{F}(\norm{f_1}{\tilde{L}_\infty^t\tSob{s}},\norm{f_2}{\tilde{L}_\infty^t\tSob{s}})(\norm{f_1}{\tilde{L}_\infty^t\tSob{s}}+\norm{f_2}{\tilde{L}_\infty^t\tSob{s}})\norm{\tilde{f}}{\tilde{L}_\infty^t\tSob{s}}.
\end{aligned}
\end{equation}

Hence, it follows from Proposition \ref{prop:a-priori-estimate}, \eqref{eq:contraction-DN}, and \eqref{eq:mean-curvature-difference} that 
\begin{align}\label{eq:estimate-II-part2}
\relphantom{=}\norm{\mathrm{II}}{\tilde{L}_\infty^t\tSob{s-3}}&\leq \norm{\mathcal{R}(f_1;f_1+H(f_1))-\mathcal{R}(f_2;f_1+H(f_1))}{\tilde{L}_\infty^t\tSob{s-3}}\\
&\relphantom{=}+ \norm{\mathcal{R}(f_2;f_1+H(f_1))-\mathcal{R}(f_2;f_2+H(f_2))}{\tilde{L}_\infty^t\tSob{s-3}}\nonumber\\
&\leq \mathcal{F}(\norm{f_1}{\tilde{L}_\infty^t\tSob{s}},\norm{f_2}{\tilde{L}_\infty^t\tSob{s}})\norm{\tilde{f}}{\tilde{L}_\infty^t\tSob{s}}\norm{f_1+H(f_1)}{\tilde{L}_\infty^t\tSob{s-2}}\nonumber\\
&\relphantom{=}+ \mathcal{F}(\norm{f_2}{\tilde{L}_\infty^t\tSob{s}})\norm{f_2}{\tilde{L}_\infty^t\tSob{s}}\norm{\tilde{f}+H(f_1)-H(f_2)}{\tilde{L}_\infty^t\tSob{s-2}}\nonumber\\
&\leq \mathcal{F}(\norm{f_1}{\tilde{L}_\infty^t\tSob{s}},\norm{f_2}{\tilde{L}_\infty^t\tSob{s}})\norm{f_1}{\tilde{L}_\infty^t\tSob{s}}\norm{\tilde{f}}{\tilde{L}_\infty^t\tSob{s}}\nonumber\\
&\relphantom{=}+\mathcal{F}(\norm{f_2}{\tilde{L}_\infty^t\tSob{s}})\norm{f_2}{\tilde{L}_\infty^t\tSob{s}}\norm{\tilde{f}}{\tilde{L}_\infty^t\tSob{s}}.\nonumber
\end{align}

Apply the Littlewood-Paley projection $P_j$ to \eqref{eq:difference-eqn} and we get 
\begin{equation*}
 \frac{d}{dt} \frac{1}{2}\norm{P_j\tilde{f}}{\Leb{2}}^2+c_1(2^j+2^{3j})\norm{P_j\tilde{f}}{\Leb{2}}^2\leq \norm{P_j\mathfrak{R}(f_1,f_2)}{\Leb{2}}\norm{P_j\tilde{f}}{\Leb{2}}. 
 \end{equation*}
 By Gronwall's inequality, we get  
 \begin{align*}
 \norm{P_j\tilde{f}(t,\cdot)}{\Leb{2}}&\leq e^{-c_1(2^j+2^{3j})t} \norm{P_j\tilde{f}(0,\cdot)}{\Leb{2}}+\int_0^t e^{-c_1(2^j+2^{3j})(t-\tau)} \norm{P_j\mathfrak{R}(f_1,f_2)(\tau)}{\Leb{2}}d\tau\\
 &\leq e^{-c_1(2^j+2^{3j})t} \norm{P_j\tilde{f}(0,\cdot)}{\Leb{2}}+\frac{C}{2^j+2^{3j}}\norm{P_j\mathfrak{R}(f_1,f_2)}{\Leb{\infty}([0,T];\Leb{2})},
  \end{align*}
 which implies that
\begin{equation}\label{eq:tilde-estimate-difference-1}
\norm{\tilde{f}}{\tilde{L}_\infty([0,T];\thSob{s})}\leq\norm{\tilde{f}(0)}{\thSob{s}}+C\norm{\mathfrak{R}(f_1,f_2)}{\tilde{L}_\infty^t\thSob{s-3}}
\end{equation}
for some constant $C>0$.

To estimate $\norm{\tilde{f}}{\Leb{\infty}([0,T];\Leb{2})}$, multiply the equation \eqref{eq:difference-eqn} by $\tilde{f}$ and take the integration over $\mathbb{R}^d$. Then 
\[ \frac{d}{dt} \frac{1}{2}\norm{\tilde{f}}{\Leb{2}}^2+\int_{\mathbb{R}^d} |A^{1/2} \tilde{f}|^2 \myd{x}=\int_{\mathbb{R}^d} \mathfrak{R}(f_1,f_2)\tilde{f}\myd{x},\]
which implies
\begin{equation}\label{eq:tilde-estimate-difference-2}
\begin{aligned}
\norm{\tilde{f}}{\Leb{\infty}([0,T];\Leb{2})}&\leq \norm{\tilde{f}|_{t=0}}{\Leb{2}}+\norm{\mathfrak{R}(f_1,f_2)}{\Leb{1}(0,T;\Leb{2})}\\
&\leq \norm{\tilde{f}|_{t=0}}{\Leb{2}}+T\norm{\mathfrak{R}(f_1,f_2)}{\tilde{L}_\infty^t\tSob{s-3}}.
\end{aligned}
\end{equation}
Then by \eqref{eq:estimate-I-part1} and \eqref{eq:estimate-II-part2}, we have 
  \begin{equation}\label{eq:nonlinear-difference}
   \norm{\mathfrak{R}(f_1,f_2)}{\tilde{L}_\infty([0,T];\tSob{s-3})}\leq \mathcal{F}(2\varepsilon_0)\varepsilon_0\norm{\tilde{f}}{\tilde{L}_\infty([0,T];\tSob{s})}.
   \end{equation}
Hence, it follows from \eqref{eq:tilde-estimate-difference-1} and \eqref{eq:tilde-estimate-difference-2} that
\begin{equation}\label{eq:difference-estimate-almost-final}
\norm{\tilde{f}}{\tilde{L}_\infty^t\tSob{s}}\leq  \norm{\tilde{f}|_{t=0}}{\tSob{s}}+(1+T)\mathcal{F}(2\varepsilon_0)\varepsilon_0\norm{\tilde{f}}{\tilde{L}_\infty^t\tSob{s}}.
\end{equation}

Choose $\varepsilon_0>0$ sufficiently small so that 
\[ 2\mathcal{F}(2\varepsilon_0)\varepsilon_0\leq \frac{1}{4}.\]

Now given $T>0$ and choose $n>T$. Define $t_j=jT/n$, $j=0,\dots,n$ and set $I_j=[t_{j-1},t_j]$. Then by \eqref{eq:difference-estimate-almost-final}, we have 
\begin{equation}\label{eq:time-splitting-argument}
\norm{\tilde{f}}{\tilde{L}_\infty(I_j;\tSob{s})}\leq \norm{\tilde{f}(t_{j-1})}{\tSob{s}}+(1+T/n)\mathcal{F}(2\varepsilon_0)\varepsilon_0\norm{\tilde{f}}{\tilde{L}_\infty(I_j;\tSob{s})}.
\end{equation}
This implies that 
\[ \norm{\tilde{f}}{\tilde{L}_\infty(I_j;\tSob{s})}\leq \frac{4}{3}\norm{\tilde{f}(t_{j-1})}{\tSob{s}},\quad j=0,\dots,n.\]
By iteration, we get 
\[ \norm{\tilde{f}}{\tilde{L}_\infty([0,T];\tSob{s})}\leq \sum_{k=1}^n\left(\frac{4}{3}\right)^k\norm{\tilde{f}|_{t=0}}{\tSob{s}} . \]
This completes the proof of Proposition \ref{prop:contraction-solution}.
\end{proof}

\subsection{Proof of Theorem \ref{thm:A}}
Now we are ready to prove Theorem \ref{thm:A}. We first impose the additional condition that $f_0 \in \tSob{s}$, $s>d/2+3$. Later, we prove Theorem \ref{thm:A} for general $f_0 \in \tSob{s}$, $s>d/2+1$.\bigskip

\noindent\emph{Step 1}. Fix $T>0$. We first show that for any sequence $R_n\rightarrow\infty$, the solution sequence $f_n=f_{R_n}$ is Cauchy in $\tilde{L}_\infty([0,T];\tSob{s_0})$ for $s>s_0>d/2+3$. Write $\tilde{f}=f_n-f_m$. Then $\tilde{f}$ satisfies
\[
\partial_t\tilde{f}+A\tilde{f}=\mathfrak{R}(f_n,f_m)=\mathcal{S}_{R_m}\tilde{\mathcal{R}}(f_m)-\mathcal{S}_{R_n}\tilde{\mathcal{R}}(f_n),
\]
where $\tilde{\mathcal{R}}(f)$ is given in \eqref{eq:tilde-R-expression}.

We note that 
\[ \mathfrak{R}(f_n,f_m)=\mathcal{S}_{R_n}[\tilde{\mathcal{R}}(f_n)-\tilde{\mathcal{R}}(f_m)]+(\mathcal{S}_{R_n}-\mathcal{S}_{R_m})(\tilde{\mathcal{R}}(f_m)). \]

Then by Lemma \ref{lem:truncation-Fourier} and \eqref{eq:final-remainder-estimate}, we have
\begin{align}\label{eq:difference-estimate-approximate-1} 
&\norm{(\mathcal{S}_{R_n}-\mathcal{S}_{R_m})(\tilde{\mathcal{R}}(f_m))}{\tSob{s_0-3}}\\
&\apprle \max \left\{(1+R_n)^{-(s-s_0)},(1+R_m)^{-(s-s_0)}\right\}\mathcal{F}(\norm{f_m}{\tilde{L}_\infty^t\tSob{s}})\norm{f_m}{\tilde{L}_\infty(0,T;\tSob{s})}^2\nonumber\\
&\apprle \max \left\{(1+R_n)^{-(s-s_0)},(1+R_m)^{-(s-s_0)}\right\}\mathcal{F}(2\varepsilon_0)\varepsilon_0^2.\nonumber
\end{align}

By Lemma \ref{lem:truncation-Fourier} and \eqref{eq:nonlinear-difference}, we have
\begin{equation}\label{eq:difference-estimate-approximate-2}
\norm{\mathcal{S}_{R_n}[\tilde{\mathcal{R}}(f_n)-\tilde{\mathcal{R}}(f_m)]}{\tilde{L}_\infty(0,T;\tSob{s_0-3})}\leq \mathcal{F}(2\varepsilon_0)\varepsilon_0\norm{\tilde{f}}{\tilde{L}_\infty(0,T;\tSob{s_0})}.
\end{equation}
Following the proof in \eqref{eq:tilde-estimate-difference-1} and \eqref{eq:tilde-estimate-difference-2}, we have 
\begin{align*}
\norm{\tilde{f}}{\tilde{L}_\infty([0,T];\tSob{s_0})}&\leq \norm{\tilde{f}|_{t=0}}{\tSob{s_0}}+(1+T)\norm{\mathfrak{R}(f_n,f_m)}{\tilde{L}_\infty^t\tSob{s_0-3}}.
\end{align*}
Hence, it follows from \eqref{eq:difference-estimate-approximate-1} and \eqref{eq:difference-estimate-approximate-2} that 
\begin{align*}
&\relphantom{=}\norm{\tilde{f}}{\tilde{L}_\infty([0,T];\tSob{s_0})}\\
&\leq \norm{\tilde{f}|_{t=0}}{\tSob{s_0}}+(1+T)\max \left\{(1+R_n)^{-(s-s_0)},(1+R_m)^{-(s-s_0)}\right\} \mathcal{F}(2\varepsilon_0)\varepsilon_0^2 \\
&\relphantom{=}+(1+T)\mathcal{F}(2\varepsilon_0)\varepsilon_0\norm{\tilde{f}}{\tilde{L}_\infty(0,T;\tSob{s_0})}.
\end{align*}
Then by the time-splitting argument as in \eqref{eq:time-splitting-argument}, we get 
\begin{align*}
\norm{\tilde{f}}{\tilde{L}_\infty(0,T;\tSob{s_0})}&\apprle_T \norm{\tilde{f}(0)}{\tSob{s_0}}+\max\left\{(1+R_n)^{-(s-s_0)},(1+R_m)^{-(s-s_0)}\right\}\mathcal{F}(2\varepsilon_0)\varepsilon_0^2.
\end{align*}
Hence, $\{f_n\}$ is Cauchy in $\tilde{L}_\infty(0,T;\tSob{s_0})$ and the limit $f=\lim_{n\rightarrow\infty}f_n$ exists in $\tilde{L}_\infty(0,T;\tSob{s_0})$. Moreover, since $T>0$ is arbitrary and $f_R \in C^1([0,\infty);\tSob{s_0})$, it follows that $f\in C([0,\infty);\tSob{s_0})$. 

Next, we show that $f$ is a solution to \eqref{eq:Muskat-reform-revisit}. Since
\begin{align*}
&\relphantom{=}\mathcal{S}_{R_n}[G(f_n)(f_n+H(f_n))]-G(f)(f+H(f))\\
&=\mathcal{S}_{R_n}[Af_n -\tilde{\mathcal{R}}(f_n)]-[Af-\tilde{\mathcal{R}}(f)]\\
&=\mathcal{S}_{R_n}[A(f_n-f)-\tilde{\mathcal{R}}(f_n)+\tilde{\mathcal{R}}(f)]+(\mathcal{S}_{R_n}-I)(Af-\tilde{\mathcal{R}}(f)),
\end{align*}
we have
\begin{align*}
&\relphantom{=}\norm{\mathcal{S}_{R_n}[G(f_n)(f_n+H(f_n))]-G(f)(f+H(f))}{\tLeb{\infty}^t\bes{s_0-3}{2,1}}\\
&\leq \norm{A(f_n-f)}{\tLeb{\infty}^t\tSob{s_0-3}}+\norm{\tilde{\mathcal{R}}(f_n)-\tilde{\mathcal{R}}(f)}{\tLeb{\infty}^t\tSob{s_0-3}}\\
&\relphantom{=}+\norm{(\mathcal{S}_{R_n}-I)(Af-\tilde{\mathcal{R}}(f))}{\tLeb{\infty}^t\tSob{s_0-3}}.
\end{align*}
Then by \eqref{eq:estimate-I-part1}, \eqref{eq:estimate-II-part2}, and Lemma \ref{lem:truncation-Fourier}, we get 
\begin{align*} 
&\apprle \norm{f_n-f}{\tLeb{\infty}^t\tSob{s_0}}+\norm{(\mathcal{S}_{R_n}-I)(Af-\tilde{\mathcal{R}}(f))}{\tLeb{\infty}^t\tSob{s_0-3}}\rightarrow 0
\end{align*}
as $n\rightarrow\infty$.

If we fix $\phi \in C_c^\infty(\mathbb{R}^d\times [0,\infty))$, then choose $T>0$ so that $\supp \phi \subset \mathbb{R}^d \times [0,T)$. Then it follows from Lemma \ref{lem:truncation-Fourier} that $f$ satisfies
\begin{equation}\label{eq:weak-solution-Muskat}
 -\int_0^\infty\int_{\mathbb{R}^d} f\partial_t \phi \myd{x}dt+\int_0^\infty \int_{\mathbb{R}^d} G(f)(f+H(f)) \phi \myd{x}dt=\int_{\mathbb{R}^d} f_0(x) \phi(x,0)\myd{x},
\end{equation}
which proves that $f$ is a weak solution of \eqref{eq:DtoN-formulation-surface-tension} with the initial data $f_0$. Since $f_R$ satisfies \eqref{eq:a-priori-global-estimate}, it follows from weak/weak* compactness that $f$ satisfies 
\begin{equation}\label{eq:a-priori-estimate-in-the-end}
\norm{f(t)}{\tSob{s}}^2+c\int_0^t \norm{f(\tau)}{\thSob{1/2}}^2+\norm{f(\tau)}{\thSob{s+3/2}}^2\myd{\tau}\leq \norm{f_0}{\tSob{s}}^2
\end{equation}
for all $t\geq 0$. Moreover, it follows from Proposition \ref{prop:continuity-DN} that $\partial_t f \in \Leb{2}(0,T;\tSob{s-3/2})$  for any $T>0$. Since $f\in \Sob{1}{2}(0,T;\tSob{s-3/2})\cap \Leb{2}(0,T;\tSob{s+3/2})$, it follows from Theorem \ref{thm:Lions-Peetre} that $f\in C([0,T];\tSob{s})$. Such a solution is unique by Proposition \ref{prop:contraction-solution}.

 \bigskip

\noindent\emph{Step 2}. To relax the assumption $f_0\in \tSob{s}$, $s>d/2+3$, we will apply the local well-posedness result of Nguyen \cite{N20} to use parabolic smoothing in a short time. After a small time, we apply Step 1 to show that the solution exists globally. 

By paralinearization argument, it was proved in \cite[Proposition 3.3]{N20} that 
\begin{equation}\label{eq:energy-estimate-surface-tension}
\frac{1}{2}\frac{d}{dt}\norm{f}{\tSob{s}}^2\leq -\frac{1}{\mathcal{F}(\norm{f}{\tSob{s}})}\norm{f}{\tSob{s+3/2}}^2+\mathcal{F}(\norm{f}{\tSob{s}})\norm{f}{\tSob{s}}^2.
\end{equation}
Then by Gronwall's inequality, we have 
\begin{equation}\label{eq:L-infty-estimate-exponential}
 \norm{f}{\Leb{\infty}([0,T];\tSob{s})}^2\leq \norm{f_0}{\tSob{s}}^2\exp\left(T\mathcal{F}(\norm{f}{\Leb{\infty}([0,T];\tSob{s})}))\right).
\end{equation}

By \eqref{eq:energy-estimate-surface-tension} and \eqref{eq:L-infty-estimate-exponential}, we have
\begin{align*}
&\relphantom{=}\norm{f}{\Leb{2}([0,T];\tSob{s+3/2})}^2\\
&\leq \norm{f_0}{\tSob{s}}^2\exp\left(T\mathcal{F}(\norm{f}{\Leb{\infty}([0,T];\tSob{s})}))\right)\mathcal{F}(\norm{f}{\Leb{\infty}([0,T];\tSob{s})})\\
&\leq  \norm{f_0}{\tSob{s}}^2\exp\left(T\mathcal{F}(\norm{f}{\Leb{\infty}([0,T];\tSob{s})}))\right)\mathcal{F}(\norm{f_0}{\tSob{s}}\exp\left(T\mathcal{F}(\norm{f}{\Leb{\infty}([0,T];\tSob{s})}) /2\right).
\end{align*}
In summary, we have obtained the following proposition.

\begin{proposition}\label{prop:a-priori-estimate-local}
Let $s>d/2+1$. Then there exists a nondecreasing function $\mathcal{F}:\mathbb{R}^+\rightarrow (0,\infty)$ depending only on $s$ and $d$ such that if $f \in \Leb{\infty}([0,T];\tSob{s})\cap\Leb{2}([0,T];\tSob{s+3/2})$ is a solution of \eqref{eq:Muskat-reform-revisit} with initial data $f_0\in\tSob{s}$, then
\begin{align*}
\norm{f}{\Leb{\infty}([0,T];\tSob{s})}&\leq \norm{f_0}{\tSob{s}}\exp\left(T\mathcal{F}(\norm{f}{\Leb{\infty}([0,T];\tSob{s})}))/2\right),\\
\norm{f}{\Leb{2}([0,T];\tSob{s+3/2})}^2&\leq \mathcal{F}_1(\norm{f_0}{\tSob{s}}\exp\left(T\mathcal{F}(\norm{f}{\Leb{\infty}([0,T];\tSob{s})})/2)\right),
\end{align*}
where $\mathcal{F}_1(m)=m^2\mathcal{F}(m)$.
\end{proposition} 

By a standard bootstrap argument, one can show that if
\[ T\leq \frac{\ln 4}{\mathcal{F}(4\norm{f_0}{\tSob{s}})},\]
then 
\begin{equation*}
\begin{aligned}
\norm{f}{\Leb{\infty}([0,T];\tSob{s})}&\leq 2\norm{f_0}{\tSob{s}},\\
\norm{f}{\Leb{2}([0,T];\tSob{s+3/2})}^2&\leq 4\mathcal{F}(2\norm{f_0}{\tSob{s}})\norm{f_0}{\tSob{s}}^2.
\end{aligned}
\end{equation*}

 We first assume 
\begin{equation*}\label{eq:initial-regularity-data-assumption}
\norm{f_0}{\tSob{s}}\leq 1.
\end{equation*}
By Proposition \ref{prop:a-priori-estimate-local} and \cite[Theorem 1.2]{N20}, if we set 
\[ T_0=\frac{\ln 4}{\mathcal{F}(4)},\]
then the problem admits a unique solution on $[0,T_0]$ satisfying 
\begin{equation*}
\begin{aligned} 
\norm{f}{\Leb{2}([0,T_0];\tSob{s+3/2})}&\leq 2\sqrt{\mathcal{F}(2\norm{f_0}{\tSob{s}})}\norm{f_0}{\tSob{s}}.
\end{aligned}
\end{equation*}

Choose $0<t_0<T_0$ so that $f(t_0)\in \tSob{s+3/2}$ and 
\begin{equation*}
 \norm{f(t_0)}{\tSob{s+3/2}}^2\leq \frac{1}{T_0} \int_0^{T_0} \norm{f(t)}{\tSob{s+3/2}}^2 dt\leq \frac{4}{T_0}\mathcal{F}(2)\norm{f_0}{\tSob{s}}^2.
 \end{equation*}

 From this $t_0$, we apply Proposition \ref{prop:a-priori-estimate-local} and \cite[Theorem 1.2]{N20} again to show that if we set
  \[ T_1=(\ln 4)\mathcal{F}\left(4\sqrt{\frac{4}{T_0}\mathcal{F}(2)} \right)^{-1}, \]
then the problem admits a unique solution on $[t_0,t_0+T_1]$ such that
\begin{equation*}
\norm{f}{\Leb{2}([t_0,t_0+T_1];\tSob{s+3})}^2\leq 4\mathcal{F}(2\norm{f(t_0)}{\tSob{s+3/2}})\norm{f(t_0)}{\tSob{s+3/2}}^2.
\end{equation*}
By a similar argument, choose $t_1>t_0$ so that $f(t_1)\in \tSob{s+3}$ and 
\begin{align*}
\norm{f(t_1)}{\tSob{s+3}}^2&\leq \frac{C_0}{T_1} \int_{t_0}^{t_0+T_1} \norm{f(t)}{\tSob{s+3}}^2 dt\\
 &\leq \frac{4C_0}{T_1}\mathcal{F}(2\norm{f(t_0)}{\tSob{s+3/2}})\norm{f(t_0)}{\tSob{s+3/2}}^2\\
 &\leq \frac{16C_0}{T_0T_1}\mathcal{F}\left(2\sqrt{\frac{4}{T_0}\mathcal{F}(2)}\right)\mathcal{F}(2)\norm{f_0}{\tSob{s}}^2.
 \end{align*}
 Since $s+3>d/2+3$, if we choose $\norm{f_0}{\tSob{s}}$ sufficiently small so that if
 \[ \frac{16C_0}{T_0T_1}\mathcal{F}\left(2\sqrt{\frac{4}{T_0}\mathcal{F}(2)}\right)\mathcal{F}(2)\norm{f_0}{\tSob{s}}^2\leq \varepsilon_0^2,\]
  then we can apply the result in Step 1 at $t=t_1$ to get the desired result. This completes the proof of Theorem \ref{thm:A}.

\begin{remark}
Using the above argument, one can show that the one-phase Muskat problem with surface tension allows instantaneous smoothing if we start with small initial data in $\tSob{s}$, $s>d/2+1$. Our argument needs the smallness assumption on the initial data.

For the critical case, Gancedo, Garc\'ia-Ju\'arez, Patel, and Strain \cite{GGPS19} proved instantaneous analyticity in Wiener space $\dot{\mathcal{F}}^{1,1}$. On the other hand, Agarwal, Patel, and Wu \cite{APW23} proved that the one-phase Muskat problem without surface tension exhibits a waiting-time phenomenon for Lipschitz initial data with an acute angle. It would be interesting to see whether we could eliminate the waiting-time phenomenon in the one-phase Muskat problem with surface tension.
\end{remark}

\subsection{Asymptotic behavior}
We prove Theorem \ref{thm:B}. By Theorem \ref{thm:A}, we know that for sufficiently small $\varepsilon_0>0$, if $\norm{f_0}{\tSob{s}}<\varepsilon_0$, then the problem admits a unique global solution in $\Leb{\infty}(0,\infty;\tSob{s})\cap\Leb{2}(0,\infty;\thSob{s+3/2})$. A similar proof to \eqref{eq:final-inequality-1} gives 
\begin{align*}
&\relphantom{=}\frac{1}{2}\norm{f(t)}{\thSob{s}}^2+\int_{t_0}^t\norm{f}{\thSob{s+3/2}}^2 d\tau\\
&\leq \frac{1}{2}\norm{f(t_0)}{\thSob{s}}^2+\int_{t_0}^t\mathcal{F}(\norm{f}{\tLeb{\infty}^t\tSob{s}})\norm{f}{\tLeb{\infty}^t\tSob{s}}(\norm{f}{\thSob{1/2}}^2+\norm{f}{\thSob{s+3/2}}^2)d\tau.
\end{align*}

Choose $\varepsilon_0>0$ sufficiently small so that $\mathcal{F}(\norm{f}{\tLeb{\infty}^t\tSob{s}})\norm{f}{\tLeb{\infty}^t\tSob{s}}\leq 1/2$, we have
\begin{equation}\label{eq:small-decay-rate}
\norm{f(t)}{\thSob{s}}^2\leq \norm{f(t_0)}{\thSob{s}}^2+\int_{t_0}^t \norm{f}{\thSob{1/2}}^2d\tau.
\end{equation}
By interpolation, $\norm{f_0}{\tSob{s}}<\varepsilon_0$, and \eqref{eq:a-priori-estimate-in-the-end}, there exists a constant $C>0$ such that
\begin{equation}\label{eq:long-time-interpolation} \norm{f(t)}{\thSob{s}}\leq C\norm{f(t)}{\Leb{2}}^{\frac{3}{2s+3}} \norm{f(t)}{\thSob{s+3/2}}^{\frac{2s}{2s+3}}\leq C\varepsilon_0^{\frac{3}{2s+3}}\norm{f(t)}{\thSob{s+3/2}}^{\frac{2s}{2s+3}}
\end{equation}
for $t\geq 0$. Since $\int_0^\infty \norm{f}{\thSob{1/2}}^2+\norm{f}{\thSob{s+3/2}}^2\myd{\tau}<\infty$, given $\delta>0$, by \eqref{eq:a-priori-estimate-in-the-end} and \eqref{eq:long-time-interpolation}, choose $t_0>0$ so that 
\begin{equation}\label{eq:smallness-control-for-large-t}
 \int_{t_0}^\infty \norm{f}{\thSob{1/2}}^2d\tau <\frac{1}{2}\delta^2\quad \text{and}\quad \norm{f(t_0)}{\thSob{s}}^2<\frac{1}{2}\delta^2. 
\end{equation}
Hence, it follows from \eqref{eq:small-decay-rate} that $\norm{f(t)}{\thSob{s}}<\delta$ for all $t>t_0$. This shows that $\norm{f(t)}{\thSob{s}}\rightarrow 0$ as $t\rightarrow\infty$. 

Furthermore, by Bernstein's inequality, if $\ell\geq0$, $s$, $p$ satisfy
\[  2<p\leq \infty\quad \text{and}\quad \frac{1}{p}-\frac{\ell}{d}\geq\frac{1}{2}-\frac{s}{d},\]
then we have
\[ \norm{|\nabla|^\ell f}{\Leb{p}}\apprle \norm{f}{\Leb{2}}^{1-\theta}\norm{f}{\thSob{s}}^\theta,   \]
where 
\[ \theta = \frac{\ell}{s}+\frac{d}{s}\left(\frac{1}{2}-\frac{1}{p}\right).\]
In particular, this implies that the Lipschitz norm of $f$ converges to zero as $t\rightarrow\infty$. This completes the proof of Theorem \ref{thm:B}. \hfill \qedsymbol

\begin{remark}\label{rem:exponential-decay}
Suppose that $f_0 \in \tSob{s}(\mathbb{T}^d)$ whose mean is zero. Then by Theorem \ref{thm:A}, there exists $\varepsilon_0>0$ such that if $\norm{f_0}{\tSob{s}(\mathbb{T}^d)}<\varepsilon_0$, then there exists a unique global strong solution $f$ to \eqref{eq:Muskat-reform-revisit} satisfying $\norm{f}{\Leb{\infty}\tSob{s}}\leq 2\varepsilon_0$. Moreover, it follows from 
\begin{equation}\label{eq:L2-energy}
\frac{d}{dt}\frac{1}{2}\norm{f}{\Leb{2}}^2+c\norm{f}{\thSob{1/2}}^2+c\norm{f}{\thSob{3/2}}^2\leq 0.
\end{equation}
Since $f_0$ has mean zero, it follows from the definition of the Dirichlet-Neumann operator that $f$ has mean zero for any $t>0$. By the Poincar\'e inequality and \eqref{eq:L2-energy}, we get 
\[ \norm{f(t)}{\Leb{2}(\mathbb{T}^d)}\leq \norm{f_0}{\Leb{2}(\mathbb{T}^d)} \exp(-Ct)\quad \text{for all } t>0.\]

Furthermore, in Step 2 of the proof of Theorem \ref{thm:A}, we proved that there exists $T>0$ such that  $\norm{f(t)}{\tSob{s+3}(\mathbb{T}^d)}\apprle \varepsilon_0$ for $t>T$. Hence, by interpolation, we have 
\[ \norm{f(t)}{\tSob{s}}\leq \norm{f(t)}{\Leb{2}}^{1-\frac{s}{s+3}} \norm{f(t)}{\tSob{s+3}}^{\frac{s}{s+3}}\apprle \varepsilon_0^{\frac{s}{s+3}}\exp(-Ct) \]
for $t>T$. 
\end{remark}

\appendix

\section{Estimates on remainders}\label{app:Qa-Qb-estimates}
Recall that $\mathcal{P}=e^{z|\nabla|}f$, $z\leq 0$. For $1\leq p\leq \infty$, it follows from Proposition \ref{prop:properties-LP} that
\[
2^{k} \norm{P_k \mathcal{P}}{\Leb{1}^z\Leb{\infty}^t\Leb{p}^x}+\norm{P_k\mathcal{P}}{\Leb{\infty}^z\Leb{\infty}^t\Leb{p}^x}\apprle \norm{P_k f}{\Leb{\infty}^t\Leb{p}^x},\quad k\in \mathbb{Z}
\]
and 
\[ \norm{P_{\leq 0} \mathcal{P}}{\Leb{1}^z\Leb{\infty}^t\Leb{p}^x}\apprle \norm{P_{\leq 0}f}{\Leb{\infty}^t\Leb{p}^x}.\]
Hence, by the Bernstein inequality, for $\sigma>0$ and $1\leq p,q,r\leq\infty$, we have
\begin{equation}\label{eq:P-graph-estimates}
\begin{aligned}
\norm{\nabla \mathcal{P}}{\tilde{L}^z_{1}\tilde{L}^t_{\infty}\hbes{\sigma}{p,q}}+\norm{\partial_z\mathcal{P}}{\tilde{L}^z_{1}\tilde{L}^t_{\infty}\hbes{\sigma}{p,q}}&\apprle \norm{f}{\tilde{L}^t_{\infty}\hbes{\sigma}{p,q}},\\
\norm{\nabla \mathcal{P}}{\Leb{\infty}^{z,t,x}}+\norm{\partial_z\mathcal{P}}{\Leb{\infty}^{z,t,x}}&\apprle \norm{f}{\tilde{L}^t_\infty\hbes{d/p+1}{p,1}}.
\end{aligned}
\end{equation}

Also, recall that $Q_a$ and $Q_b$ can be rewritten as 
\begin{align*}
Q_a[w,v;f]&=\frac{1}{1+\mathcal{B}}\nabla  \mathcal{P}\cdot \nabla  v-\frac{\mathcal{B}}{1+\mathcal{B}}(w+|\nabla|v),\\
Q_b[w,v;f]&=(|\nabla|v+w+Q_a[w,v;f])\nabla \mathcal{P}-(\partial_z\mathcal{P})\nabla  v,
\end{align*}
where 
\[ \mathcal{B}=\frac{|\nabla \mathcal{P}|^2-\partial_z\mathcal{P}}{1+\partial_z\mathcal{P}}.\]

For $\sigma>0$ and $p,q\in [1,\infty]$, recall that 
\[ \norm{v}{Z^\sigma_{p,q}}=\norm{v}{\Leb{\infty}^{z,t,x}}+\norm{v}{\tLeb{1}^z\tLeb{\infty}^t\hbes{\sigma}{p,q}}. \]

Then by Theorem \ref{thm:Besov-product}, we have 
\begin{equation}\label{eq:Z-sigma-algebra}
 \norm{fg}{Z^{\sigma}_{p,q}}\leq C_0 \norm{f}{Z^{\sigma}_{p,q}}\norm{g}{Z^{\sigma}_{p,q}}.
\end{equation}
Also, for $1\leq p,q\leq \infty$, $s\geq \sigma>0$, and $s>d/p+1$, it follows from \eqref{eq:P-graph-estimates} that 
\begin{equation}\label{eq:P-graph-estimates-Zsigma}
\norm{\nabla \mathcal{P}}{Z^{\sigma}_{p,q}}+\norm{\partial_z\mathcal{P}}{Z^{\sigma}_{p,q}}\apprle \norm{f}{\tilde{L}^t_\infty\hbes{d/p+1}{p,1}\cap \tilde{L}^t_\infty\hbes{\sigma}{p,q}}\apprle \norm{f}{\tLeb{\infty}^t\bes{s}{p,q}}.
\end{equation}

\begin{lemma}\label{lem:nonlinear-contribution}
Let $1\leq p,q\leq \infty$, $s\geq \sigma >0$, and $s>d/p+1$. Then there exists a constant $\varepsilon_0>0$ such that if $\norm{f}{\tilde{L}^t_\infty\bes{s}{p,q}}<\varepsilon_0$, then 
\begin{equation*}\label{eq:P-control-Linfty} 
\norm*{\frac{\nabla \mathcal{P}}{1+\mathcal{B}}}{Z^{\sigma}_{p,q}}+\norm*{\frac{\mathcal{B}}{1+\mathcal{B}}}{Z^{\sigma}_{p,q}}\leq \mathcal{F}(\norm{f}{\tLeb{\infty}^t\bes{s}{p,q}})\norm{f}{\tLeb{\infty}^t\bes{s}{p,q}}.
\end{equation*}
\end{lemma}
\begin{proof}
By definition of $\mathcal{B}$, we have 
\begin{equation*}
\frac{\nabla  \mathcal{P}}{1+\mathcal{B}}=\frac{\nabla \mathcal{P}(1+\partial_z\mathcal{P})}{1+|\nabla \mathcal{P}|^2}.
\end{equation*}
Then by \eqref{eq:Z-sigma-algebra} and \eqref{eq:P-graph-estimates-Zsigma}, we get 
\begin{align*}
\norm*{\frac{\nabla \mathcal{P}}{1+\mathcal{B}}}{Z^{\sigma}_{p,q}}&\apprle (1+\norm{\partial_z \mathcal{P}}{Z^{\sigma}_{p,q}})\norm*{\frac{\nabla \mathcal{P}}{1+|\nabla \mathcal{P}|^2}}{Z^{\sigma}_{p,q}}\\
&\apprle  (1+\norm{\partial_z \mathcal{P}}{Z^{\sigma}_{p,q}})\norm{\nabla \mathcal{P}}{Z^{\sigma}_{p,q}}\norm*{\frac{1}{1+|\nabla \mathcal{P}|^2}}{Z^{\sigma}_{p,q}}.
\end{align*}

Since 
\[ \frac{\mathcal{B}}{1+\mathcal{B}}=\frac{|\nabla  \mathcal{P}|^2-\partial_z\mathcal{P}}{1+|\nabla \mathcal{P}|^2},\]
it follows from \eqref{eq:Z-sigma-algebra} and \eqref{eq:P-graph-estimates-Zsigma} that
\begin{align*}
\norm*{ \frac{\mathcal{B}}{1+\mathcal{B}}}{Z^{\sigma}_{p,q}}&\apprle \norm{\nabla \mathcal{P}}{Z^{\sigma}_{p,q}}\norm*{\frac{\nabla \mathcal{P}}{1+|\nabla \mathcal{P}|^2}}{Z^{\sigma}_{p,q}}+\norm*{\frac{\partial_z\mathcal{P}}{1+|\nabla \mathcal{P}|^2}}{Z^{\sigma}_{p,q}}\\
&\apprle \left(\norm{\nabla \mathcal{P}}{Z^{\sigma}_{p,q}}\norm{\nabla \mathcal{P}}{Z^{\sigma}_{p,q}}+\norm{\partial_z\mathcal{P}}{Z^{\sigma}_{p,q}}\right)\norm*{\frac{1}{1+|\nabla \mathcal{P}|^2}}{Z^{\sigma}_{p,q}}.
\end{align*}

To estimate $(1+|\nabla \mathcal{P}|^2)^{-1}$ in $Z^{\sigma}_{p,q}$, it follows from \eqref{eq:P-graph-estimates} that  $\norm{(1+|\nabla \mathcal{P}|^2)^{-1}}{\Leb{\infty}^{z,t,x}}\leq 1$. Hence, it remains to estimate it in $\tilde{L}^z_1\tLeb{\infty}^t\hbes{\sigma}{p,q}$.  By \eqref{eq:P-graph-estimates-Zsigma}, if we assume smallness assumption on $\norm{f}{\tLeb{\infty}^t\bes{s}{p,q}}$ so that $\norm{\nabla \mathcal{P}}{\Leb{\infty}^{z,t,x}}<1$, then 
\begin{equation}\label{eq:series-expression}
\frac{1}{1+|\nabla \mathcal{P}|^2}=\sum_{k=1}^\infty (-1)^k |\nabla \mathcal{P}|^{2k}.
\end{equation}

 For this purpose, set 
\[ a_k=\norm{|\nabla \mathcal{P}|^{2k}}{\tilde{L}^z_1\tilde{L}^t_\infty\hbes{\sigma}{p,q}}.\]
Then we will show that for sufficiently small $\varepsilon_0>0$, we have 
\[ \sum_{k=1}^\infty a_k \leq \mathcal{F}(\norm{f}{\tLeb{\infty}^t\bes{s}{p,q}}) \]
for some nondecreasing function $\mathcal{F}$. We need to track the exact dependence on the constant to guarantee the convergence. Write $C_0$ and $C_1$, the implicit constants in \eqref{eq:P-graph-estimates} and Theorem \ref{thm:Besov-product}, respectively. 

By Theorem \ref{thm:Besov-product} and \eqref{eq:P-graph-estimates}, we have 
\begin{align*}
a_1=\norm{\nabla  \mathcal{P}\cdot \nabla \mathcal{P}}{\tilde{L}^z_1\tilde{L}^t_\infty\hbes{\sigma}{p,q}}&\leq 2C_1\norm{\nabla \mathcal{P}}{\Leb{\infty}^{z,t,x}}\norm{\nabla \mathcal{P}}{\tilde{L}^z_1\tilde{L}^t_\infty\hbes{\sigma}{p,q}}\\
&\leq 2C_1 C_0^2 \norm{f}{\tLeb{\infty}^t\bes{s}{p,q}}^2.
\end{align*}

We claim that 
\begin{equation}\label{eq:convergence}
 a_k \leq (2C_0^2 C_1\norm{f}{\tLeb{\infty}^t\bes{s}{p,q}}^2)^k .
\end{equation}
Suppose that the claim holds for $k$. Then by Theorem \ref{thm:Besov-product} and \eqref{eq:P-graph-estimates}, we have 
\begin{align*}
a_{k+1}&\leq C_1 \left(a_k\norm{\nabla \mathcal{P}}{\Leb{\infty}^{z,t,x}}^2+a_1\norm{|\nabla \mathcal{P}|^{2k}}{\Leb{\infty}^{z,t,x}} \right)\\
&\leq C_1(C_0^2\norm{f}{\tilde{L}^t_\infty \bes{s}{p,q}}^2 a_k+2C_1 C_0^{2k+2}\norm{f}{\tilde{L}^t_\infty\bes{s}{p,q}}^{2k+2})\\
&\leq C_1(C_0^2 (2C_0^2 C_1)^k +2C_1C_0^{2k+2})\norm{f}{\tilde{L}^t_\infty \bes{s}{p,q}}^{2k+2}\\
&\leq (2C_0^2 C_1)^{k+1}\norm{f}{\tilde{L}^t_\infty \bes{s}{p,q}}^{2k+2}.
\end{align*}
Now if we choose $\varepsilon_0>0$ sufficiently small so that $2C_0^2 C_1 \varepsilon_0^2<1$, then it follows from \eqref{eq:series-expression} and \eqref{eq:convergence} that 
\begin{equation}\label{eq:1-over-x-square-control}
\norm*{\frac{1}{1+|\nabla \mathcal{P}|^2}}{\tilde{L}^z_1\tilde{L}^t_\infty\hbes{\sigma}{p,q}}\leq \sum_{k=1}^\infty a_k \leq \mathcal{F}(\norm{f}{\tilde{L}^t_\infty\bes{s}{p,q}}).
\end{equation}
This completes the proof of Lemma \ref{lem:nonlinear-contribution}.
\end{proof}

\begin{proof}[Proof of Lemma \ref{lem:Qa-Qb-control}]
We first estimate $Q_b$ in $\tilde{L}^z_1\tLeb{\infty}^t\hbes{\sigma}{p,q}$. By Theorem \ref{thm:Besov-product} and \eqref{eq:P-graph-estimates-Zsigma},  we have 
\begin{align*}
&\norm{Q_b}{\tilde{L}^z_1\tilde{L}^t_\infty\hbes{\sigma}{p,q}}\\
&\apprle \norm{(|\nabla|v+w+Q_a)\nabla \mathcal{P}}{\tilde{L}^z_1\tilde{L}^t_\infty\hbes{\sigma}{p,q}}+\norm{(\partial_z\mathcal{P})\nabla v}{\tilde{L}^z_1\tilde{L}^t_\infty\hbes{\sigma}{p,q}}\\
&\apprle \norm{\nabla \mathcal{P}}{Z^{\sigma}_{p,q}}(\norm{|\nabla|v}{Z^{\sigma}_{p,q}}+\norm{w}{Z^\sigma_{p,q}}+\norm{Q_a}{Z^{\sigma}_{p,q}})+\norm{\partial_z\mathcal{P}}{Z^{\sigma}_{p,q}}\norm{\nabla  v}{Z^{\sigma}_{p,q}}\\
&\apprle \norm{f}{\tLeb{\infty}^t\bes{s}{p,q}}(\norm{|\nabla|v}{Z^{\sigma}_{p,q}}+\norm{w}{Z^{\sigma}_{p,q}}+\norm{Q_a}{Z^{\sigma}_{p,q}}+\norm{\nabla  v}{Z^{\sigma}_{p,q}}).
\end{align*}
Then by Bernstein's inequality, we get 
\begin{equation}\label{eq:Zsigma-norm-gradient-w}
\begin{aligned}
 \norm{\nabla v}{Z^{\sigma}_{p,q}}+\norm{|\nabla|v}{Z^{\sigma}_{p,q}}&\apprle \norm{\nabla  v}{\tilde{L}^z_1\tilde{L}^t_\infty\hbes{\sigma}{p,q}}+\norm{\nabla  v}{\tilde{L}^z_\infty\tilde{L}^t_\infty\hbes{d/p}{p,1}},\\
\norm{w}{Z^{\sigma}_{p,q}}&\apprle \norm{w}{\tilde{L}^z_1\tilde{L}^t_\infty\hbes{\sigma}{p,q}}+\norm{w}{\tilde{L}^z_\infty\tilde{L}^t_\infty\hbes{d/p}{p,1}}.
\end{aligned}
\end{equation}
Hence, it remains for us to estimate $Q_a$ in $Z^\sigma_{p,q}$. By \eqref{eq:Z-sigma-algebra}, \eqref{eq:Zsigma-norm-gradient-w}, and Lemma \ref{lem:nonlinear-contribution}, we have
\begin{equation}\label{eq:Qa-Z-sigma}
\begin{aligned}
\norm{Q_a}{Z^\sigma_{p,q}}&\apprle \norm*{\frac{\nabla \mathcal{P}}{1+\mathcal{B}}}{Z^\sigma_{p,q}}\norm{\nabla  v}{Z^\sigma_{p,q}}+\norm*{\frac{\mathcal{B}}{1+\mathcal{B}}}{Z^\sigma_{p,q}}\left(\norm{w}{Z^\sigma_{p,q}}+\norm{|\nabla|v}{Z^\sigma_{p,q}}\right)\\
&\leq  \mathcal{F}(\norm{f}{\tLeb{\infty}^t\bes{s}{p,q}})\norm{f}{\tLeb{\infty}^t\bes{s}{p,q}}(\norm{\nabla  v}{Z^\sigma_{p,q}}+\norm{w}{Z^\sigma_{p,q}}).
\end{aligned}
\end{equation}
This completes the proof of Lemma \ref{lem:Qa-Qb-control}.
\end{proof}

\begin{proof}[Proof of Lemma \ref{lem:difference-estimate}]
To estimate $\delta Q_a$ and $\delta Q_b$, we write 
\[ \mathcal{B}_j=\frac{|\nabla  \mathcal{P}_j|^2-\partial_z\mathcal{P}_j}{1+\partial_z\mathcal{P}_j}\quad \text{and}\quad \mathcal{P}_j=e^{z|\nabla|} f_j,\quad j=1,2.\]
By \eqref{eq:P-graph-estimates}, we have 
\begin{equation}\label{eq:P-graph-estimates-difference}
\begin{aligned}
\norm{\nabla \delta\mathcal{P}}{\tilde{L}^z_1\tilde{L}^t_\infty\hbes{\sigma}{p,q}}+\norm{\partial_z\delta \mathcal{P}}{\tilde{L}^z_1\tilde{L}^t_\infty\hbes{\sigma}{p,q}}&\apprle \norm{\delta f}{\tilde{L}^t_\infty\hbes{\sigma}{p,q}},\\\norm{\nabla \delta\mathcal{P}}{\Leb{\infty}^{z,t,x}}+\norm{\partial_z\delta\mathcal{P}}{\Leb{\infty}^{z,t,x}}&\apprle \norm{\delta f}{\tilde{L}^t_\infty\hbes{d/p+1}{p,1}}.
\end{aligned}
\end{equation}

If we define 
\[ B_j=\frac{1}{1+\mathcal{B}_j} \nabla \mathcal{P}_j,\,\quad\quad E_j=\frac{\mathcal{B}_j}{1+\mathcal{B}_j},\quad  \delta g = g_1-g_2,\]
\[
Q_a^j=Q_a[v_j,w_j;f_j],\quad  \text{and}\quad Q_b^j=Q_b[v_j,w_j;f_j],\quad j=1,2,\]
then 
\begin{equation}\label{eq:Qa-Qb-delta}
\begin{aligned}
\delta Q_a&=(\delta B)\nabla  v_1+B_2(\delta \nabla  v)-(\delta E)(w_1+|\nabla|v_1)-E_2(\delta w+|\nabla|\delta v),\\
\delta Q_b&=(|\nabla|\delta v+\delta w+\delta Q_a)\nabla \mathcal{P}_1+(|\nabla|v_2+w_2+Q_a^2)\nabla \delta\mathcal{P}\\
&\relphantom{=}-(\partial_z\delta\mathcal{P})\nabla  v_1-(\partial_z\mathcal{P}_2)\nabla \delta v.
\end{aligned}
\end{equation}

To estimate $\delta B$ and $\delta E$, we need the following lemma.
\begin{lemma}\label{lem:delta-B-delta-E-estimates}
Let $1\leq p,q\leq \infty$, $s\geq \sigma >0$, and $s>d/p+1$. There exists $\varepsilon_0>0$ such that if 
\[ \norm{f_i}{\tLeb{\infty}^t\bes{s}{p,q}}<\varepsilon_0,\quad i=1,2,\]
then
\[
\norm{\delta B}{Z^\sigma_{p,q}}+\norm{\delta E}{Z^\sigma_{p,q}}\leq \mathcal{F}(\norm{f_1}{\tLeb{\infty}^t\bes{s}{p,q}},\norm{f_2}{\tLeb{\infty}^t\bes{s}{p,q}})\norm{\delta f}{\tLeb{\infty}^t\hbes{\sigma}{p,q}\cap\tLeb{\infty}^t\hbes{d/p+1}{p,1}}.
\]
\end{lemma}
\begin{proof}
We write 
\begin{align*}
\delta E &=-(1+\partial_z\mathcal{P}_2)\left[G_1(\nabla \mathcal{P}_1)-G_1(\nabla \mathcal{P}_2) \right]-(\partial_z\delta\mathcal{P})G_1(\nabla \mathcal{P}_1),
\end{align*}
where 
\[ G_1(x)=\frac{1}{1+|x|^2}.\]

Since $G_1$ is smooth and $\nabla G_1(0)=0$, it follows from Theorem \ref{thm:Moser-estimates}, \eqref{eq:P-graph-estimates-difference}, and the mean value theorem that 
\begin{equation}\label{eq:G1-difference-estimate}
\begin{aligned}
&\relphantom{=}\norm{G_1(\nabla \mathcal{P}_1)-G_1(\nabla \mathcal{P}_2)}{Z^\sigma_{p,q}}\\
&\leq \mathcal{F}(\norm{f_1}{\tLeb{\infty}^t\bes{s}{p,q}},\norm{f_2}{\tLeb{\infty}^t\bes{s}{p,q}})\norm{\delta f}{\tLeb{\infty}^t\hbes{\sigma}{p,q}\cap\tLeb{\infty}^t\hbes{d/p+1}{p,1}}.
\end{aligned}
\end{equation} 
Then it follows from \eqref{eq:P-graph-estimates}, \eqref{eq:Z-sigma-algebra}, and \eqref{eq:P-graph-estimates-difference} that 
\begin{align*}
\norm{\delta E}{Z^\sigma_{p,q}}&\apprle (1+\norm{\partial_z\mathcal{P}_2}{Z^\sigma_{p,q}})\norm{G_1(\nabla \mathcal{P}_1)-G_1(\nabla \mathcal{P}_2)}{Z^\sigma_{p,q}}+\norm{\partial_z \delta\mathcal{P}}{Z^\sigma_{p,q}}\norm{G_1(\nabla \mathcal{P}_1)}{Z^\sigma_{p,q}}\\
&\apprle (1+\norm{f_2}{\tLeb{\infty}^t\bes{s}{p,q}})\norm{G_1(\nabla \mathcal{P}_1)-G_1(\nabla \mathcal{P}_2)}{Z^\sigma_{p,q}}+\norm{\delta f}{\tilde{L}^t_\infty\hbes{\sigma}{p,q}}\norm{G_1(\nabla \mathcal{P}_1)}{Z^\sigma_{p,q}}.
\intertext{Then by \eqref{eq:G1-difference-estimate} and \eqref{eq:1-over-x-square-control}, we have }
&\leq \mathcal{F}(\norm{f_1}{\tLeb{\infty}^t\bes{s}{p,q}},\norm{f_2}{\tLeb{\infty}^t\bes{s}{p,q}})\norm{\delta f}{\tLeb{\infty}^t\hbes{\sigma}{p,q}\cap\tLeb{\infty}^t\hbes{d/p+1}{p,1}}.
\end{align*} 
 
 To estimate $\delta B$, we rewrite
\begin{align*}
\delta B&=\frac{1}{1+\mathcal{B}_1}\nabla \delta \mathcal{P}+\left(\frac{1}{1+\mathcal{B}_1}-\frac{1}{1+\mathcal{B}_2} \right)\nabla \mathcal{P}_2\\
&=(1+\partial_z\mathcal{P}_1)G_1(\nabla \mathcal{P}_1)\nabla \delta \mathcal{P}+\partial_z(\delta\mathcal{P})G_1(\nabla \mathcal{P}_2)\nabla \mathcal{P}_2\\
&\relphantom{=}+(1+\partial_z\mathcal{P}_1)[G_1(\nabla \mathcal{P}_1)-G_1(\nabla \mathcal{P}_2)]\nabla \mathcal{P}_2.
\end{align*}
Then by \eqref{eq:Z-sigma-algebra}, \eqref{eq:P-graph-estimates-Zsigma}, and \eqref{eq:P-graph-estimates-difference} we have 
\begin{align*}
\norm{\delta B}{Z^\sigma_{p,q}}&\apprle (1+\norm{\partial_z \mathcal{P}_1}{Z^\sigma_{p,q}})\norm{G_1(\nabla \mathcal{P}_1)}{Z^\sigma_{p,q}}\norm{\nabla (\delta\mathcal{P})}{Z^\sigma_{p,q}}\\
&\relphantom{=}+\norm{\partial_z(\delta\mathcal{P})}{Z^\sigma_{p,q}}\norm{G_1(\nabla \mathcal{P}_2)}{Z^\sigma_{p,q}}\norm{\nabla \mathcal{P}_2}{Z^\sigma_{p,q}}\\
&\relphantom{=}+(1+\norm{\partial_z\mathcal{P}_1}{Z^\sigma_{p,q}})\norm{G_1(\nabla \mathcal{P}_1)-G_1(\nabla \mathcal{P}_2)}{Z^\sigma_{p,q}}\norm{\nabla \mathcal{P}_2}{Z^\sigma_{p,q}}.
\intertext{Hence, it follows from \eqref{eq:1-over-x-square-control} and \eqref{eq:G1-difference-estimate} that}
&\leq \mathcal{F}(\norm{f_1}{\tLeb{\infty}^t\bes{s}{p,q}},\norm{f_2}{\tLeb{\infty}^t\bes{s}{p,q}})\norm{\delta f}{\tLeb{\infty}^t\hbes{\sigma}{p,q}\cap\tLeb{\infty}^t\hbes{d/p+1}{p,1}}.
\end{align*}
This completes the proof of Lemma \ref{lem:delta-B-delta-E-estimates}.
\end{proof}

Now we are ready to complete the proof of Lemma \ref{lem:difference-estimate}. By \eqref{eq:P-graph-estimates-Zsigma}, \eqref{eq:Qa-Qb-delta}, Lemmas \ref{lem:nonlinear-contribution} and \ref{lem:delta-B-delta-E-estimates}, we have 
\begin{equation}\label{eq:Qa-difference-Zsigma}
\begin{aligned}
\norm{\delta Q_a}{Z^\sigma_{p,q}}&\apprle \norm{\delta B}{Z^\sigma_{p,q}}\norm{\nabla  v_1}{Z^\sigma_{p,q}}+\norm{B_2}{Z^\sigma_{p,q}}\norm{\delta \nabla  v}{Z^\sigma_{p,q}}\\
&\relphantom{=}+\norm{\delta E}{Z^\sigma_{p,q}}\norm{w_1+|\nabla|v_1}{Z^\sigma_{p,q}}+\norm{E_2}{Z^\sigma_{p,q}}\norm{\delta w+|\nabla|\delta v}{Z^\sigma_{p,q}}\\
&\leq \mathcal{F}(\norm{f_1}{\tLeb{\infty}^t\bes{s}{p,q}},\norm{f_2}{\tLeb{\infty}^t\bes{s}{p,q}})\norm{\delta f}{\tLeb{\infty}^t\hbes{\sigma}{p,q}}(\norm{\nabla  v_1}{Z^\sigma_{p,q}}+\norm{w_1}{Z^\sigma_{p,q}})\\
&\relphantom{=}+\mathcal{F}(\norm{f_2}{\tLeb{\infty}^t\bes{s}{p,q}})\norm{f_2}{\tilde{L}^t_\infty\bes{s}{p,q}}(\norm{\delta \nabla  v}{Z^\sigma_{p,q}}+\norm{\delta w}{Z^\sigma_{p,q}}).
\end{aligned}
\end{equation}
Similarly, we have
\begin{align*}
\norm{\delta Q_b}{Z^\sigma_{p,q}}&\apprle \norm{|\nabla|\delta v+\delta w +\delta Q_a}{Z^\sigma_{p,q}}\norm{\nabla \mathcal{P}_1}{Z^\sigma_{p,q}}+\norm{|\nabla|v_2+w_2+Q_a^2}{Z^\sigma_{p,q}}\norm{\nabla \delta\mathcal{P}}{Z^\sigma_{p,q}}\\
&\relphantom{=}+\norm{|\nabla|\delta\mathcal{P}}{Z^\sigma_{p,q}}\norm{\nabla  v_1}{Z^\sigma_{p,q}}+\norm{\partial_z\mathcal{P}_2}{Z^\sigma_{p,q}}\norm{\nabla (\delta v)}{Z^\sigma_{p,q}}\\
&\apprle \norm{f_1}{\tLeb{\infty}^t\bes{s}{p,q}}(\norm{|\nabla|\delta v}{Z^\sigma_{p,q}}+\norm{\delta w}{Z^\sigma_{p,q}}+\norm{\delta Q_a}{Z^\sigma_{p,q}})\\
&\relphantom{=}+\norm{\delta f}{\tLeb{\infty}^t\hbes{\sigma}{p,q}}(\norm{\nabla  v_1}{Z^\sigma_{p,q}}+\norm{\nabla  v_2}{Z^\sigma_{p,q}}+\norm{w_2}{Z^\sigma_{p,q}}+\norm{Q_a^2}{Z^\sigma_{p,q}})\\
&\relphantom{=}+\norm{f_2}{\tLeb{\infty}^t\bes{s}{p,q}}\norm{\nabla (\delta v)}{Z^\sigma_{p,q}}.
\end{align*}
Hence, it follows from \eqref{eq:Qa-Z-sigma} and \eqref{eq:Qa-difference-Zsigma} that
\begin{align*}
&\relphantom{=}\norm{\delta Q_b}{Z^\sigma_{p,q}}\\
&\apprle \norm{\delta f}{\tLeb{\infty}^t\hbes{\sigma}{p,q}}\mathcal{F}(\norm{f_1}{\tLeb{\infty}^t\bes{s}{p,q}},\norm{f_2}{\tLeb{\infty}^t\bes{s}{p,q}})\left(\sum_{i=1}^2(\norm{\nabla v_i}{Z^\sigma_{p,q}}+\norm{w_i}{Z^\sigma_{p,q}})\right)\\
&\relphantom{=}+(\norm{f_1}{\tLeb{\infty}^t\bes{s}{p,q}}+\norm{f_2}{\tLeb{\infty}^t\bes{s}{p,q}})(\norm{\nabla_x(\delta v)}{Z^\sigma_{p,q}}+\norm{\delta w}{Z^\sigma_{p,q}})\\
&\relphantom{=}+\mathcal{F}(\norm{f_2}{\tLeb{\infty}^t\bes{s}{p,q}})\norm{f_2}{\tLeb{\infty}^t\bes{s}{p,q}}(\norm{\nabla_x(\delta v)}{Z^\sigma_{p,q}}+\norm{\delta w}{Z^\sigma_{p,q}}).
\end{align*}
This completes the proof of Lemma \ref{lem:difference-estimate}.
\end{proof}

\section{Interpolation theorem}\label{app:interpolation}

In this appendix, we list several theorems on real interpolation spaces and prove real interpolation theorems on Chemin-Lerner type spaces for the sake of completeness. The following result is due to Lions and Peetre, whose proof can be found in standard literature on interpolation spaces. See e.g. \cite[Theorem L.4.1]{HvNVW23}.
\begin{theorem}\label{thm:Lions-Peetre}
Let $(X_0,X_1)$ be an interpolation couple of Banach spaces and let $p\in (1,\infty)$. Then we have the following continuous embedding 
\begin{equation}
\Sob{1}{p}(0,\infty;X_0)\cap\Leb{p}(0,\infty;X_1)\hookrightarrow C_b([0,\infty);(X_0,X_1)_{1-1/p,p}).
\end{equation}
\end{theorem}
 
The following theorem shows that Chemin-Lerner type spaces are real interpolation scales. 
\begin{theorem}\label{thm:CL-interpolation}
Let $\theta \in (0,1)$, $s_0\neq s_1\in\mathbb{R}$, and $p$, $q_0$, $q_1$, $a$, $b$, $\kappa$ be given numbers in $[1,\infty]$ satisfying $s=(1-\theta)s_0+\theta s_1$. Then
\begin{align*}
(\tLeb{a}^t\bes{s_0}{p,q_0},\tLeb{a}^t\bes{s_1}{p,q_1})_{\theta,\kappa}&=\tLeb{a}^t\bes{s}{p,\kappa}, \\
(\tLeb{b}^z\tLeb{a}^t\bes{s_0}{p,q_0},\tLeb{b}^z\tLeb{a}^t\bes{s_1}{p,q_1})_{\theta,\kappa}&=\tLeb{b}^z\tLeb{a}^t\bes{s}{p,\kappa}.
\end{align*}  
Similar results also hold for homogeneous Chemin-Lerner type spaces.
\end{theorem}

To show this theorem, we first introduce interpolation results on sequential spaces. Let $X$ be a Banach space. For $s\in \mathbb{R}$ and $q\in [1,\infty]$. We define 
\begin{align*}
\norm{\{a_k\}}{\ell^s_q(X)}&=\left(\sum_{k=-1}^\infty 2^{skq}\norm{a_k}{X}^q \right)^{1/q},\\
\norm{\{a_k\}}{\dot{\ell}^s_q(X)}&=\left(\sum_{k=-\infty}^\infty 2^{skq}\norm{a_k}{X}^q \right)^{1/q}.
\end{align*}
We denote by $\ell_q^s(X)$ and $\dot{\ell}^s_q(X)$ the space of all sequences $\{a_k\}$, $a_k\in X$ such that its $\ell_q^s(X)$ and $\dot{\ell}^s_q(X)$ norms are finite, respectively.

The following interpolation result can be found in \cite[Theorems 5.6.1]{BL76}.
\begin{proposition}\label{prop:sequential-interpolation}
Let $Y$ be a Banach space, let $q_0,q_1 \in [1,\infty]$, and let $s_0,s_1\in \mathbb{R}$, and $\theta \in (0,1)$. If $s_0\neq s_1$, $q\in [1,\infty]$, then
\[ (\ell_{q_0}^{s_0}(Y),\ell_{q_1}^{s_1}(Y))_{\theta,q}=\ell_q^{s}(Y) \]
with equivalent norms, where $s=(1-\theta)s_0+\theta s_1$. 
Similar results also hold for $\dot{\ell}_q^s$. 
\end{proposition}

Another ingredient of proof of Theorem \ref{thm:CL-interpolation} is the retraction properties of Chemin-Lerner type spaces. 
\begin{proposition}\label{prop:retract-Chemin-Lerner}
Let $p,q,a,b\in [1,\infty]$ and $s\in\mathbb{R}$. For $k\geq 0$, set $\tilde{P}_k=P_{k-1}+P_k+P_{k+1}$, where $P_{-1}=P_{\leq 0}$ by convention. Define 
$$
\left\{\begin{aligned}
	R&:\ell_q^s(\Leb{b}^z\Leb{a}^t\Leb{p}^x)\rightarrow\tLeb{b}^z\tLeb{a}^t\bes{s}{p,q}\\
	S&:\tLeb{b}^z\tLeb{a}^t\bes{s}{p,q}\rightarrow\ell_q^s(\Leb{b}^z\Leb{a}^t\Leb{p}^x) 
\end{aligned}
\right.
$$
by
\[ R((f_k)_{k\geq -1})=\sum_{k \geq -1} \tilde{P}_kf_k,\quad Sf=({P}_kf)_{k\geq -1}.\]
Then $R$ is bounded, $S$ is an isometry, and $RS=I$. Similar results also hold for its homogeneous Chemin-Lerner spaces.
\end{proposition}
\begin{proof}
Clearly, $S$ is an isometry by the definition of the Chemin-Lerner spaces. By Young's inequality, we have
\begin{align*}
\norm*{\sum_{k\geq -1} \tilde{P}_kf_k}{\tLeb{b}^z\tLeb{a}^t\bes{s}{p,q}}&\apprle \sum_{|l|\leq 2} \norm{(P_k \tilde{P}_{k+l} f_{k+l})_{k\geq -1}}{\ell_q^s(\Leb{b}^z\Leb{a}^t\Leb{p}^x)}\\
&\apprle \sum_{|l|\leq 2}  \norm{(f_{k+l})_{k\geq -1}}{\ell_q^s(\Leb{b}^z\Leb{a}^t\Leb{p}^x)}\\
&\apprle \norm{(f_k)_{k\geq -1}}{\ell^s_q(\Leb{b}^z\Leb{a}^t\Leb{p}^x)}.
\end{align*}
This implies that $R$ is bounded from $\ell_q^s(\Leb{b}^z\Leb{a}^t\Leb{p}^x)$ to $\tLeb{b}^z\tLeb{a}^t\bes{s}{p,q}$. Also, by \eqref{eq:almost-orthogonality-LP}, we see that $RS=I$. This completes the proof of of the proposition.
\end{proof}

\begin{proof}[Proof of Theorem \ref{thm:CL-interpolation}]
Theorem \ref{thm:CL-interpolation} is the consequence of Propositions \ref{prop:sequential-interpolation} and \ref{prop:retract-Chemin-Lerner}. See \cite[Theorem 6.4.5]{BL76} or \cite[Theorems 14.4.31]{HvNVW23} for details.
\end{proof}

\bibliographystyle{amsplain}

%\bibliography{Refs}

\begin{thebibliography}{10}

\bibitem{AP25}
S.~Agrawal and N.~Patel, \emph{Self-{S}imilar {S}olutions to the {H}ele-{S}haw
  problem with {S}urface {T}ension}, 2025.

\bibitem{APW23}
S.~Agrawal, N.~Patel, and S.~Wu, \emph{Rigidity of acute angled corners for one
  phase {M}uskat interfaces}, Adv. Math. \textbf{412} (2023), Paper No. 108801,
  71.% \MR{4520423}

\bibitem{A25}
T.~Alazard, \emph{Paralinearization of free boundary problems in fluid
  dynamics}, Partial differential equations: waves, nonlinearities and
  nonlocalities, Abel Symp., vol.~18, Springer, Cham, [2025] \copyright 2025,
  pp.~1--31.% \MR{4952918}

\bibitem{AB24}
T.~Alazard and D.~Bresch, \emph{Functional inequalities and strong {L}yapunov
  functionals for free surface flows in fluid dynamics}, Interfaces Free Bound.
  \textbf{26} (2024), no.~1, 1--30.% \MR{4705639}

\bibitem{ABZ14}
T.~Alazard, N.~Burq, and C.~Zuily, \emph{On the {C}auchy problem for gravity
  water waves}, Invent. Math. \textbf{198} (2014), no.~1, 71--163.% \MR{3260858}

\bibitem{AD15}
T.~Alazard and D.~Jean-Marc, \emph{Sobolev estimates for two-dimensional
  gravity water waves}, Ast\'erisque \textbf{374} (2015), viii+241 pp.

\bibitem{AMS20}
T.~Alazard, N.~Meunier, and D.~Smets, \emph{Lyapunov functions, identities and
  the {C}auchy problem for the {H}ele-{S}haw equation}, Comm. Math. Phys.
  \textbf{377} (2020), no.~2, 1421--1459.% \MR{4115021}

\bibitem{AN23}
T.~Alazard and Q.-H. Nguyen, \emph{Endpoint {S}obolev theory for the {M}uskat
  equation}, Comm. Math. Phys. \textbf{397} (2023), no.~3, 1043--1102.
 % \MR{4541917}

\bibitem{A04}
D.~M. Ambrose, \emph{Well-posedness of two-phase {H}ele-{S}haw flow without
  surface tension}, European J. Appl. Math. \textbf{15} (2004), no.~5,
  597--607.% \MR{2128613}

\bibitem{BCD11}
H.~Bahouri, J.-Y. Chemin, and R.~Danchin, \emph{Fourier analysis and nonlinear
  partial differential equations}, Grundlehren der mathematischen
  Wissenschaften [Fundamental Principles of Mathematical Sciences], vol. 343,
  Springer, Heidelberg, 2011.% \MR{2768550}

\bibitem{BL76}
J.~Bergh and J.~L\"ofstr\"om, \emph{Interpolation spaces. {A}n introduction},
  Grundlehren der Mathematischen Wissenschaften, vol. No. 223, Springer-Verlag,
  Berlin-New York, 1976.% \MR{482275}

\bibitem{BCG25}
{E}. Bocchi, \'A. Castro, and {F}. Gancedo, \emph{Global-in-time estimates for
  the 2d one-phase {M}uskat problem with contact points}, Comm. Math. Phys.
  (2026), to appear, arXiv:2502.19286.

\bibitem{BN24}
J.~Brownfield and H.~Q. Nguyen, \emph{Slowly traveling gravity waves for
  {D}arcy flow: existence and stability of large waves}, Comm. Math. Phys.
  \textbf{405} (2024), no.~10, Paper No. 222, 25.% \MR{4797733}

\bibitem{CCFG13}
\'A. Castro, D.~C\'ordoba, C.~Fefferman, and F.~Gancedo, \emph{Breakdown of
  smoothness for the {M}uskat problem}, Arch. Ration. Mech. Anal. \textbf{208}
  (2013), no.~3, 805--909.% \MR{3048596}

\bibitem{CCFG16}
\bysame, \emph{Splash singularities for the one-phase {M}uskat problem in
  stable regimes}, Arch. Ration. Mech. Anal. \textbf{222} (2016), no.~1,
  213--243.% \MR{3519969}

\bibitem{CCFG12}
{{\'A}}. Castro, D.~C\'ordoba, C.~Fefferman, F.~Gancedo, and
  J.~G\'omez-Serrano, \emph{Finite time singularities for water waves with
  surface tension}, J. Math. Phys. \textbf{53} (2012), no.~11, 115622, 26.
 % \MR{3026567}

\bibitem{CCCG12}
\'A. Castro, D.~C\'ordoba, C.~Fefferman, F.~Gancedo, and
  M.~L\'opez-Fern\'andez, \emph{Rayleigh-{T}aylor breakdown for the {M}uskat
  problem with applications to water waves}, Ann. of Math. (2) \textbf{175}
  (2012), no.~2, 909--948.% \MR{2993754}

\bibitem{CL95}
J.-Y. Chemin and N.~Lerner, \emph{Flot de champs de vecteurs non lipschitziens
  et \'equations de {N}avier-{S}tokes}, J. Differential Equations \textbf{121}
  (1995), no.~2, 314--328.% \MR{1354312}

\bibitem{CHN24}
K.~Chen, R.~Hu, and Q.-H. Nguyen, \emph{Well-posedness for local and nonlocal
  quasilinear evolution equations in fluids and geometry}, arXiv:2407.05313,
  2024.

\bibitem{CNX22}
K.~Chen, Q.-H. Nguyen, and Y.~Xu, \emph{The {M}uskat problem with {$C^1$}
  data}, Trans. Amer. Math. Soc. \textbf{375} (2022), no.~5, 3039--3060.
 % \MR{4402655}

\bibitem{C93}
X.~Chen, \emph{The {H}ele-{S}haw problem and area-preserving curve-shortening
  motions}, Arch. Rational Mech. Anal. \textbf{123} (1993), no.~2, 117--151.
 % \MR{1219420}

\bibitem{CGS16}
C.~H.~A. Cheng, R.~Granero-Belinch\'on, and S.~Shkoller, \emph{Well-posedness
  of the {M}uskat problem with {$H^2$} initial data}, Adv. Math. \textbf{286}
  (2016), 32--104.% \MR{3415681}

\bibitem{CJK07}
S.~Choi, D.~Jerison, and I.~Kim, \emph{Regularity for the one-phase
  {H}ele-{S}haw problem from a {L}ipschitz initial surface}, Amer. J. Math.
  \textbf{129} (2007), no.~2, 527--582.% \MR{2306045}

\bibitem{CJK09}
\bysame, \emph{Local regularization of the one-phase {H}ele-{S}haw flow},
  Indiana Univ. Math. J. \textbf{58} (2009), no.~6, 2765--2804.% \MR{2603767}

\bibitem{CK06}
S.~Choi and I.~Kim, \emph{Waiting time phenomena of the {H}ele-{S}haw and the
  {S}tefan problem}, Indiana Univ. Math. J. \textbf{55} (2006), no.~2,
  525--551.% \MR{2225444}

\bibitem{CCGRS16}
P.~Constantin, D.~C\'ordoba, F.~Gancedo, L.~Rodr\'iguez-Piazza, and R.~M.
  Strain, \emph{On the {M}uskat problem: global in time results in 2{D} and
  3{D}}, Amer. J. Math. \textbf{138} (2016), no.~6, 1455--1494.% \MR{3595492}

\bibitem{CCGS13}
P.~Constantin, D.~C\'ordoba, F.~Gancedo, and R.~M. Strain, \emph{On the global
  existence for the {M}uskat problem}, J. Eur. Math. Soc. (JEMS) \textbf{15}
  (2013), no.~1, 201--227.% \MR{2998834}

\bibitem{CP93}
P.~Constantin and M.~Pugh, \emph{Global solutions for small data to the
  {H}ele-{S}haw problem}, Nonlinearity \textbf{6} (1993), no.~3, 393--415.
 % \MR{1223740}

\bibitem{CCG11}
A.~C\'ordoba, D.~C\'ordoba, and F.~Gancedo, \emph{Interface evolution: the
  {H}ele-{S}haw and {M}uskat problems}, Ann. of Math. (2) \textbf{173} (2011),
  no.~1, 477--542.% \MR{2753607}

\bibitem{CG07}
D.~C\'ordoba and F.~Gancedo, \emph{Contour dynamics of incompressible 3-{D}
  fluids in a porous medium with different densities}, Comm. Math. Phys.
  \textbf{273} (2007), no.~2, 445--471.% \MR{2318314}

\bibitem{CL21}
D.~C\'ordoba and O.~Lazar, \emph{Global well-posedness for the 2{D} stable
  {M}uskat problem in {$H^{3/2}$}}, Ann. Sci. \'Ec. Norm. Sup\'er. (4)
  \textbf{54} (2021), no.~5, 1315--1351.% \MR{4363243}

\bibitem{CP17}
D.~C\'ordoba and T.~Pernas-Casta\~no, \emph{Non-splat singularity for the
  one-phase {M}uskat problem}, Trans. Amer. Math. Soc. \textbf{369} (2017),
  no.~1, 711--754.% \MR{3557791}

\bibitem{CS14}
D.~Coutand and S.~Shkoller, \emph{On the finite-time splash and splat
  singularities for the 3-{D} free-surface {E}uler equations}, Comm. Math.
  Phys. \textbf{325} (2014), no.~1, 143--183.% \MR{3147437}

\bibitem{D1856}
H.~Darcy, \emph{Les fontaines publiques de la ville de dijon}, Dalmont, Paris,
  1856.

\bibitem{DGN23}
H.~Dong, F.~Gancedo, and H.~Q. Nguyen, \emph{Global well-posedness for the
  one-phase {M}uskat problem}, Comm. Pure Appl. Math. \textbf{76} (2023),
  no.~12, 3912--3967.% \MR{4655356}

\bibitem{DGN23b}
\bysame, \emph{Global well-posedness for the one-phase {M}uskat problem in
  3{D}}, arXiv:2308.14230, 2023.

\bibitem{EM11}
J.~Escher and B.-V. Matioc, \emph{On the parabolicity of the {M}uskat problem:
  well-posedness, fingering, and stability results}, Z. Anal. Anwend.
  \textbf{30} (2011), no.~2, 193--218.% \MR{2793001}

\bibitem{ES97b}
J.~Escher and G.~Simonett, \emph{Classical solutions for {H}ele-{S}haw models
  with surface tension}, Adv. Differential Equations \textbf{2} (1997), no.~4,
  619--642.% \MR{1441859}

\bibitem{ES97a}
\bysame, \emph{Classical solutions of multidimensional {H}ele-{S}haw models},
  SIAM J. Math. Anal. \textbf{28} (1997), no.~5, 1028--1047.% \MR{1466667}

\bibitem{FN21}
P.~T. Flynn and H.~Q. Nguyen, \emph{The vanishing surface tension limit of the
  {M}uskat problem}, Comm. Math. Phys. \textbf{382} (2021), no.~2, 1205--1241.
 % \MR{4227171}

\bibitem{GGPS19}
F.~Gancedo, E.~Garc\'ia-Ju\'arez, N.~Patel, and R.~M. Strain, \emph{On the
  {M}uskat problem with viscosity jump: global in time results}, Adv. Math.
  \textbf{345} (2019), 552--597.% \MR{3899970}

\bibitem{GGPS23}
\bysame, \emph{Global regularity for gravity unstable {M}uskat bubbles}, Mem.
  Amer. Math. Soc. \textbf{292} (2023), no.~1455, v+87.% \MR{4679708}

\bibitem{GL22}
F.~Gancedo and O.~Lazar, \emph{Global well-posedness for the three dimensional
  {M}uskat problem in the critical {S}obolev space}, Arch. Ration. Mech. Anal.
  \textbf{246} (2022), no.~1, 141--207.% \MR{4487512}

\bibitem{GS14}
F.~Gancedo and R.~M. Strain, \emph{Absence of splash singularities for surface
  quasi-geostrophic sharp fronts and the {M}uskat problem}, Proc. Natl. Acad.
  Sci. USA \textbf{111} (2014), no.~2, 635--639.% \MR{3181769}

\bibitem{GGHP24}
E.~Garc\'ia-Ju\'arez, J.~G\'omez-Serrano, S.~V. Haziot, and B.~Pausader,
  \emph{Desingularization of small moving corners for the {M}uskat equation},
  Ann. PDE \textbf{10} (2024), no.~2, Paper No. 17, 71.% \MR{4790934}

\bibitem{GGNP22}
E.~Garc\'ia-Ju\'arez, J.~G\'omez-Serrano, H.~Q. Nguyen, and B.~Pausader,
  \emph{Self-similar solutions for the {M}uskat equation}, Adv. Math.
  \textbf{399} (2022), Paper No. 108294, 30.% \MR{4385135}

\bibitem{GP25}
D.~Ginsberg and F.~Pusateri, \emph{Long {T}ime {R}egularity for 3{D} {G}ravity
  {W}aves with {V}orticity}, Ann. PDE \textbf{11} (2025), no.~2, Paper No. 23.
 % \MR{4935665}

\bibitem{GHS07}
Y.~Guo, C.~Hallstrom, and D.~Spirn, \emph{Dynamics near unstable, interfacial
  fluids}, Comm. Math. Phys. \textbf{270} (2007), no.~3, 635--689.% \MR{2276460}

\bibitem{HLS94}
T.~Y. Hou, J.~S. Lowengrub, and M.~J. Shelley, \emph{Removing the stiffness
  from interfacial flows with surface tension}, J. Comput. Phys. \textbf{114}
  (1994), no.~2, 312--338.% \MR{1294935}

\bibitem{HvNVW23}
T.~Hyt\"onen, J.~van Neerven, M.~Veraar, and L.~Weis, \emph{Analysis in
  {B}anach spaces. {V}ol. {III}. {H}armonic analysis and spectral theory},
  Ergebnisse der Mathematik und ihrer Grenzgebiete. 3. Folge. A Series of
  Modern Surveys in Mathematics [Results in Mathematics and Related Areas. 3rd
  Series. A Series of Modern Surveys in Mathematics], vol.~76, Springer, Cham,
  2023.% \MR{4696978}

\bibitem{IP18}
A.~D. Ionescu and F.~Pusateri, \emph{Recent advances on the global regularity
  for irrotational water waves}, Philos. Trans. Roy. Soc. A \textbf{376}
  (2018), no.~2111, 20170089, 28.% \MR{3744204}

\bibitem{JKM21}
M.~Jacobs, I.~Kim, and A.~R. M\'esz\'aros, \emph{Weak solutions to the {M}uskat
  problem with surface tension via optimal transport}, Arch. Ration. Mech.
  Anal. \textbf{239} (2021), no.~1, 389--430.% \MR{4198722}

\bibitem{KZ24}
I.~Kim and Y.~P. Zhang, \emph{Regularity of {H}ele-{S}haw flow with source and
  drift}, Ann. PDE \textbf{10} (2024), no.~2, Paper No. 20, 56.% \MR{4800678}

\bibitem{K03}
I.~C. Kim, \emph{Uniqueness and existence results on the {H}ele-{S}haw and the
  {S}tefan problems}, Arch. Ration. Mech. Anal. \textbf{168} (2003), no.~4,
  299--328.% \MR{1994745}

\bibitem{KH40}
M.~King~Hubbert, \emph{The {T}heory of {G}round-{W}ater {M}otion}, J. of
  {G}eology \textbf{48} (1940), no.~8, 785--944.

\bibitem{L24}
O.~Lazar, \emph{Global well-posedness of arbitrarily large {L}ipschitz
  solutions for the {M}uskat problem with surface tension}, 2024.

\bibitem{MM21}
A.-V. Matioc and B.-V. Matioc, \emph{The {M}uskat problem with surface tension
  and equal viscosities in subcritical {$L_p$}-{S}obolev spaces}, J. Elliptic
  Parabol. Equ. \textbf{7} (2021), no.~2, 635--670.% \MR{4342643}

\bibitem{M19}
B.-V. Matioc, \emph{The {M}uskat problem in two dimensions: equivalence of
  formulations, well-posedness, and regularity results}, Anal. PDE \textbf{12}
  (2019), no.~2, 281--332.% \MR{3861893}

\bibitem{M34}
M.~Muskat, \emph{Two fluid systems in porous media. {T}he encroachment of water
  into an oil sand}, J. Appl. Phys. \textbf{5} (1934), no.~9, 250--264.

\bibitem{N25}
J.~Na, \emph{Global self-similar solutions for the 3{D} {M}uskat equation},
  Arch. Ration. Mech. Anal. \textbf{249} (2025), no.~5, Paper No. 53, 64.
 % \MR{4947622}

\bibitem{N20}
H.~Q. Nguyen, \emph{On well-posedness of the {M}uskat problem with surface
  tension}, Adv. Math. \textbf{374} (2020), 107344, 35.% \MR{4131404}

\bibitem{N22}
\bysame, \emph{Global solutions for the {M}uskat problem in the scaling
  invariant {B}esov space {$\dot {B}^1_{\infty,1}$}}, Adv. Math. \textbf{394}
  (2022), Paper No. 108122, 28.% \MR{4348695}

\bibitem{N23}
\bysame, \emph{Coercivity of the {D}irichlet-to-{N}eumann operator and
  applications to the {M}uskat problem}, Acta Math. Vietnam. \textbf{48}
  (2023), no.~1, 51--62.% \MR{4581109}

\bibitem{N24}
\bysame, \emph{Large traveling capillary-gravity waves for {D}arcy flow},
  Nonlinearity \textbf{39} (2026), no.~3, 035008.

\bibitem{NP20}
H.~Q. Nguyen and B.~Pausader, \emph{A paradifferential approach for
  well-posedness of the {M}uskat problem}, Arch. Ration. Mech. Anal.
  \textbf{237} (2020), no.~1, 35--100.% \MR{4090462}

\bibitem{NT24}
H.~Q. Nguyen and I.~Tice, \emph{Traveling wave solutions to the one-phase
  {M}uskat problem: existence and stability}, Arch. Ration. Mech. Anal.
  \textbf{248} (2024), no.~1, Paper No. 5, 58.% \MR{4690615}

\bibitem{STT24}
R.~Schwab, S.~Tu, and O.~Turanova, \emph{Well-posedness for viscosity solutions
  of the one-phase {M}uskat problem in all dimensions}, arXiv:2404.10972.

\bibitem{S23}
J.~Shi, \emph{Regularity of solutions to the {M}uskat equation}, Arch. Ration.
  Mech. Anal. \textbf{247} (2023), no.~3, Paper No. 36, 46.% \MR{4575396}

\bibitem{S24}
\bysame, \emph{The regularity of the solutions to the {M}uskat equation: {T}he
  degenerate regularity near the turnover points}, Adv. Math. \textbf{454}
  (2024), 109850, 205.

\bibitem{NS26}
N.~Stevenson and H.~Q. Nguyen, \emph{On large periodic traveling surface waves
  in porous media}, arXiv:2601.11800.

\end{thebibliography}

\providecommand{\bysame}{\leavevmode\hbox to3em{\hrulefill}\thinspace}
\providecommand{\MR}{\relax\ifhmode\unskip\space\fi MR }
\providecommand{\MRhref}[2]{%
  \href{http://www.ams.org/mathscinet-getitem?mr=#1}{#2}
}
\providecommand{\href}[2]{#2}

\end{document}